\documentclass[sn-basic]{sn-jnl}

\usepackage{amsmath,amssymb,amsfonts,amsthm}
\usepackage{latexsym}
\usepackage{graphicx}
\usepackage{bm}
\usepackage[]{subfig}
\usepackage{fancyhdr}
\usepackage{booktabs}
\usepackage{caption}
\usepackage{physics}
\usepackage{color}
\definecolor{hillencolor}{rgb}{0.66, 0.030, 0.30}
\usepackage{lineno}

\jyear{2022}%

\theoremstyle{thmstyleone}%
\newtheorem{theorem}{Theorem}
\newtheorem{proposition}[theorem]{Proposition}%
\newtheorem{conjecture}[theorem]{Conjecture}%
\newtheorem{lemma}[theorem]{Lemma}%

\theoremstyle{thmstyletwo}%
\newtheorem{remark}{Remark}%

\theoremstyle{thmstylethree}%

\begin{document}

\title[Energy minima in nonlocal advection-diffusion models]{Detecting minimum energy states and multi-stability in nonlocal advection-diffusion models for interacting species}

%


\author*[1]{\fnm{Valeria} \sur{Giunta}}\email{v.giunta@sheffield.ac.uk}

\author[2]{\fnm{Thomas} \sur{Hillen}}\email{thillen@ualberta.ca}

\author[2,3]{\fnm{Mark A.} \sur{Lewis}}\email{marklewis@uvic.ca}

\author[1]{\fnm{Jonathan R.} \sur{Potts}}\email{j.potts@sheffield.ac.uk}

\affil*[1]{\orgdiv{School of Mathematics and Statistics}, \orgname{University of Sheffield}, \orgaddress{\street{Hicks Building, Hounsfield Road}, \city{Sheffield}, \postcode{S3 7RH}, \country{UK}}}

\affil[2]{\orgdiv{Department of Mathematical and Statistical Sciences}, \orgname{University of Alberta}, \orgaddress{\city{Edmonton}, \postcode{T6G 2G1}, \state{Alberta}, \country{Canada}}}

\affil[3]{\orgdiv{Department of Mathematics and Statistics and Department of Biology},  \orgname{University of Victoria}, \orgaddress{\postcode{PO Box 1700 Station CSC}, \city{Victoria},  \state{BC}, \country{Canada}}}

\abstract{
Deriving emergent patterns from models of biological processes is a core concern of mathematical biology.  In the context of partial differential equations (PDEs), these emergent patterns sometimes appear as local minimisers of a corresponding energy functional. 
Here we give methods for determining the qualitative structure of local minimum energy states of a broad class of multi-species nonlocal advection-diffusion models, recently proposed for modelling the spatial structure of ecosystems.  We show that when each pair of species respond to one another in a symmetric fashion (i.e. via mutual avoidance or mutual attraction, with equal strength), the system admits an energy functional that decreases in time and is bounded below.  This suggests that the system will eventually reach a local minimum energy steady state, rather than fluctuating in perpetuity.  We leverage this energy functional to develop tools, including a novel application of computational algebraic geometry, for making conjectures about the number and qualitative structure of local minimum energy solutions. These conjectures give a guide as to where to look for numerical steady state solutions, which we verify through numerical analysis. Our technique shows that even with two species, multi-stability with up to four classes of local minimum energy states can emerge. The associated dynamics include spatial sorting via aggregation and repulsion both within and between species.  The emerging spatial patterns include a mixture of territory-like segregation as well as narrow spike-type solutions. Overall, our study reveals a general picture of rich multi-stability in systems of moving and interacting species.
}





\maketitle

\section{Introduction}

A central purpose of mathematical biology is to provide a way of linking biological processes to emergent patterns \citep{Levin92, murray2001mathematical}.  In cell biology, such insights can illuminate the mechanisms behind the growth of cancerous tumours, and inform the development of interventions to slow or halt that growth \citep{altrock2015mathematics, byrne2010dissecting, painter2013mathematical}.  In ecology, the insights on mechanisms behind animal space use can be valuable for species conservation \citep{bellis2004home,macdonald2003modelling,zeale2012home}, ensuring maintenance of biodiversity \citep{hirt2021environmental, jeltsch2013integrating}, and controlling biological invasions \citep{hastings2005spatial, lewis2016mathematics, shigesada1997biological}. 

For partial differential equation (PDE) models of biological systems, one useful method to link process to pattern is to construct an energy functional for a system, if it exists.  Then the local minima of this energy functional give possible final configurations of the system.  Our focus here is to develop techniques for finding such local energy minima in a particular system of PDEs describing symmetric nonlocal multi-species interactions, with the parallel biological aim of being able to detect and describe the possible long-term patterns that may emerge from underlying processes.

The PDE system we focus on is a multi-species system of nonlocal advection diffusion equations recently introduced \citep{PL19} and slightly generalised by \cite{HPLG21}. This system models the spatial structure of ecosystems over timescales where births and deaths are negligible and has the following functional form
\begin{linenomath*}\begin{equation}\label{eq:system}
		\frac{\partial u_i}{\partial t}=D_i \Delta u_i+\nabla \cdot \left(u_i \sum_{j=1}^{N}\gamma_{i j} \nabla (K \ast u_j) \right), 
\end{equation}\end{linenomath*}
for $i\in \{1,\dots,N\}$, where $D_i$ and $\gamma_{ij}$ are constants, and $u_i(x,t)$ is the density of a species of moving organisms in location $x$ at time $t$.  Individuals detect the presence of others over a spatial neighborhood described by spatial averaging kernel $K$, which is a symmetric, non-negative function with $\|K\|_{L^1}=1$. 
The magnitude of $\gamma_{ij}$ gives the rate at which species $i$ advects towards (resp. away) from species $j$ if $\gamma_{ij}<0$ (resp. $\gamma_{ij}>0$).  Whilst the detection of individuals may be direct, e.g. through sight smell or sound, \cite{PL19} showed that the above formalism can also be used when interactions are mediated by marks in the environment or memory of past interactions. Note that, as well as modelling different species of organism, Equation (\ref{eq:system}) can also be used to model $N$ different groups within a species, or to describe more complex situations where organisms may be spatially delineated by something other than species, e.g. mixed-species territorial flocks of birds \citep{mokross2018can}.  However, we use the term `species' for simplicity.

Equation (\ref{eq:system}) generalises a variety of existing models.  In the case $N=1$ and $\gamma_{11}<0$, Equation \eqref{eq:system} is an aggregation-diffusion equation \citep{carrillo2018aggregation, carrillo2019nonlinear} and also arises in model of animal home ranges \citep{briscoe2002home}.  For $N=2$ and $\gamma_{12},\gamma_{21}>0$, Equation \eqref{eq:system} can be related to models of territory formation \citep{ellison2020mechanistic, potts2016territorial, rodriguez2020steady} and cell sorting \citep{burger2018sorting} (the latter also includes $\gamma_{12},\gamma_{21}<0$).  The case of arbitrary $N$ with $\gamma_{ij}=1$ has also been recently studied in the context of territories \citep{ellefsen2021equilibrium}.  Finally, the $N=2$ case with $\gamma_{12}$ and $\gamma_{21}$ having different signs has been studied in the context of predator-prey dynamics \citep{di2016nonlocal}.  So there is a wide range of possible applications arising from Equation \eqref{eq:system}.

Whilst our approach is quite general in potential applicability, there are various specific biological questions that might be addressed by classifying minimum energy solutions.  A simple example is that of animal territory formation. How much avoidance is necessary for segregated territories to form?  Is the emergence of territories history dependent?  Do symmetric avoidance mechanisms always lead to symmetric territories?  As another example, in the case of mutualistic species, we can ask similar questions.  How much attraction is necessary for aggregation?  Is it history dependent?  All of these questions can benefit from the insight provided by classifying minimum energy solutions to Equation \eqref{eq:system}, as well as more complex questions regarding multi-species questions that may exhibit a mixture of attraction and avoidance mechanisms.

The model given by Equation (\ref{eq:system}) has been shown to exhibit rich pattern formation properties, including aggregation, segregation, oscillatory patterns and non-periodic spatio-temporal solutions suggestive of strange attractors \citep{PL19}.  In \cite{PL19}, for the simple case where $N=2$, $\gamma_{ii}=0$, and $\gamma_{12}=\gamma_{21}$, an energy functional was constructed that is decreasing in time, bounded below, and becomes a steady state of Equation (\ref{eq:system}) as $t\rightarrow \infty$.  Furthermore, numerical experiments suggest that only stationary patterns emerge in this case \citep{PL19}.  Here, our first task is to generalise this $N=2$ energy functional to arbitrary $N$, but where $\gamma_{ij}=\gamma_{ji}$ for all $i,j,\in \{1,\dots,N\}$. 
Related work by \cite{Jungel2022} found two more energy functionals which are based on the Shannon entropy on the one hand and a Rao-like entropy on the other.  However, our focus here is on the generalization of the energy function from \cite{PL19}.

Once this energy functional has been constructed, our second task is to minimise it to ascertain the functional form of the local minimum energy solutions.  For this, we work in the local limit, i.e. where $K$ tends towards a Dirac-$\delta$ function.  We give a numerical technique for showing that, if we start with a class of stable steady state solutions for different $K$, then take the local limit, we return a piecewise constant function.  This technique makes use of the theory of Gr\"obner bases and associated methods from computational algebraic geometry.  It is a generalisation of a method first used in \citet{potts2016territorial}.  

In situations where the local limit is piecewise constant, local minima of the energy functional can be found by searching through the space of piecewise constant functions.  We show that this can sometimes be done analytically, using some basic examples in one spatial dimension to illustrate the methods.  Even in case $N=2$, this process reveals a range of situations where there are multiple local energy minima, all of which we verify via numerics away from the local limit.  Overall, the methods presented here enable users to detect local minimum energy states of Equation (\ref{eq:system}), including multiple minima, in any situation where $\gamma_{ij}=\gamma_{ji}$.

This paper is organized as follows.  We begin with linear stability analysis, in Section \ref{sec:linear_analysis}.  This sets the stage by showing that the $\gamma_{ij}=\gamma_{ji}$ case (for all $i,j$) leads to stationary pattern formation at small times (from perturbations of the homogeneous steady state) as long as the species have the same-sized populations. In Section \ref{sec:energy}, we construct an energy functional associated with Equation (\ref{eq:system}) in the case $\gamma_{ij}=\gamma_{ji}$ (for all $i,j$) and analyze its properties, particularly that it decreases in time and is bounded below. Noteably, unlike the linear analysis, this does not require the species to have the same-sized populations.  This section ends with a conjecture about the structure of the attractor, which is somewhat stronger than what we are able to show in this paper, but for which we have numerical evidence to suggest it might be true. In Section \ref{sec:struct}, we describe our technique for finding stable steady states, assuming that the local limit of stable steady states is piecewise constant, generalising a method used in \citet{potts2016memory}.  In Section \ref{sec:grobner}, we give a method for proving that this local limit is piecewise constant, demonstrating our proof for $N=2$ and arbitrary $\gamma_{ij}$, then for $N=3$ with specific examples of $\gamma_{ij}$.

\subsection{Notation and assumptions}\label{sec:assumptions}

We use the following notation conventions throughout. Let $S \subset \mathbb{R}^n$ be a measurable set. Then we denote the measure of $S$ by $\lvert S \lvert$, so that
\begin{linenomath*}\begin{equation}
	\label{eq:def_measure}
	\lvert S\lvert=\int_{S} \mathbf{1}(x) \text{ dx},
\end{equation}\end{linenomath*}
where $\mathbf{1}:\mathbb{R}^n \rightarrow \mathbb{R}$ is the constant function $\mathbf{1}(x)= 1$.

Let $\Omega\subset\mathbb{R}^n$ and $f:L^p(\Omega) \rightarrow \mathbb{R} $. We use the following norms
\begin{itemize}
	\item $ \lVert f \lVert_{L^p(\Omega)} =( \int_{\Omega} \lvert f\lvert ^p)^{1/p} $, where $ 1\leq p < \infty $,
	\item $ \lVert f\lVert_{L^{\infty}(\Omega)} = \inf \{C \geq 0 : \lvert f(x)\lvert  \leq C, \text{ a.e. in } \Omega \} $.
\end{itemize}
Let $ M \in \mathbb{N} $ and $ g=(g_1, g_2, \dots, g_M):(L^p(\Omega))^M \rightarrow \mathbb{R} $. Then we define
\begin{itemize}
	\item $ \lVert g\lVert_{(L^p(\Omega))^M} = \sum_{i=1}^{M}\lVert g_i \lVert_{L^p(\Omega)} $, where $ 1\leq p < \infty $,
	\item $ \lVert g\lVert_{(L^{\infty}(\Omega))^M} = \max_{i=1, 2, \dots, M} \{\lVert g_i\lVert_{L^{\infty}(\Omega)}\} $.
\end{itemize}
To ease notation, we usually write $  \lVert g \lVert_{L^{p}(\Omega)} $ instead of $  \lVert g\lVert_{(L^{p}(\Omega))^M} $, if the meaning is clear from the context.  We also may drop explicit dependence on $\Omega$.

We analyze Equation \eqref{eq:system} on the spatial domain $ \Omega=[0,L_1]\times[0,L_2]\times\dots\times[0,L_n] \subset \mathbb{R}^n $, for $n\geq1$, with periodic boundary conditions
\begin{linenomath*}\begin{equation}\label{eq:pbc}
		\begin{aligned}
		&	u_i (x_1,\dots,x_N,t)\lvert_{x_j=0}=u_i (x_1,\dots,x_N,t)\lvert_{x_j=L_j},\\
		&	\partial_{x_j}u_i (x_1,\dots,x_N,t)\lvert_{x_j=0}=\partial_{x_j}u_i (x_1,\dots,x_N,t)\lvert_{x_j=L_j},
		\end{aligned}\end{equation}
\end{linenomath*}
for all $ i=1,\dots, N $, $ j=1, \dots, n $ and $ t \geq0 $. A spatial domain with these periodic boundary conditions is a torus and we denote it by $ \mathbb{T} $. For the kernel $K$ we assume that $K\in L^s (\mathbb{T})$ with $s=\frac{m}{2}$ for $m\geq 2$ and $s=1$ for $m=1$.  For the non-local terms in Sections \ref{sec:energy} and \ref{sec:struct} (but not Sections \ref{sec:linear_analysis} and \ref{sec:grobner}), we assume a detailed balance  for all $i,j \in \{1,\dots,N\}$, i.e. $\gamma_{ij}=\gamma_{ji}$. Finally, in Sections \ref{sec:struct} and \ref{sec:grobner} we assume $n=1$.

\section{Linear stability analysis}\label{sec:linear_analysis}

Inhomogeneous solutions of PDEs can emerge when a change in a parameter causes the loss of stability of a homogeneous steady state, leading to the formation of inhomogeneous solutions (sometimes referred to as Turing patterns after \citet{turing1952chemical}),  which can be either stationary or periodically oscillating in time. In this Section, we will analyze the linear patterns supported by Equation \eqref{eq:system}.

In Equation \eqref{eq:system}, the total mass of each species $i$ is conserved in time, indeed on the periodic domain $ \mathbb{T} $, on which conditions \eqref{eq:pbc} hold, the following identities are satisfied
\begin{linenomath*}\begin{equation}
		\frac{d}{dt} \int_{\mathbb{T}} u_i(\mathbf{x},t)  \text{d}\mathbf{x}=0, \qquad \text{ for } i=1, \dots, N,
\end{equation}\end{linenomath*}
where $ \mathbf{x}=(x_1, x_2, \dots, x_N)  \in \mathbb{T}$. Hence, for all $ i=1, \dots, N $, 
\begin{linenomath*}\begin{equation}\label{eq:int_cond}
p_i:=\int_{\mathbb{T}} u_i (\mathbf{x},t) \text{d}\mathbf{x}= \int_{\mathbb{T}} u_i(\mathbf{x},0) d \mathbf{x}, \text{ for all }  t \geq0,
\end{equation}\end{linenomath*}
where the constant $ p_i $ is the population size of species $i$. Therefore, Equation \eqref{eq:system} has an homogeneous steady state
\begin{linenomath*}\begin{equation}\label{equilibrium}
		\mathbf{\bar{u}}=(\bar{u}_1, \bar{u}_2, \dots, \bar{u}_N), \quad \text{ where }	\bar{u}_i=\frac{p_i}{\lvert{\mathbb{T}\lvert}}, \text{ for } i=1, \dots, N ,
\end{equation}\end{linenomath*}
unique for each value of $p_i$ (determined by the initial condition). To study the stability of $ \mathbf{\bar{u}}$, we introduce the vector
\begin{linenomath*}\begin{equation}\label{eq:w}
		\mathbf{w}= (u_1 - \bar{u}_1, \dots, u_N-\bar{u}_N) = \mathbf{u}^{(0)} e^{\lambda t + i \bm{\kappa} \cdot \mathbf{x}},
\end{equation}\end{linenomath*}
where $\mathbf{u}^{(0)} $ is a constant vector, $ \lambda \in \mathbb{R} $ is the growth rate of the perturbation, $ \mathbf{x}=(x_1, \dots, x_n) \in \mathbb{T} $ and $ \bm{\kappa}=(\kappa_1, \dots, \kappa_n) $ is the wave vector, whose components are the wave numbers of the perturbation and must satisfy the boundary conditions \eqref{eq:pbc}.  We thus have 
\begin{linenomath*}\begin{equation}
		\kappa_i=\frac{2 \pi q_{i}}{L_i}, \text{ with } q_i \in \mathbb{N}, \text{ for } i=1, \dots, n.
\end{equation}\end{linenomath*}
Substituting Equation \eqref{eq:w} into Equation \eqref{eq:system} and neglecting nonlinear terms, we obtain the following eigenvalue problem
\begin{linenomath*}\begin{equation}
		\lambda(\bm{\kappa}) \mathbf{w} = \lvert\bm{\kappa}\lvert^2 \mathcal{L}(\bm{\kappa}) \mathbf{w}
\end{equation}\end{linenomath*}
where
\begin{linenomath*}\begin{equation}\label{eq:L}
		\mathcal{L}(\bm{\kappa})= \begin{bmatrix}
			-D_1 -\gamma_{11} \bar{u}_1 \hat{K}(\bm{\kappa}) & -\gamma_{12} \bar{u}_1 \hat{K}(\bm{\kappa})  & \dots  & -\gamma_{1N} \bar{u}_1 \hat{K}(\bm{\kappa}) \\
			&  &  &  \\
			-\gamma_{21} \bar{u}_2 \hat{K}(\bm{\kappa}) & -D_2 -\gamma_{22} \bar{u}_2 \hat{K}(\bm{\kappa})  & \dots & -\gamma_{2N} \bar{u}_2 \hat{K}(\bm{\kappa})  \\
			
			\vdots & & & \\
			-\gamma_{N1}\bar{u}_N \hat{K}(\bm{\kappa}) & -\gamma_{N2} \bar{u}_N \hat{K}(\bm{\kappa}) & \dots & -D_N -\gamma_{NN} \bar{u}_N \hat{K}(\bm{\kappa})
		\end{bmatrix},
\end{equation}\end{linenomath*}	
and where
\begin{linenomath*}
$$ \hat{K}(\mathbf{\bm{\kappa}})= \int_{\mathbb{R}^n} K(\mathbf{x}) e^{-i  \bm{\kappa} \cdot \mathbf{x}} \text{d}\mathbf{x} $$
\end{linenomath*}
is the Fourier transform of the kernel $ K $. 


For each $ \bm{\kappa} $, the eigenvalue with greatest real part (called the dominant eigenvalue) determines whether or not non-constant perturbations of the constant steady state at wavenumber $\bm{\kappa}$ will grow or shrink at short times.  If the dominant eigenvalue has positive real part and non-zero imaginary part, then these perturbations oscillate in time as they emerge. If the dominant eigenvalue is real, such oscillations will not occur at short times.



Now, if $ \bar{u}_i= \bar{u}_j$ and $ \gamma_{i j}=\gamma_{ji}$ for all $ i,j=1,2, \dots, N$ then $ \mathcal{L} $ is symmetric, so all its eigenvalues are real \citep{Artin}. Therefore  non-constant perturbations of the constant steady state will not oscillate at short times. In practice, situations where the dominant eigenvalue is real and positive are often accompanied by non-constant stable steady states.  Although this does not follow by necessity \citep{GLS21}, this observation nonetheless suggests that the this case provides a good starting point in searching for non-constant stationary patterns. 

In the following sections, we will study the $ \gamma_{i j}=\gamma_{ji} $ case through an energy functional analysis, showing how this can give us insights into the structure of non-constant stable steady states.  It turns out that for this analysis, we do not need the additional assumption $ \bar{u}_i= \bar{u}_j$.



We conclude this section by analysing the $ N=2 $ case in detail, to provide some results required in later sections.  In this case, the characteristic polynomial of the matrix $ \mathcal{L} $ is
\begin{linenomath*}\begin{align}\label{eq:characteristic_poly}
		P(\lambda)=& \lambda^2+((\gamma_{1 1}\bar{u}_1+\gamma_{22}\bar{u}_2)\hat{K}(\bm{\kappa})+(D_1 + D_2))\lambda +(\gamma_{1 1} \gamma_{2 2}-\gamma_{1 2} \gamma_{21})\bar{u}_1 \bar{u}_2\hat{K}(\bm{\kappa})^2
		\\&+(D_1 \gamma_{2 2}\bar{u}_2+D_2 \gamma_{1 1}\bar{u}_1)\hat{K}(\bm{\kappa})+D_1 D_2,
\end{align}\end{linenomath*}
whose roots are
\begin{linenomath*}\begin{align}\label{eq:eigenvalues}\nonumber
		&\lambda^{\pm}(\bm{\kappa})=\frac{1}{2}\left[-(\gamma_{11}\bar{u}_1+\gamma_{22}\bar{u}_2)\hat{K}(\bm{\kappa})-(D_1 +D_2) \pm\left(((\gamma_{11}\bar{u}_1-\gamma_{22}\bar{u}_2)^2 \right.\right.
		\\& \left.\left. +4 \gamma_{12}\gamma_{21}\bar{u}_1 \bar{u}_2)\hat{K}(\bm{\kappa})^2+2(D_1-D_2)(\gamma_{11}\bar{u}_1-\gamma_{22}\bar{u}_2)\hat{K}(\bm{\kappa})+(D_1-D_2)^2\right)^{1/2}\right],
\end{align}\end{linenomath*}
giving the eigenvalues of $ \mathcal{L} $. The condition $ \gamma_{12}=\gamma_{21} $ ensures that the argument of the square root is always positive and therefore the eigenvalues $ \lambda^{\pm} $ are real.	As a concrete  example, if $ p_1=p_2=1 $, $ L_1=\dots=L_N=1 $, $ D_1=D_2 $, $ \gamma_{1 2}=\gamma_{2 1} $ and $ \gamma_{1 1}=\gamma_{2 2} $ then the system admits a linear instability if there exists at least one $ \bm{\kappa} >0$ such that

\begin{linenomath*}\begin{equation}\label{eq:turingcondition}
		-\gamma_{11}\hat{K}(\bm{\kappa}) + \lvert\gamma_{12}\hat{K}(\bm{\kappa})\lvert>D_1.
\end{equation}\end{linenomath*}

\section{Energy Functional}\label{sec:energy}

In this section, we will define an energy functional associated to Equation \eqref{eq:system} with $\gamma_{12}=\gamma_{21}$, and show that it is continuous, bounded below, decreases in time, and that its stationary points coincide with those of Equation \eqref{eq:system}.  This gives evidence to suggest that Equation \eqref{eq:system} with $\gamma_{12}=\gamma_{21}$ will tend towards a steady state, 
which will be inhomogeneous in space if the constant steady state $\mathbf{\bar{u}}$ is linearly unstable.  

During this section, we will assume a positivity result, namely that $ u_i(x,0)>0 $ implies $ u_i(x,t)>0 $, for all $ i=1, \dots, N $, for all $ t>0 $. This result has been already proved in one spatial dimension \citep{HPLG21}. This proof relies on a Sobolev embedding theorem only valid in one dimension, so other tools will be needed to give a proof in arbitrary dimensions.Indeed, at the time of writing, this positivity result has not yet been established in arbitrary dimensions. 

First, we re-write Equation \eqref{eq:system} as follows
\begin{linenomath*}\begin{equation}\label{eq:system2}
		\frac{\partial u_i}{\partial t}=\nabla \cdot \left[u_i \nabla \left(D_i \text{ln}(u_i)+\sum_{j=1}^{N}\gamma_{i j}  K \ast u_j \right)\right], \, i=1,\dots, N.
\end{equation}\end{linenomath*}
Then we define the following energy functional
\begin{linenomath*}\begin{equation}\label{eq:energy}
		E[u_1, \dots, u_N]=\int_{\mathbb{T}} \sum_{i=1}^{N} u_i \left(D_i \text{ln}(u_i)+\frac{1}{2}\sum_{j=1}^{N}\gamma_{i j} K \ast u_j\right) dx,
\end{equation}\end{linenomath*}
where $ x=(x_1, x_2, \dots, x_n) $. 
{The first term $\sum D_i u_i \ln u_i$ is the entropy of each of the populations on their own and the second term $\sum \gamma_{ij} (K\ast u_j) u_i$ denotes the interaction energy between the populations (\cite{carrillo2020}).} The factor $\frac{1}{2}$ before the sum is required so that we can leverage the $\gamma_{ij}=\gamma_{ji}$ symmetry later on.

\begin{proposition}\label{pr:Econtinuous} 
	The energy functional $ E $, defined in Equation \eqref{eq:energy}, is a continuous function of the variables $ u_1, u_2, \dots, u_N $.	
\end{proposition}
\begin{proof}
	First we show that the following functions are continuous as long as $u_i$ is positive across space and time 
	\begin{linenomath*}\begin{align}
			\label{eq:continuity1}
			&u_i\longmapsto u_i \ln(u_i), \\
			\label{eq:continuity2}
			& (u_i ,u_j)\longmapsto u_iK\ast u_j.
	\end{align}\end{linenomath*}
	Equation \eqref{eq:continuity1} is continuous since it is the product of continuous functions.	For Equation \eqref{eq:continuity2}, we first observe that if $ K \in L^1 $ and $ u \in L^p $, with $1\leq p\leq \infty $, then
	\begin{linenomath*}\begin{equation}\label{eq:prop_conv}
			\lVert K\ast u \lVert_{L^p}  \leq \lVert K\lVert_{L^1} \lVert u\lVert_{L^p},
	\end{equation}\end{linenomath*}
	by Young's convolution inequality.  Moreover, since $ K \ast u - K\ast v= K\ast(u-v) $, we have
	\begin{linenomath*}\begin{equation}\label{eq:lips}
			\lVert K\ast u -K \ast v\lVert_{L^p} = \lVert K \ast (u-v) \lVert_{L^p} \leq \lVert K\lVert_{L^1} \lVert u-v \lVert_{L^p}=\lVert u-v \lVert_{L^p},
	\end{equation}\end{linenomath*}
	where the last equality uses $\lVert K \lVert_{L^1}=1$.
	Equation \eqref{eq:lips} shows that $ u\mapsto K\ast u $  is a Lipschitz function and thus a continuous function. Therefore Equation \eqref{eq:continuity2} is continuous because it is the product of continuous functions.  This shows that the integrand in Equation \eqref{eq:energy} is continuous.
	
	Now let $ 1\leq p\leq \infty $ and $ g:L^p(\Omega)\rightarrow L^p(\Omega) $ be a continuous function. Define a function $ G: L^p(\Omega)\rightarrow\mathbb{R} $ by
	\begin{linenomath*}\begin{equation}
			G(u)=\int_\Omega g(u) dx.
	\end{equation}\end{linenomath*}
	It remains to show that $G$ is continuous.	To this end, let $\epsilon>0$ and $ u \in L^p(\Omega) $.  Then since $g$ is continuous, there exists $\delta_{\epsilon}>0$ such that for any $v \in L^p(\Omega)$ with $\lVert v-u\lVert_{{L^p}}<\delta_{\epsilon}$, we have $\lVert g(v)-g(u)\lVert_{L^{p}}<\epsilon$.  Since 
	$ \lvert G(u)-G(v)\lvert \leq \lVert g(u)-g(v)\lVert_{L^{p}}$ for all $u, v \in L^{p}(\Omega)$, we have $\lvert G(v)-G(u)\lvert\leq \lVert g(v)-g(u)\lVert_{L^{p}}<\epsilon$. 
\end{proof}
\begin{remark} Note that whilst we have used $\lVert K \lVert_{L^1}=1$, the previous proposition also holds for any $K \in L^1$.\end{remark}

\begin{proposition}\label{pr:E_non_incr}
	Suppose $ \gamma_{ij}=\gamma_{ji} $, for all $ i,j=1, \dots, N $. For any positive (for each component) initial data $  (u_{1,0}, \dots, u_{N,0})$, the energy functional $ E[u_1(x,t), u_2(x,t), \dots, u_N(x,t)] $ is non-increasing over time, where $ (u_1, u_2, \dots, u_N) $ is the trajectory of Equation \eqref{eq:system} starting from $ (u_{1,0}, \dots, u_{N,0}) $. Moreover, if $E$ is constant then we are at a steady state of Equation \eqref{eq:system}.
\end{proposition}

\begin{proof}
	Examining the time-derivative of the energy functional in Equation \eqref{eq:energy} gives
	\begin{linenomath*}\begin{equation}\label{eq:dEdt}
			\begin{aligned}		
				\frac{dE}{dt}=&\int_{\mathbb{T}} \sum_{i=1}^{N} \left[\frac{\partial u_i}{\partial t} \left(D_i \text{ln}(u_i)+\frac{1}{2}\sum_{j=1}^{N}\gamma_{i j} K \ast u_j\right)+u_i\left(\frac{D_i}{u_i}\frac{\partial u_i}{\partial t}+\frac{1}{2} \sum_{j=1}^{N}\gamma_{i j} K \ast \frac{\partial u_j}{\partial t}\right) \right] dx \\
				=&\int_{\mathbb{T}} \sum_{i=1}^{N} \left[\frac{\partial u_i}{\partial t} \left(D_i \text{ln}(u_i)+\frac{1}{2}\sum_{j=1}^{N}\gamma_{i j} K \ast u_j +D_i\right)+\frac{1}{2} \sum_{j=1}^{N}\gamma_{i j} \frac{\partial u_j}{\partial t} K \ast u_i   \right] dx \\
				=&\int_{\mathbb{T}}\left[ \sum_{i=1}^{N} \frac{\partial u_i}{\partial t} \left(D_i \text{ln}(u_i)+\frac{1}{2}\sum_{j=1}^{N}\gamma_{i j} K \ast u_j +D_i\right) +\frac{1}{2} \sum_{i,j=1}^{N}\gamma_{ji} \frac{\partial u_i}{\partial t} K \ast u_j \right]   dx \\
				=&\int_{\mathbb{T}}\left[ \sum_{i=1}^{N} \frac{\partial u_i}{\partial t} \left(D_i \text{ln}(u_i)+\frac{1}{2}\sum_{j=1}^{N}\gamma_{i j} K \ast u_j +D_i\right) +\frac{1}{2} \sum_{i,j=1}^{N}\gamma_{ij} \frac{\partial u_i}{\partial t} K \ast u_j \right]   dx \\
				=&\int_{\mathbb{T}}\left[ \sum_{i=1}^{N} \frac{\partial u_i}{\partial t} \left(D_i \text{ln}(u_i)+\sum_{j=1}^{N}\gamma_{i j} K \ast u_j +D_i\right) \right]   dx \\
				=&\int_{\mathbb{T}} \sum_{i=1}^{N} \nabla \cdot \left[u_i \nabla \left(D_i \text{ln}(u_i)+\sum_{j=1}^{N}\gamma_{i j}  K \ast u_j\right)\right] \left[ D_i\text{ln}(u_i)  +\sum_{j=1}^{N}\gamma_{ij} K \ast u_j +D_i \right] dx.
			\end{aligned}
	\end{equation}\end{linenomath*}
	Here, the second equality uses that $ \int_{\mathbb{T}} g (K\ast h) dx=\int_{\mathbb{T}} h (K\ast g)dx$ as long as $ K(x)=K(-x) $ for $ x\in \mathbb{R}^{n} $.  The fourth equality uses $ \gamma_{ij}=\gamma_{ji} $ and the sixth uses Equation \eqref{eq:system2}.
	
	Before continuing the computations in Equation \eqref{eq:dEdt}, we simplify notation by setting
	\begin{linenomath*}\begin{equation}
			f_i= D_i\text{ln}(u_i)  +\sum_{j=1}^{N}\gamma_{ij} K \ast u_j +D_i.
	\end{equation}\end{linenomath*}
	Observing that  
	\begin{linenomath*}\begin{equation}\label{eq:div_grad}
			\nabla \cdot (u_i \nabla f_i)= \sum_{h=1}^{n}	\partial_{x_h} (u_i \partial_{x_h}f_i),
	\end{equation}\end{linenomath*}
	we continue the previous computation to give
	\begin{linenomath*}\begin{equation}\label{eq:dEdt_2}
			\begin{aligned}		
				\frac{dE}{dt}=&\int_{\mathbb{T}} \sum_{i=1}^{N} \sum_{h=1}^{n} \partial_{x_h} \left(u_i \partial_{x_h} f_i \right)f_i dx\\
				=&\int_{\mathbb{T}} \sum_{i=1}^{N} \sum_{h=1}^{n} \left( \partial_{x_h} (u_i f_i \partial_{x_h} f_i) - u_i (\partial_{x_h} f_i)^2 \right) dx
				\\=& -\int_{\mathbb{T}} \sum_{i=1}^{N} \sum_{h=1}^{n} u_i  (\partial_{x_h} f_i)^2 dx\\
				=&-\int_{\mathbb{T}} \sum_{i=1}^{N} u_i \lvert \nabla f_i \lvert^2 dx \\
				=& -\int_{\mathbb{T}} \sum_{i=1}^{N} u_i \left\lvert\nabla \left(D_i \text{ln}(u_i)+\sum_{j=1}^{N}\gamma_{i j}  K \ast u_j\right) \right\lvert^2 dx \leq 0.
			\end{aligned}
	\end{equation}\end{linenomath*}
	The final inequality uses the assumption that $ u_i > 0 $. The second equality uses integration by parts. The third equality follows from the following equalities
	\begin{linenomath*}\begin{flalign}\nonumber
		\int_{\mathbb{T}}& \sum_{i=1}^{N} \sum_{h=1}^{n}  \partial_{x_h} (u_i f_i \partial_{x_h} f_i)\\\nonumber
		=& \int_{0}^{L_1} \int_{0}^{L_2 }\cdots \int_{0}^{L_n}  \sum_{i=1}^{N} \sum_{h=1}^{n}	\partial_{x_h} (u_i f_i \partial_{x_h}f_i)dx_1 dx_2\dots dx_n\\ \nonumber
		=& \int_{0}^{L_2 }dx_2\cdots \int_{0}^{L_n}dx_n  \left[\sum_{i=1}^{N} (u_i f_i \partial_{x_1}f_i)\right]_{x_1=0}^{x_1=L_1}\\ \nonumber
		&+\int_{0}^{L_1 }dx_1\cdots \int_{0}^{L_n}dx_n  \left[\sum_{i=1}^{N} (u_i f_i \partial_{x_2}f_i)\right]_{x_2=0}^{x_2=L_2}\\ \label{eq:pbc_int}
		&+\cdots +\int_{0}^{L_1 }dx_1\cdots \int_{0}^{L_{n-1}}dx_{n-1}  \left[\sum_{i=1}^{N} (u_i f_i \partial_{x_n}f_i)\right]_{x_n=0}^{x_n=L_n},
	\end{flalign}
	\end{linenomath*}
	and we observe that each term in Equation \eqref{eq:pbc_int} is equal to zero due to the periodic boundary conditions in Equation \eqref{eq:pbc}.
	
	Equation \eqref{eq:dEdt_2} shows that $ E $ is decreasing over time unless 
	\begin{linenomath*}\begin{equation}\label{eq:steady_state}
			\nabla \left(D_i \text{ln}(u_i)+\sum_{j=1}^{N}\gamma_{i j}  K \ast u_j\right) = 0, \text{ for all } i=1,\dots, N,
	\end{equation}\end{linenomath*}
	which is a steady state of Equation \eqref{eq:system2}, or equivalently of Equation \eqref{eq:system}. 
\end{proof}

\begin{remark}
	Proposition \ref{pr:E_non_incr} rules out the existence of non-stationary, time-periodic solutions. Indeed, as $ E $ is monotonic decreasing, if there exist $ t,\tau>0 $ such that $ E[u_1(x,t), \dots, u_N(x,t)] = E[u_1(x,t+\tau), \dots, u_N(x,t+\tau)] $, then $\dot{E}(t)=0$, so Equation \eqref{eq:steady_state} holds and $ (u_1(x,t), \dots, u_N(x,t)) $ is a stationary solution. 
\end{remark}

\begin{proposition}\label{pr:E_bounded_below}
	Let $  \lVert K\lVert_{L^{\infty}} <\infty $ 
	and let $  (u_{1,0},u_{2,0}, \dots, u_{N,0}) \in L^1(\mathbb{T})^N $ be positive initial data and $ (u_1, u_2, \dots, u_N) $ be the trajectory of Equation \eqref{eq:system} starting from $(u_{1,0},u_{2,0}, \dots, u_{N,0}) $. Then $ E[u_1, u_2, \dots, u_N] $ is bounded below by a constant.
\end{proposition}
\begin{proof}
We first observe that for all $ \gamma \in \mathbb{R} $, the following inequalities hold
\begin{linenomath*}\begin{align}\label{eq:prelim}\nonumber
			\int_{\mathbb{T}} \gamma u_i K\ast u_j dx 
			&\geq- \abs{\gamma} \int_{\mathbb{T}} \abs{u_i K\ast u_j} dx \\ \nonumber 
			&\geq  - \abs{\gamma} \lVert u_i\lVert_{1} \lVert K\ast u_j\lVert_{\infty} \\ 
			&\geq - \abs{\gamma} \lVert u_i\lVert_{L^{1}} \lVert K\lVert_{L^{\infty}} \lVert u_j\lVert_{L^{1}}.
	\end{align}\end{linenomath*}
The first inequality uses the fact that $ \gamma \geq -\abs{\gamma} $, for all $ \gamma \in \mathbb{R} $, the second uses H\"older's inequality and the third uses Young's convolution inequality. Moreover, since $ u_i>0 $, condition \eqref{eq:int_cond} ensures that $\lVert u_i(x,t)\lVert_{L^1}=p_i$ for all $t \geq 0$ and thus the right-hand side of Equation \eqref{eq:prelim} is finite.
	
Finally, by observing that $ \inf_{u_i > 0}\{u_i \text{ln}(u_i)\}=-e^{-1} $ and also by using Inequality \eqref{eq:prelim}, we obtain the following estimates 
\begin{linenomath*}\begin{equation}\label{eq:Ebounded_below}
	\begin{aligned}
		E[u_1, u_2, \cdots, u_N]&=\int_{\mathbb{T}} \sum_{i=1}^{N} u_i D_i \text{ln}(u_i) dx+\frac{1}{2}\int_{\mathbb{T}}\sum_{i,j=1}^{N}\gamma_{i j}u_i K \ast u_j dx \\
		& \geq - e^{-1} \lvert \mathbb{T} \rvert \sum_{i=1}^{N} D_i  -\frac{1}{2}  \lVert K\lVert_{L^{\infty}} \sum_{i,j=1}^{N}\lvert \gamma_{ij}\lvert \lVert u_i\lVert_{L^1} \lVert u_j\lVert_{L^1} \\
		& = - e^{-1} \lvert \mathbb{T} \rvert \sum_{i=1}^{N} D_i -\frac{1}{2} \lVert K\lVert_{L^{\infty}} \sum_{i,j=1}^{N}\lvert\gamma_{ij}\lvert p_i p_j,
	\end{aligned}
\end{equation}\end{linenomath*}
where the last equality uses the integral condition \eqref{eq:int_cond}. Thus $ E $ is bounded below.
\end{proof}

\begin{proposition}\label{pr:E_convergence}	
	Suppose $ \lvert \lvert K\lvert \lvert _{L^{\infty}} < \infty $ and $ \gamma_{ij}=\gamma_{ji} $, for all $ i,j=1, \dots, N $. For any positive initial data $  (u_{1,0}, \dots, u_{N,0}) \in L^1(\mathbb{T})^N $, there exists a constant $ l_{u_0} $, depending on $ u_0$, such that
	\begin{linenomath*}\begin{equation}
			\lim_{t\rightarrow \infty} E[u_{1}(x,t), \dots, u_{N}(x,t)] = l_{u_0},
	\end{equation}\end{linenomath*}
	where $ (u_{1}(x,t), \dots, u_{N}(x,t)) $ is the trajectory of Equation \eqref{eq:system} starting from $ (u_{1,0}, \dots, u_{N,0}) $.
\end{proposition}

\begin{proof}
	Since $  \lVert K\lVert_{L^{\infty}} <\infty\ $, Prop. \ref{pr:E_bounded_below} ensures that the following set
	\begin{linenomath*}\begin{equation}\label{eq:setE}
			\{E[u_{1}(x,t), \dots, u_{N}(x,t)] : t \in \mathbb{R}^+\}
	\end{equation}\end{linenomath*}
	is bounded below.
	Due to the Completeness Axiom of the real numbers, the set in \eqref{eq:setE} has an infimum $ l_{u_0} $, which is determined by the initial condition $ u_0 $.
	Moreover, by Proposition \ref{pr:E_non_incr}, $ E $ is a non-increasing monotonic function of time, so tends to its infimum $ l_{u_0} $ as $ t\rightarrow \infty $.
\end{proof}

Proposition \ref{pr:E_convergence} shows that for any initial data $  \mathbf{u}_0 \in L^1(\mathbb{T})^N $ the trajectory starting from $  \mathbf{u}_0 $ evolves over time towards a configuration that is a local minimiser of $ E $, with energy $E=l_{u_0}$.  We also observe that if $ E $ reaches the minimum value $ l_{u_0} $ at a finite time $ T $, then the trajectory becomes stationary. Indeed, if $ E(\mathbf{u}(T))=l_{u_0}  $ then $ E(\mathbf{u}(t))\equiv l_{u_0}  $ for all $ t \geq T $. Hence, the minimum at  $E= l_{u_0} $ corresponds to a steady state that is Lyapunov stable (i.e. any solution that starts arbitrarily close to the steady state will remain arbitrarily close).  However, it does not guarantee asymptotic stability (i.e. any solution that starts arbitrarily close to the steady state tend toward the steady state).
In the next Section, we will propose a method to determine the structure of these minimum energy states of Equation \eqref{eq:system}.

Finally, we note that the convergence of $ E $ towards a finite minimum value does not guarantee that every solution converges towards a steady state when $ \gamma_{ij}=\gamma_{ji} $, as opposed to fluctuating in perpetuity. Nevertheless, this is something we would like to establish. 
Indeed, in all our numerical investigations, both here (in Section \ref{sec:struct}) and in previous works \citep{PL19,HPLG21}, we have only every observed (numerically) stable steady state solutions emerging, and have never observed perpetually fluctuating solutions.
Therefore, we conclude this section formulating the following conjecture. This is left as an open problem, but one possible means of attack might be the via the $S^1$-equivariant theory of \citet{buttenschon2021non}, applied there to a single-species system with a similar (but not identical) non-local advection term.

\begin{conjecture}
	Let $ \lVert K\lVert_{L^{\infty}} < \infty $ and $ \gamma_{ij}=\gamma_{ji} $, for all $ i,j=1, \dots, N $. For any positive (for each component) initial datum $  u_0=(u_{1,0}, \dots, u_{N,0}) \in L^1(\mathbb{T})^N $, the corresponding solution to Equation \eqref{eq:system} converges towards a steady state.
\end{conjecture}

\section{A method to find minimum energy states}
\label{sec:struct}

In this section, we will propose a method to gain insight into the possible structures of minimum energy to Equation \eqref{eq:system}. We  build on methods first proposed in \cite[Section 3.4]{potts2016memory} and recent existence results of \cite{Jungel2022}. We work in one spatial dimension and assume the assumptions of Section \ref{sec:assumptions}.  

As shown in the previous section, the energy will always tends towards a local minimum, leading to a minimum energy state for the system, which is also a steady state. 
When solving Equation \eqref{eq:system} for the top-hat kernel 
\begin{linenomath*}\begin{equation}\label{eq:top-hat}
		K_{\alpha}(x)= \begin{cases}
			\frac{1}{2 \alpha}, \quad  x\in [-\alpha, \alpha],\\
			0, \quad \text{ otherwise},
		\end{cases}
\end{equation}\end{linenomath*}	
numerically, we find that for decreasing $\alpha$,
 the asymptotic steady state solutions look increasingly like piece-wise constant functions, or the limit of arbitrarily narrow, arbitrarily high piece-wise constant functions, with single or multiple peaks. These structures become more singular as $\alpha\to 0$.   In Figure \ref{fig:SolutionsDifferentAlpha}, we see this for some simple examples. Note that as $ \alpha\to 0$, the top-hat kernel in Equation (\ref{eq:top-hat}) becomes a Dirac delta measure, and the model (\ref{eq:system}) becomes a local cross-diffusion model. Hence we call this limit $\alpha\to 0$ as the {\it local limit}.  

\cite{Jungel2022} derived a solution theory for non-smooth interaction kernels $K$, which includes the case of a top-hat kernel as in Equation (\ref{eq:top-hat}).  They consider Equation \eqref{eq:system} for the case where there are constants $\pi_i$ such that the matrix $(\pi_i\gamma_{ij})_{ij}$ is positive definite. For that case they showed global existence of weak solutions in Sobolev spaces. They also show a local-limit result. As $\alpha\to 0$ there exists a subsequence of solutions of Equation (\ref{eq:system}), with $K$ as in Equation (\ref{eq:top-hat}), that converge to a solution of the local version of Equation (\ref{eq:system}). The norm of this convergence varies depending on the space dimension. In $n=1$ we can use any $L^p$-norm and in dimensions $n\geq 2$ we use the $L^{\frac{n}{n-1}}$-norm. 
 These limits are piece-wise constant solutions, and spike solutions, depending on the sign of $\gamma_{ij}$. They arise as minimizers of the local version of the energy functional (Equation \eqref{eq:energy}), which is 
 \begin{linenomath*}\begin{equation}\label{eq:energy-local}
		E_{\tiny{\mbox{local}}}[u_1, \dots, u_N]=\int_{\mathbb{T}} \sum_{i=1}^{N} u_i \left(D_i \text{ln}(u_i)+\frac{1}{2}\sum_{j=1}^{N}\gamma_{i j} u_j\right) dx,
\end{equation}\end{linenomath*}
where $ x=(x_1, x_2, \dots, x_n) $.
Hence in the following we consider piece-wise constant energy minimizers, assuming that they are close to the minimizers of the non-local problem and we confirm this relation numerically.  We also focus on the $n=1$ case and write $L=L_1$ for simplicity.

\begin{figure}[H]
	\centering
	\subfloat[]
	{\includegraphics[width=.71\textwidth]{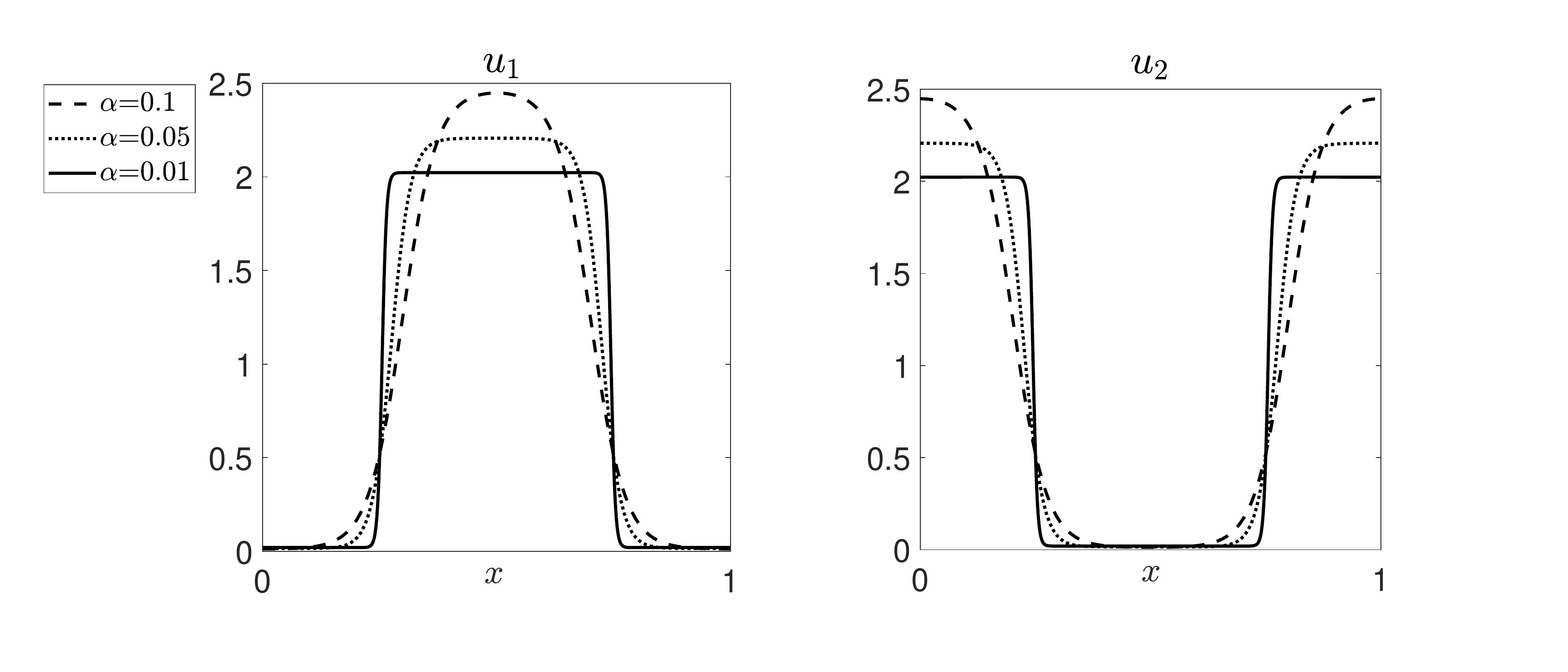}}
	\hspace{0.0cm}
	\subfloat[]
	{\includegraphics[width=.28\textwidth]{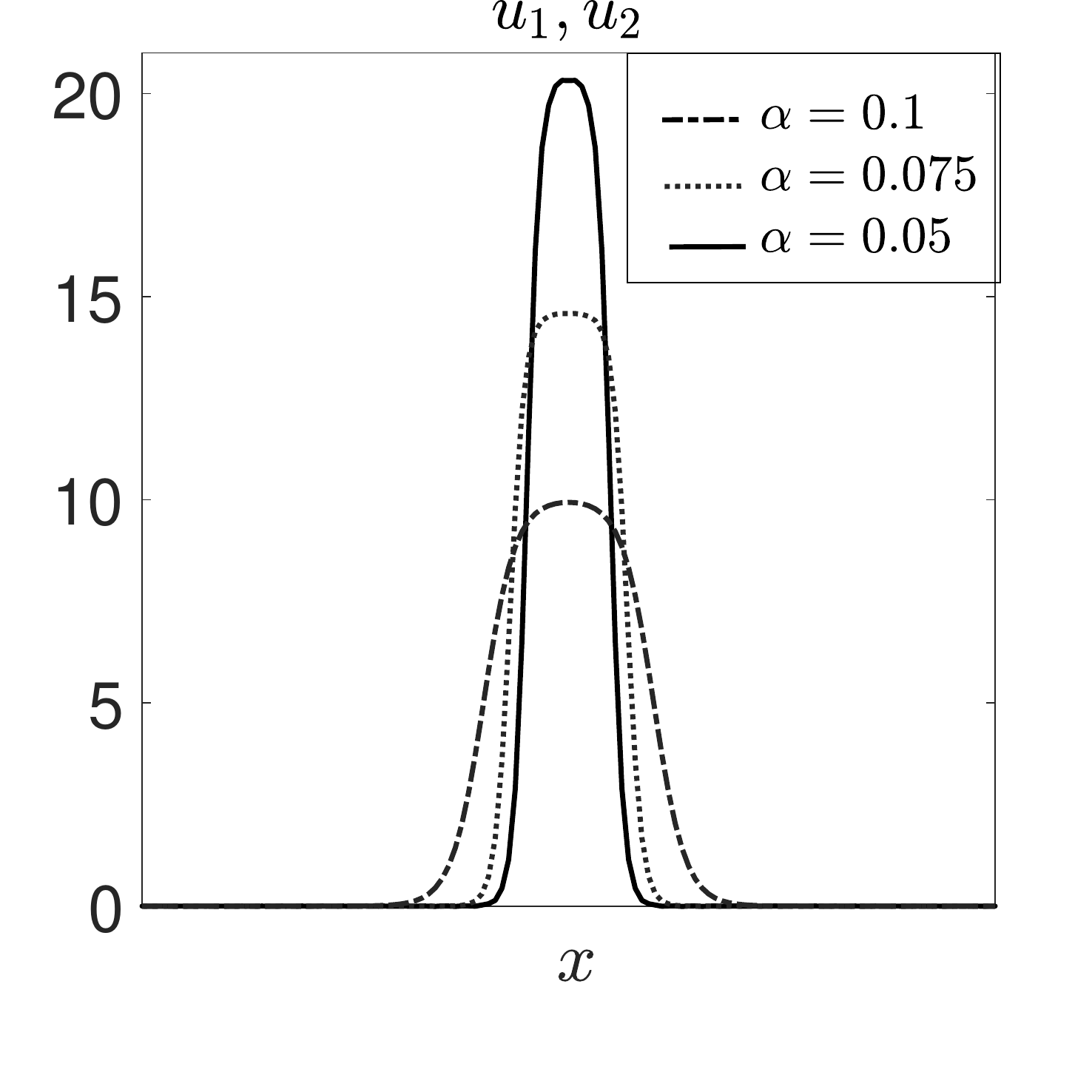}}
	\hspace{0.0cm}
	\subfloat[]
	{\includegraphics[width=.71\textwidth]{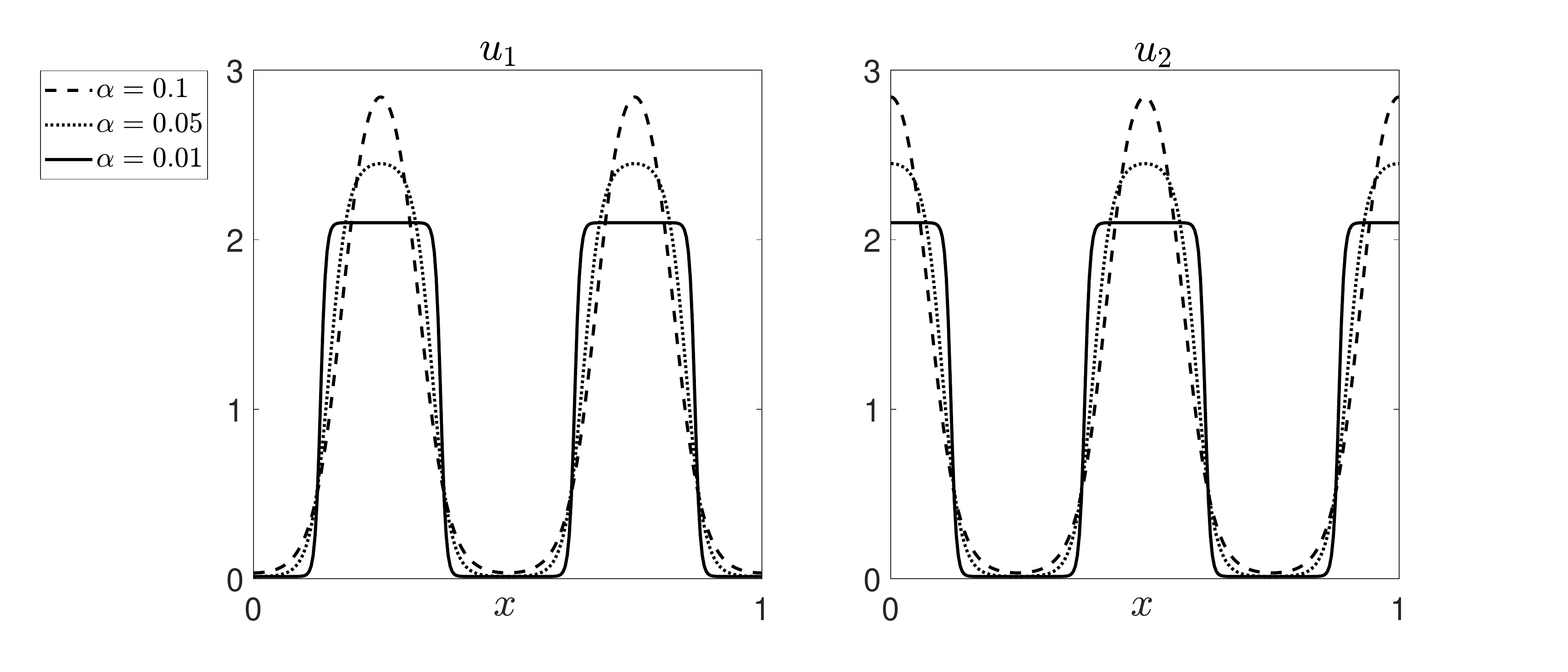}}
	\hspace{0.0cm}
	\subfloat[]
	{\includegraphics[width=.282\textwidth]{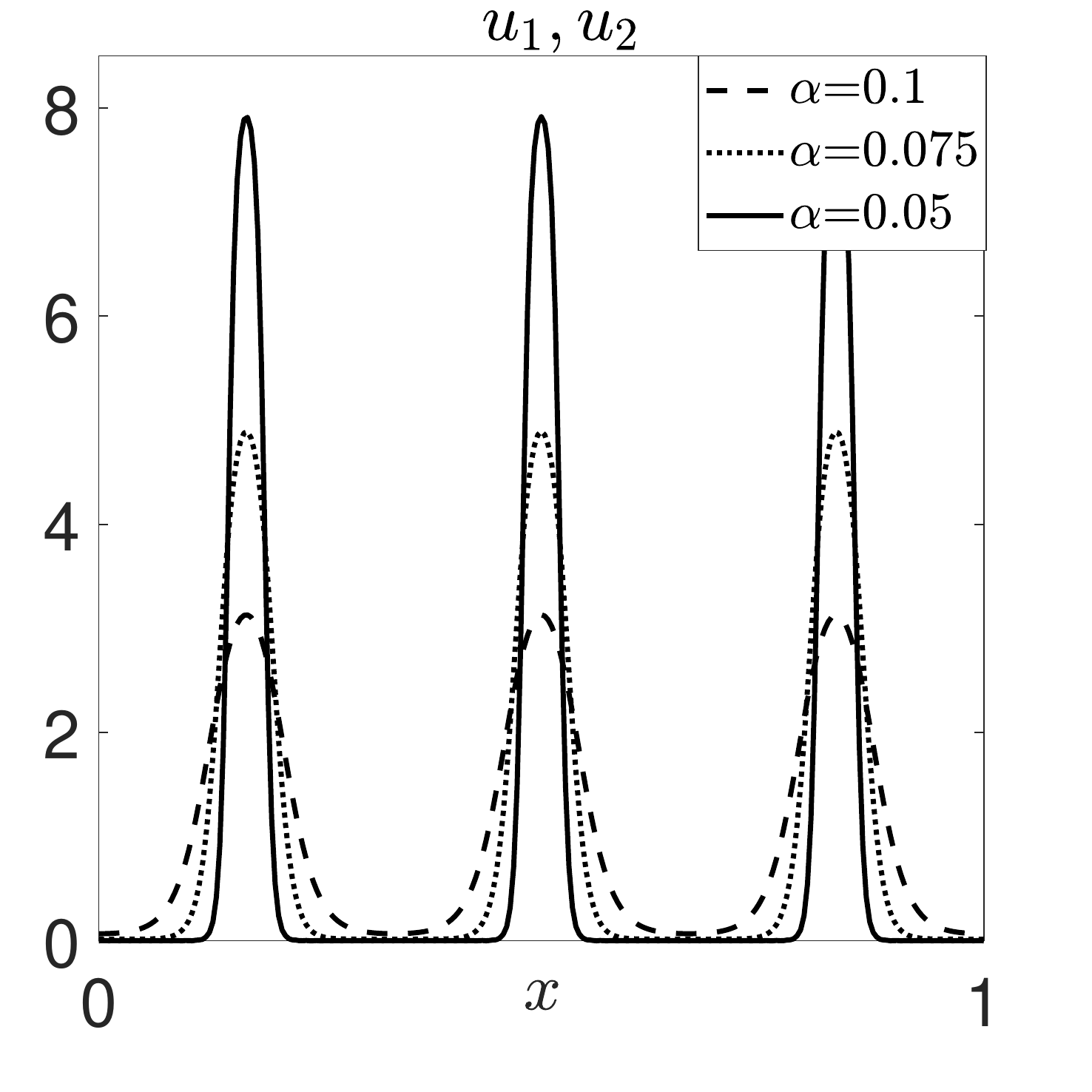}}
	\caption{Numerical steady solutions to Equation \eqref{eq:system}, with $ N=2 $, $ K= K_{\alpha}(x) $ (Equation \eqref{eq:top-hat}), for different values of $ \alpha $. As $ \alpha $ tends to zero, the solution appears to tend towards a piece-wise constant function (Panel (a) and (c)) or the limit of arbitrarily narrow, arbitrarily high piece-wise constant functions (Panel (b) and (d)).  The parameter values used in the simulations are $ D_1=D_2=1 $, $ p_1=p_2=1 $, $ \gamma_{11}=\gamma_{22}=0 $, $ \gamma_{12}=1.05 $ in Panel (a) and (c), $ \gamma_{12}=-1.05 $ in Panel (b) and (d)}
	\label{fig:SolutionsDifferentAlpha}
\end{figure}

We now explain our method in detail. First, Equation \eqref{eq:dEdt_2} in one dimension tells us that any minimum energy solution, $u_i (x)$, occurs when
\begin{linenomath*}\begin{align}
		\label{ssss}
		0= 
		u_i 
		\left[\frac{\partial}{\partial x} \left(D_i \text{ln}(u_i )+\sum_{j=1}^{N}\gamma_{i j}  K \ast u_j \right) \right]^2,
\end{align}\end{linenomath*}
for each $i \in \{1,\dots,N\}$.  Next we take the local limit of Equation (\ref{ssss}), which in the case $K=K_\alpha$ is the limit $ \alpha \rightarrow 0$.  In this limit, Equation (\ref{ssss}) becomes
\begin{linenomath*}\begin{align}
		\label{ssssl}
		0 = u_i \left[\frac{\partial}{\partial x} \left(D_i \text{ln}(u_i )+\sum_{j=1}^{N}\gamma_{i j} u_j \right) \right]^2.
\end{align}\end{linenomath*}
Therefore, either $u_i (x)=0$, or, for any subinterval on which $u_i (x) \neq 0$, there exists a constant $c_i \in \mathbb{R}$ such that
\begin{linenomath*}\begin{align}
		\label{ci2}
		c_i=D_i \text{ln}(u_i )+\sum_{j=1}^{N}\gamma_{i j}  u_j, \text{ for } i =1,\dots,N.
\end{align}\end{linenomath*}
In principle, there might exist infinitely many subintervals on which $u_i(x) \neq 0$, and $c_i$ may vary between these different subintervals. However, for each set of constants $c_1,\dots,c_N$, Equation \eqref{ci2} will typically have a finite number of common solutions (indeed, Section \ref{sec:grobner} shows how to determine whether we are in this `typical' situation). 

Therefore, on each subinterval $I$ in which $u_i(x) \neq 0$, there exists a finite set of values $u_{i1}^c, \dots, u_{ih}^c$, with $h \in \mathbb{N}$, satisfying Equation \eqref{ci2}, such that
\begin{linenomath*}\begin{align}
			\label{eq:uih}
			u_i (x)=\begin{cases}
				u_{i1}^c, &\mbox{for $x \in I_{i1}$}, \\
				\vdots & \\
				u_{ih}^c, &\mbox{for $x \in I_{ih}$}, \\
			\end{cases}
	\end{align}\end{linenomath*}
where $I_{il}$, for $i=1,\dots,N$ and $l=1, \dots, h$, are disjoint subsets of $I$ such that $\cup_{l}I_{il}=I$ for each $i$.  By considering all such subintervals $I$ together, Equation (\ref{eq:uih}) defines a class of piece-wise constant functions on $[0,L]$.  The aim here is to examine which of these functions is a local minimum of the energy and satisfies all model assumptions.



The general case is too complicated to deal with in one go, so we demonstrate our method on some simple examples for the case of two species, $N=2$.
We start by studying the case  $\gamma_{1 1} =\gamma_{22}=0 $, so there is neither self-attraction nor self-repulsion.  We split this analysis further into the cases of mutual avoidance ($\gamma_{12}>0$) and mutual attraction ($\gamma_{12}<0$).  Then we analyze the case where $\gamma_{1 1}, \gamma_{22} \neq 0 $.

\subsection{The case $\mathbf{\gamma_{1 1} =\gamma_{22}=0 }$ with mutual avoidance, $\mathbf{\gamma_{12}=\gamma_{21}>0}$ }\label{sub:1}

\subsubsection{Analytic results in the local limit}
\label{sec:analytic1}

{\noindent}Minimising the energy over the full class of functions given by Equation (\ref{eq:uih}) turns out to be too complicated.  However, our numerics (see Figure \ref{fig:SolutionsDifferentAlpha}) suggest that the local limit  (i.e. $\alpha \rightarrow 0$ in the case $K=K_\alpha$) of any solution to Equation \eqref{eq:system} is a function of the following form
\begin{linenomath*}\begin{align}
		\label{eq:ui}
		u_i (x)=\begin{cases}
			u_{i}^c, &\mbox{for $x \in S_{i}$}, \\
			0, & \mbox{for $x \in [0,L]$\textbackslash$S_{i}$},
		\end{cases}
\end{align}\end{linenomath*}
where $ u_{i}^c \in \mathbb{R}^{+} $ and $S_{i}$ are subsets of $[0,L]$, for $i\in\{1,2\}$. Therefore we restrict our search by looking for the minimisers of the energy (Equation \eqref{eq:energy}) in the class of piece-wise constant functions defined as in Equation \eqref{eq:ui}.
	
	By Equation (\ref{eq:int_cond}), in Equation \eqref{eq:ui} we require the following constraint 
	\begin{linenomath*}\begin{equation}\label{uic}
			u_{i}^c\lvert S_{i}\rvert  =p_i, \text{ for }i=1,2,
	\end{equation}\end{linenomath*}
	recalling from Equation \eqref{eq:def_measure} that $\lvert S\rvert $ denotes the measure of a set $S$, not the cardinality, and $p_i$ denotes the total population size of species $i$. We wish to find the solutions of the form in Equation (\ref{eq:ui}), subject to Equation (\ref{uic}), that are local minimisers of the energy, Equation \eqref{eq:energy}. Placing Equation \eqref{eq:ui} into Equation \eqref{eq:energy}, and taking the spatially-local limit (i.e. $\alpha \rightarrow 0$ in the case $K=K_\alpha$), gives 
	\begin{linenomath*}\begin{align}\nonumber
			E[u_1 ,u_2 ]=&\int_0^L \left( D_1 u_1 \ln(u_1) + D_2 u_2 \ln(u_2)+\gamma_{12} u_{1}  u_{2}  \right)dx \nonumber \\\nonumber
			=& \lvert S_{1}\lvert  D_1  u_{1}^c \text{ln}(u_{1}^c)+\lvert S_{2}\lvert  D_2  u_{2}^c \text{ln}(u_{2}^c)  + \gamma_{12} u_{1}^c u_{2}^c \lvert S_{1} \cap S_{2}\lvert \\ \label{eq:energyloc0}
			=& p_{1} D_1 \text{ln}(u_{1}^c)+p_{2} D_2 \text{ln}(u_{2}^c)  + \gamma_{12} u_{1}^c u_{2}^c \lvert S_{1} \cap S_{2}\lvert ,
	\end{align}\end{linenomath*}	
	where the first equality uses $\gamma_{12}=\gamma_{21}$, the second equality uses Equation \eqref{eq:ui} and the third equality uses Equation \eqref{uic}.
	
In Equation \eqref{eq:energyloc0}, notice that if we keep $ \lvert S_{1}\lvert  $ and $ \lvert S_{2}\lvert $ fixed whilst lowering $ \lvert S_{1} \cap S_{2}\lvert  $ then the energy decreases. Thus, if $ \lvert S_{1}\lvert+\lvert S_{2}\lvert  \leq L $, we can construct disjoint sets $S_{1}$ and $S_{2}$, and these will correspond to lower energy solutions than any pair of non-disjoint sets of equal measure. Furthermore, if $ \lvert S_{1}\rvert+\lvert S_{2}\rvert  > L $, we can construct sets $S_{1}$ and $S_{2}$, such that $\lvert S_1\cap S_2\rvert=\lvert S_1\rvert+\lvert S_2\rvert-L$ and these will correspond to lower energy solutions than any other pair of sets of equal measure. 
Therefore henceforth, when $ \lvert S_{1}\rvert+\lvert S_{2}\rvert  \leq L $, we will assume that $S_{1} \cap S_{2}=\emptyset$, and when $ \lvert S_{1}\rvert+\lvert S_{2}\rvert  > L $, we will assume that $\lvert S_1\cap S_2\rvert=\lvert S_1\rvert+\lvert S_2\rvert-L$. 
	
	To search for the local minimizers of the energy in Equation \eqref{eq:energyloc0}, we thus define
	\begin{linenomath*}\begin{equation}\label{eq:mathcalE}
			\mathcal{E}({u}_{1}^c,{u}_{2}^c)=
			\begin{cases}
				\sum_{i=1}^{2}  p_{i} D_i \text{ln}(u_{i}^c), & \text{ if }  \lvert S_{1}\rvert + \lvert S_{2}\rvert  \leq L , 
				\\
				\\
				\sum_{i=1}^{2}  p_{i} D_i \text{ln}(u_{i}^c) + \gamma_{12} u_{1}^c u_{2}^c (\lvert S_1\rvert+\lvert S_2\rvert-L) , &\text{ if } 	\lvert S_{1}\rvert + \lvert S_{2}\rvert > L.
			\end{cases}
	\end{equation}\end{linenomath*}
	To constrain our search, notice that Equation \eqref{uic} and $ \lvert S_i\lvert  \leq L $ imply that
	\begin{linenomath*}\begin{align}
			\label{ssssic}
			u_i^c=\frac{p_i}{\lvert S_i\lvert} \geq \frac{p_i}{\lvert L\lvert }, \text{ for } i=1,2.
	\end{align}\end{linenomath*}
	The region of the $(u_1^c,u_2^c)$-plane defined by Equation \eqref{ssssic} is shown as white region in Figure \ref{fig:Diagram}.
	Our strategy will be as follows. First we will look for the local minima of Equation \eqref{eq:mathcalE}, subject to Equation \eqref{ssssic}, in the case where $ \lvert S_1\rvert  + \lvert S_2\rvert  \leq L $. Then we will look in the region $ \lvert S_1\rvert  + \lvert S_2\rvert  > L $.  Combining these results will then give us a complete picture of the local minima of $\mathcal{E}(u_1^c,u_2^c)$. 
	
Starting with $ \lvert S_1\rvert  + \lvert S_2\rvert  \leq L$,  Equation \eqref{uic} shows that this case is equivalent to the following condition
\begin{linenomath*}\begin{equation}\label{eq:constr}
	\frac{p_1}{u_1^c}+\frac{p_2}{u_2^c}=\lvert S_1\rvert + \lvert S_2\rvert\leq L.
\end{equation}\end{linenomath*}
By analysing the partial derivatives of $\mathcal{E}(u_1^c,u_2^c)$ in the region of the $(u_1^c,u_2^c)$-plane defined by Equation \eqref{eq:constr}, we see that there are no critical points in this region. Furthermore, $\mathcal{E}(u_1^c,u_2^c)\rightarrow \infty$ as either $u_1^c \rightarrow \infty$ or $u_2^c\rightarrow \infty$.  Therefore minima in this region must lie on the boundary, $ {p_1}/{u_1^c}+{p_2}/{u_2^c}= L $, which is shown as solid black line in Figure \ref{fig:Diagram}. Analysis of the partial derivative of $\mathcal{E}(u_1^c,u_2^c)$ on this boundary shows that $ \mathcal{E}(u_1^c,u_2^c)$ has a unique minimum point, given by
\begin{linenomath*}\begin{equation}\label{eq:minimum}
	\mathcal{M}_S=(u_{1S}^c, u_{2S}^c):=\left(\frac{p_1 D_1+p_2 D_2}{D_1 L},\frac{p_1 D_1 +p_2D_2}{D_2 L}\right).
\end{equation}\end{linenomath*}
This is also a local minimum of the region defined by Equation \eqref{eq:constr}. This can be shown by performing a Taylor expansion of $ \mathcal{E}(u_1^c,u_2^c) $ about the point $\mathcal{M}_S$ in the region given by $p_1/u_1^c+p_2/u_2^c \leq L$. Since the slope of the tangent line to the curve $p_1/u_1^c+p_2/u_2^c  = L$ at the point  $ \mathcal{M}_S $  is $ -\frac{D_1^ 2 p_1}{D_2^2 p_2} $, we choose two arbitrarily small constants, $\epsilon$ and $\delta$, such that $D_1^ 2 p_1 \epsilon+D_2^2 p_2\delta\geq 0$ and then perform a Taylor expansion of $\mathcal{E}(u_1^c,u_2^c)$ in a neighbourhood of $ \mathcal{M}_S $, which shows that
    \begin{linenomath*}\begin{align} \nonumber
			\mathcal{E}(u_{1S}^c+\epsilon,u_{2S}^c+\delta)& \approx    \mathcal{E}(u_{1S}^c,u_{2S}^c)+ \partial_{u_1^c} \mathcal{E}(u_{1S}^c,u_{2S}^c) \epsilon + \partial_{u_2^c} \mathcal{E}(u_{1S}^c,u_{2S}^c)\delta\\ \nonumber
			&= \mathcal{E}(u_{1S}^c,u_{2S}^c)+ \frac{p_1 D_1}{u_{1S}^c}\epsilon + \frac{p_2 D_2}{u_{2S}^c}\delta \\ \nonumber
			&= \mathcal{E}(u_{1S}^c,u_{2S}^c) + \frac{L}{p_1 D_1+p_2 D_2}(D_1^ 2 p_1 \epsilon+D_2^2 p_2\delta)
			\\
			& \geq \mathcal{E}(u_{1S}^c,u_{2S}^c).
	\end{align}\end{linenomath*}
Since $ \mathcal{M}_S $ lies on the boundary curve $ \lvert S_1\lvert  + \lvert S_2\lvert = L $ (Figure \ref{fig:Diagram}), we have so far only established that it is a minimum of the region where $ \lvert S_1\lvert  + \lvert S_2\lvert \leq L $.  We now need to find out whether it is a minimum for the whole admissible region (the white region in Figure \ref{fig:Diagram}).  

To this end, we perform a Taylor expansion of $\mathcal{E}(u_1^c,u_2^c)$ in a neighbourhood of $ \mathcal{M}_S $ within the region $ \lvert S_1\lvert  + \lvert S_2\lvert \geq L $, which is also the region where $p_1/u_1^c+p_2/u_2^c \geq L$, by Equation \eqref{uic}. Since the slope of  the tangent line to the curve $p_1/u_1^c+p_2/u_2^c  = L$ at the point  $ \mathcal{M}_S $  is $ -\frac{D_1^ 2 p_1}{D_2^2 p_2} $, we choose two arbitrary constants, $\epsilon$ and $\delta$, such that $D_1^ 2 p_1 \epsilon+D_2^2 p_2\delta\leq 0$. Using Equation \eqref{uic}, the function $\mathcal{E}({u}_{1}^c,{u}_{2}^c) $ in Equation \eqref{eq:mathcalE} becomes
	\begin{linenomath*}\begin{align} \nonumber
			\mathcal{E}(u_1^c,u_2^c) & =\sum_{i=1}^{2} p_iD_i \text{ln}(u_i^c)  +  \gamma_{12}  u_1^c u_2^c (\lvert S_1\lvert  + \lvert S_2\lvert  -L),
			\\ 
			&=	\sum_{i=1}^{2} p_iD_i \text{ln}(u_i^c)+ \gamma_{12}  u_1^c u_2^c \left(\frac{p_1}{u_1^c} + \frac{p_2}{u_2^c} -L\right).
			\label{eq:mathcalEs0}
	\end{align}\end{linenomath*} 
Then the Taylor expansion of $\mathcal{E}(u_1^c,u_2^c)$ in a neighbourhood of $ \mathcal{M}_S $ within the region $p_1/u_1^c+p_2/u_2^c \geq L$ is
\begin{linenomath*}\begin{align}\nonumber
	\mathcal{E}(u_{1S}^c+\epsilon,u_{2S}^c+\delta)& \approx    \mathcal{E}(u_{1S}^c,u_{2S}^c)+ \partial_{u_1^c} \mathcal{E}(u_{1S}^c,u_{2S}^c) \epsilon + \partial_{u_2^c} \mathcal{E}(u_{1S}^c,u_{2S}^c)\delta\\ \nonumber
	&= \mathcal{E}(u_{1S}^c,u_{2S}^c)\\ \nonumber
			& \qquad + \frac{p_1 D_1}{D_2}\frac{D_1 D_2  L - \gamma_{12}(p_1 D_1 +p_2 D_2)}{p_1 D_1 +p_2 D_2}
			\epsilon\\ \nonumber
			& \qquad + \frac{p_2 D_2}{D_1}\frac{D_1 D_2  L - \gamma_{12}(p_1 D_1 +p_2 D_2)}{p_1 D_1 +p_2 D_2}
			\delta\\ \nonumber
			&= \mathcal{E}(u_{1S}^c,u_{2S}^c)\\ \nonumber
			& \qquad + \frac{p_1 D_1^2}{D_1 D_2}\frac{D_1 D_2  L - \gamma_{12}(p_1 D_1 +p_2 D_2)}{p_1 D_1 +p_2 D_2}
			\epsilon\\ \nonumber
			& \qquad + \frac{p_2 D_2^2}{D_1 D_2}\frac{D_1 D_2  L - \gamma_{12}(p_1 D_1 +p_2 D_2)}{p_1 D_1 +p_2 D_2}
			\delta\\ \nonumber
			&=  \mathcal{E}(u_{1S}^c,u_{2S}^c) \\ \nonumber
			& \qquad+ \frac{D_1 D_2  L - \gamma_{12}(p_1 D_1 +p_2 D_2)}{(D_1 D_2)(p_1 D_1 +p_2 D_2)} (D_1^ 2 p_1 \epsilon+D_2^2 p_2\delta) \\
			& \geq \mathcal{E}(u_{1S}^c,u_{2S}^c),
\end{align}\end{linenomath*}
if $\gamma_{12} > \frac{D_1 D_2  L}{p_1 D_1 +p_2 D_2}$, where the inequality uses $D_1^ 2 p_1 \epsilon+D_2^2 p_2\delta\leq 0$.
	
We now examine whether there are any other minima of $ \mathcal{E}(u_1^c,u_2^c)$ in the region where $ \lvert S_1\lvert  + \lvert S_2\rvert  > L $. By Equation \eqref{ssssic}, the condition $ \lvert S_1\rvert  + \lvert S_2\rvert  > L $ is equivalent to $ {p_1}/{u_1^c}+{p_2}/{u_2^c}> L $.  Therefore we have the following constraints
	\begin{linenomath*}\begin{align}
			\frac{p_1}{u_1^c}+\frac{p_2}{u_2^c}&> L, \nonumber \\
			u_i^c&\geq \frac{p_i}{\lvert L\lvert }, \text{ for } i=1,2. 
			\label{eq:constr2}
	\end{align}\end{linenomath*}
	A direct calculation using partial derivatives shows that there are no local minima of $ \mathcal{E}(u_1^c,u_2^c) $ (Equation \eqref{eq:mathcalEs0}) in the interior of the region of the plane $ (u_1^c, u_2^c) $ defined by Equation \eqref{eq:constr2}. Therefore any local minimum must occur on the boundary.  On the part of the boundary given by $ u_i^c =p_i/L $, for $ i=1,2 $, there is a unique minimum at \begin{linenomath*}\begin{equation}\label{eq:minimum2}
			\mathcal{M}_H=(u_{1H}^c, u_{2H}^c):=\left(\frac{p_1}{L},\frac{p_2}{L}\right).
	\end{equation}\end{linenomath*}
	This is also a local minimum of the region defined by Equation \eqref{eq:constr2}.  This can be shown by performing a Taylor expansion of $ \mathcal{E}(u_1^c,u_2^c) $ about the point $\mathcal{M}_H$, to give
	\begin{linenomath*}\begin{align}\nonumber
			\mathcal{E}(u_{1H}^c+\epsilon,u_{2H}^c+\delta)& \approx    \mathcal{E}(u_{1H}^c, u_{2H}^c)+ \partial_{u_1^c} \mathcal{E}(u_{1H}^c, u_{2H}^c) \epsilon + \partial_{u_2^c} \mathcal{E}(u_{1H}^c, u_{2H}^c)\delta\\\nonumber
			&= \mathcal{E}(u_{1H}^c, u_{2H}^c)+ L D_1 \epsilon + L D_2 \delta \\ \nonumber
			& \geq \mathcal{E}(u_{1H}^c, u_{2H}^c),
	\end{align}\end{linenomath*}
	where the inequality uses $\epsilon\geq0$, $\delta\geq0$, so that we remain in the $ u_i \geq p_i/L$ region in Figure \ref{fig:Diagram}.

In summary, if  $ 0<\gamma_{12}<\frac{D_1 D_2  L}{p_1 D_1 +p_2 D_2} $ then $ \mathcal{E}(u_1^c,u_2^c) $ (Equation \eqref{eq:mathcalE}) has a unique minimum, given by $\mathcal{M}_H$.  However, if $ \gamma_{1 2}> \frac{D_1 D_2  L}{p_1 D_1 +p_2 D_2} $ then $ \mathcal{E}(u_1^c,u_2^c) $ has two local minima, given by $\mathcal{M}_H$ and $\mathcal{M}_S$ (see Figure \eqref{fig:Diagram}).
	
Now, we recover the local minimizer $u_i(x)$ (Equation \eqref{eq:ui}) of the energy (Equation \eqref{eq:energy-local}). To give a concrete example, we use the parameter values $p_1=p_2=D_1=D_2=L=1$. If $(u_1^c,u_2^c)=\mathcal{M}_H$ then $u_1 (x)=u_2 (x)=1$, the homogeneous steady state, which we denote by $\mathcal{S}_H$.  If $(u_1^c,u_2^c)=\mathcal{M}_S$ then 
\begin{linenomath*}\begin{align}\label{eq:St}
			u_i (x)=\begin{cases}
				2, &\mbox{for $x \in S_i$} \\
				0, & \mbox{for $x \in [0,1]$\textbackslash$S_i$},
		\end{cases}
\end{align}\end{linenomath*}
with $ \lvert S_i\lvert  =1/2 $, for $ i=1,2 $, and $ \lvert S_1  \cap S_2\lvert =0 $.  This is a class of solutions we denote by $\mathcal{S}_S^{2,2}$, where the subscript $S$ stands for segregation and the superscript $2,2$ denotes the finite positive value that functions $u_1(x)$ and $u_2(x)$ take, respectively. To avoid any confusion, we want to stress that the points $\mathcal{M}_H$ (Equation \eqref{eq:minimum2}) and $\mathcal{M}_S$ (Equation \eqref{eq:minimum}) are local minima of $\mathcal{E}(u_1^c,u_2^c)$ (Equation \eqref{eq:mathcalE}), while the functions $\mathcal{S}_H$ and $\mathcal{S}_S^{2,2}$ are minimizers of the energy $E[u_1,u_2]$ (Equation \eqref{eq:energyloc0}).
    
In our example, if  $ 0<\gamma_{12}<1/2 $, $E(u_1^c,u_2^c)$ (Equation \eqref{eq:energy-local}) has a unique minimum, given by $\mathcal{S}_H$.  If $ \gamma_{1 2}>1/2 $ the energy has two local minima, given by $\mathcal{S}_H$ and $\mathcal{S}_S^{2,2}$. However, recall that $\mathcal{S}_H$ and $\mathcal{S}_S^{2,2}$ are derived by minimizing the energy (Equation \eqref{eq:energy-local}) in a particular  class of piece-wise constant functions given by Equation \eqref{eq:ui}. Therefore, the steady states $\mathcal{S}_H$ and $\mathcal{S}_S^{2,2}$ may not be minima of the full function space where solutions might live. However, the linear stability analysis performed in Section \ref{sec:linear_analysis}, and particularly Equation \eqref{eq:turingcondition}, suggests that in the limit as $\alpha$ tends to zero, $\mathcal{S}_H$ is stable if $ \gamma_{12} < 1 $. This gives rise to the diagram of analytically-predicted steady states given by the red and black lines in Figure \ref{fig:BifurcationDiagramHysteresis2}.
	
	
	
	\color{black}
	\begin{figure}[H]
		\centering	
		\includegraphics[width=.7\textwidth]{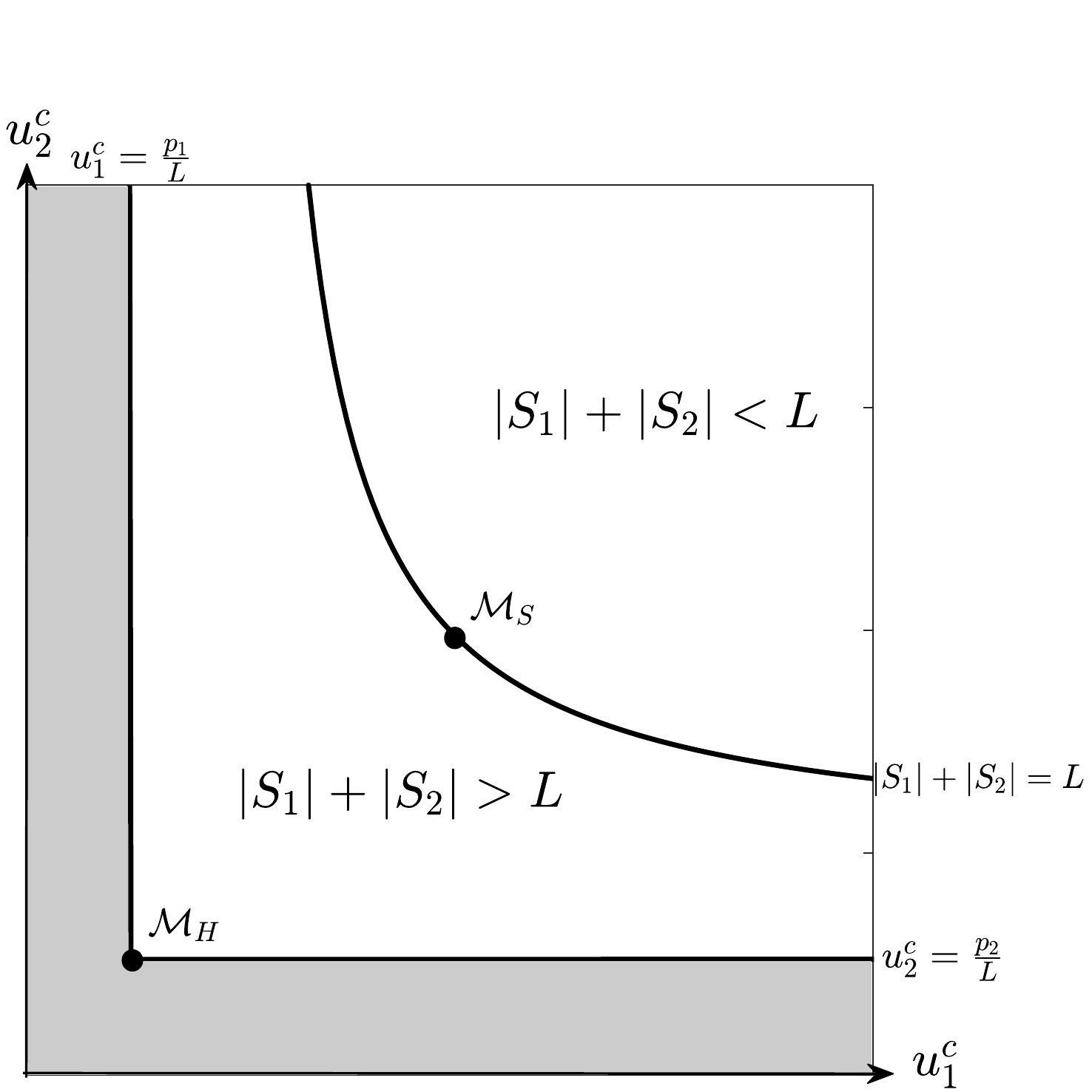}
		\caption{The white region represents the admissible domain, in which we look for the local minima of the function $ \mathcal{E}(u_1^c,u_2^c) $ (Equation \eqref{eq:mathcalE}). The point $ \mathcal{M}_H $, corresponding to the homogeneous steady state, is always a local minimum. Whether the point $ \mathcal{M}_S $ is a local minimum depends upon the value of $ \gamma_{1 2} $}
		\label{fig:Diagram}
	\end{figure}

	\subsubsection{Numerical verification}
	\label{sec:numerics1}
	
	{\noindent}The analysis of Section \ref{sec:analytic1} suggests that for $p_1=p_2=D_1=D_2=L=1$, when $ 1/2 < \gamma_{12} < 1 $ and the averaging kernel $K$ is arbitrarity small, Equation \eqref{eq:system} 
	should exhibit bistability between the homogeneous solution, $\mathcal{S}_H$, and an inhomogeneous solution arbitrarily close to $\mathcal{S}_S^{2,2}$. Here, we verify this numerically.  
	
	Figure \ref{fig:BifurcationDiagramHysteresis2} summarises our results.  To produce this figure, we start with $ K=K_\alpha $ and $\gamma_{12}=1.2$, so that the homogeneous steady state is unstable. The initial condition is a small perturbation of the solution given in Equation \eqref{eq:St} which we run to numerical steady state.  We then reduce the magnitude of $\gamma_{12}$ by $\Delta \gamma_{12}=0.05$ and solve the system again using a small random perturbation of the previous simulation as initial condition. We then repeat this process of reducing $\gamma_{12}$ and re-running to steady state until the system returns to the homogeneous steady state. This process of slowly changing one parameter and re-running to steady state is a type of numerical bifurcation analysis used in many previous studies, e.g. \cite{painter2011spatio}.  The numerical scheme we use for solving our particular system is detailed in \cite{HPLG21}.
	
	We examine three different values of $ \alpha $ in Figure \ref{fig:BifurcationDiagramHysteresis2}.  For each of these, we observe that the inhomogeneous solution persists below $\gamma_{12}=1$ and above $\gamma_{12}=1/2$, as predicted by our the calculations of Section \ref{sec:analytic1}.  Furthermore, as $ \alpha $ decreases (towards the local limit), the numerical branches appear to tend towards the branch calculated in Section \ref{sec:analytic1}.
	\begin{figure}[H]
		\centering	
		\includegraphics[width=0.5\textwidth]{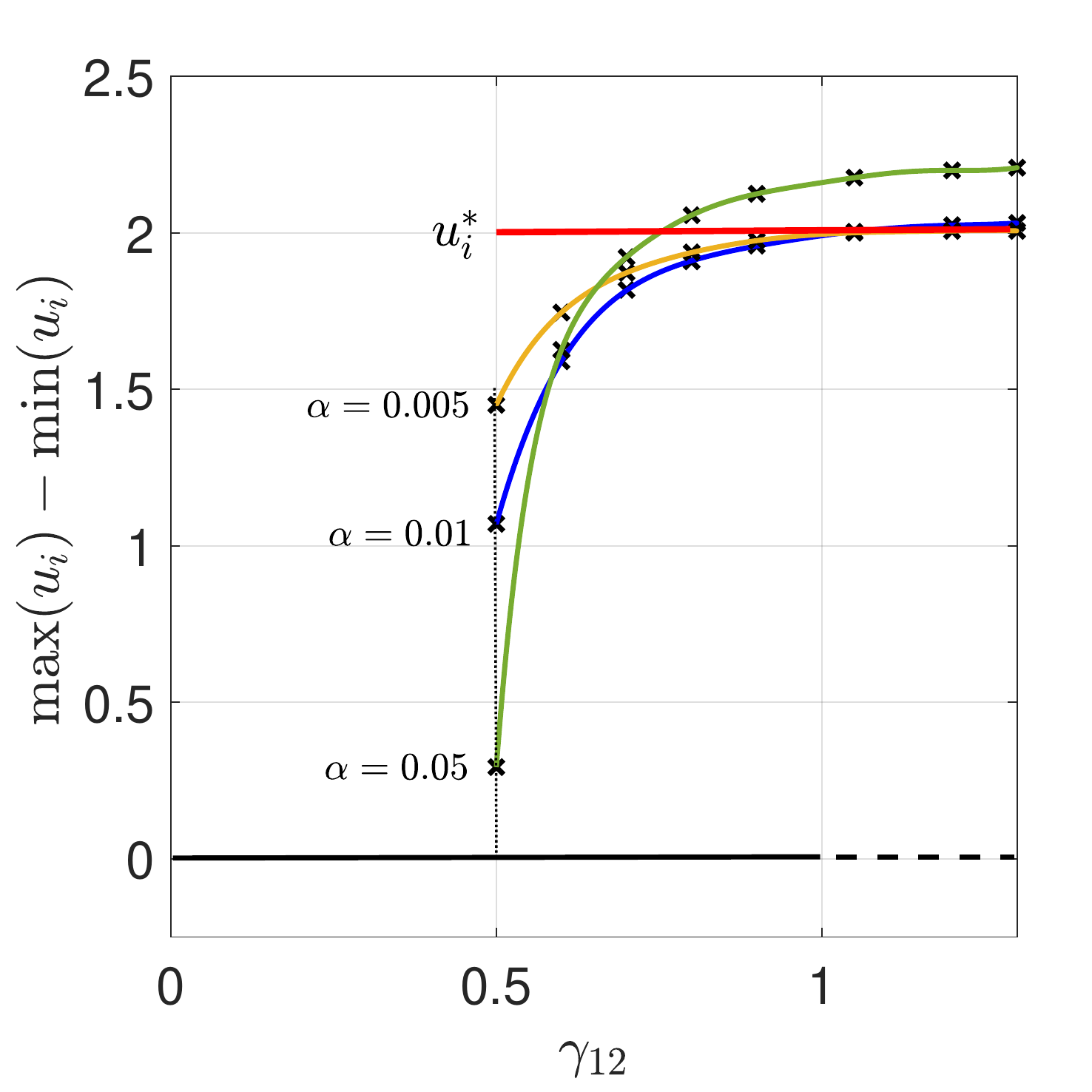}
		\caption{Numerically computed bifurcation diagram of Equation \eqref{eq:system}, with $ K=K_{\alpha}$ (Equation \eqref{eq:top-hat}), for different values of $ \alpha $. The other parameter values are $p_1=p_2=D_1=D_2= L=1 $. The red solid line shows the minimum energy branch computed analytically using the techniques in Section \ref{sec:analytic1}, pertaining to the limit $\alpha \rightarrow 0$.  The numerical simulations show that the system admits bistability for $ 0.5 < \gamma_{12} < 1 $, in agreement with our analytic predictions}
		\label{fig:BifurcationDiagramHysteresis2}
	\end{figure}
	Finally, in Figure \eqref{fig:Hysteresis_complete}, we show some numerical stationary solutions for different values of $ \alpha $, as $ \gamma_{12} $ varies in the range $ [0.55,1.05] $. We observe that, as $ \alpha $ decreases, the numerical solution appears to tend to a piece-wise constant function of the class given in Equation \eqref{eq:St} and predicted by the analysis of Section \ref{sec:analytic1}.
	\begin{figure}[H]
		\centering	
		\includegraphics[width=1\textwidth]{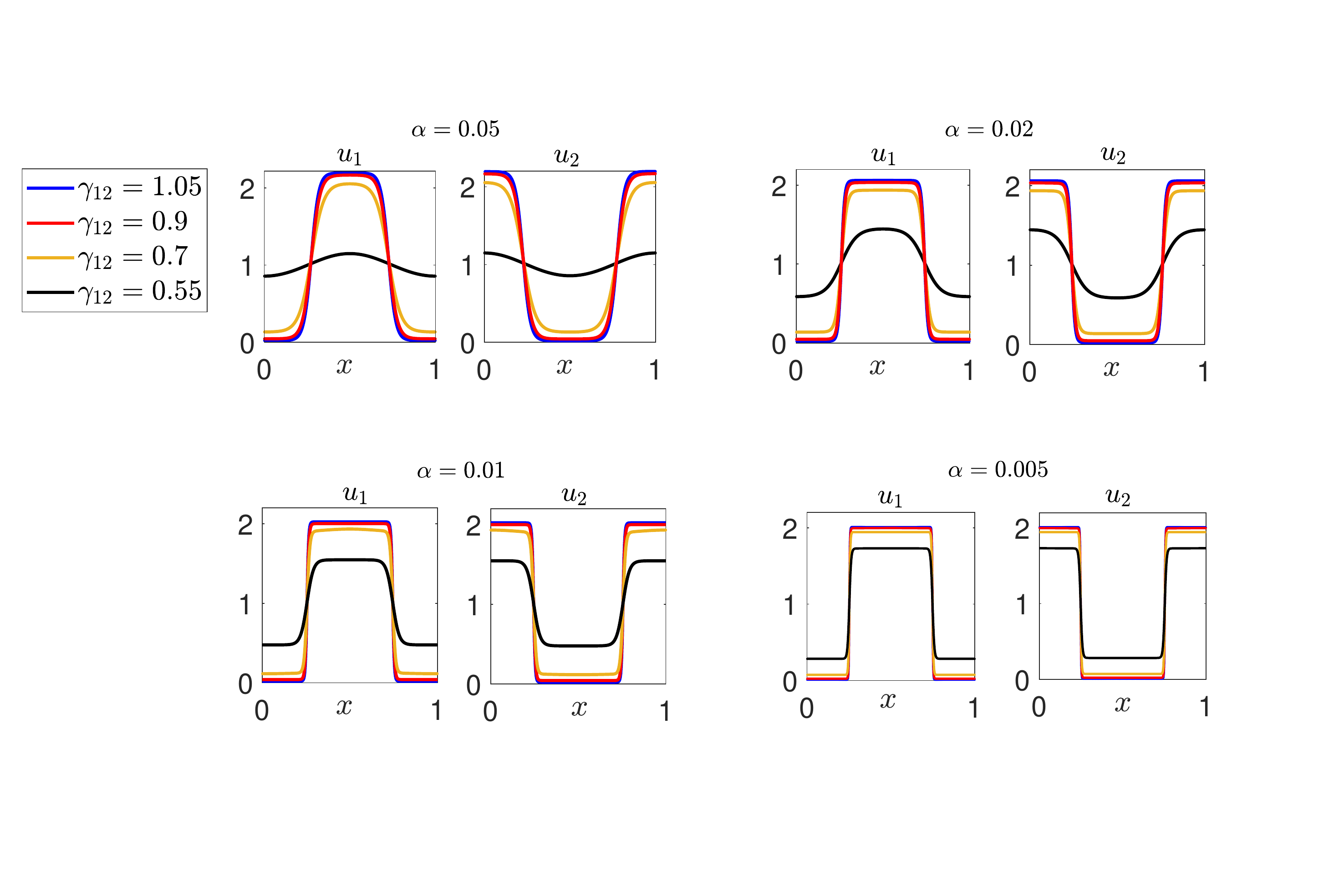}
		\caption{Comparison between numerically computed stationary  $ \mathcal{S}_S^{2,2} $ solutions to Equation \eqref{eq:system}, with $ K=K_{\alpha} $ (Equation \eqref{eq:top-hat}), for different values of $ \gamma_{12}>0 $ and $ \alpha $. The parameter values used in the simulations are $ D_1=D_2=1 $, $ p_1=p_2=1 $ }
		\label{fig:Hysteresis_complete}
	\end{figure}
	
	\subsection{The case $\mathbf{\gamma_{1 1} =\gamma_{22}=0 }$ with mutual attraction, $\mathbf{\gamma_{12}=\gamma_{21}<0}$ }\label{sub:2}
	
	\subsubsection{Analytic results in the local limit}
	\label{sec:analytic2}	
	
	{\noindent}As in Section \ref{sec:analytic1}, here we will look for the minimizers of the local version of the energy (Equation \eqref{eq:energy-local}) in the class of piece-wise constant functions defined as 
	\begin{linenomath*}\begin{align}
			\label{eq:ui1}
			u_i (x)=\begin{cases}
				u_{i}^c, &\mbox{for $x \in S_{i}$}, \\
				0, & \mbox{for $x \in [0,L]$\textbackslash$S_{i}$},
			\end{cases}
	\end{align}\end{linenomath*}
	where $ u_{i}^c \in \mathbb{R}^{+} $ and $S_{i}$ are subsets of $[0,L]$, for $i\in\{1,2\}$. 
	
	Placing Equation \eqref{eq:ui1} into Equation \eqref{eq:energy-local}, and repeating the same argument of Section \ref{sec:analytic1}, we obtain
	\begin{linenomath*}\begin{align}\label{eq:energyloc1}
			E[u_1 ,u_2 ]=&	\sum_{i=1}^{2}  p_{i} D_i \text{ln}(u_{i}^c)  + \gamma_{12} u_{1}^c u_{2}^c \lvert S_{1} \cap S_{2}\lvert.
	\end{align}\end{linenomath*}
	In this case, to minimize Equation \eqref{eq:energyloc1} we note that, since $ \gamma_{1 2}<0 $, $E[u_1 ,u_2 ] $ can be lowered by increasing $\lvert S_1\cap S_2\lvert$, whilst keeping everything else the same. Therefore if we keep $\lvert S_1\lvert$ and $\lvert S_2 \lvert $ unchanged, then $\lvert S_1\cap S_2\lvert $ is maximised when either $S_1 \subseteq S_2$ or $S_2 \subseteq S_1$, so that $\lvert S_1\cap S_2 \lvert =\min_i\lvert S_i \lvert$. Thus 
	\begin{linenomath*}\begin{align}\label{argmin_neg_gamma}
			\mbox{argmin}_{u_1,u_2}E[u_1 ,u_2 ]=\mbox{argmin}_{u_1,u_2}\left[\sum_{i=1}^{2} p_iD_i \text{ln}[u_i^c]+ \gamma_{12}\min\{p_1 u_2^c,p_2u_1^c\}\right],
	\end{align}\end{linenomath*}
	and therefore we have that $E[u_1 ,u_2 ] \rightarrow -\infty$ as $\min\{p_1 u_2^c,p_2u_1^c\} \rightarrow \infty$. As we approach this limit, $u_1^c, u_2^c$ become arbitrarily large, so $u_1 $ and $ u_2 $ (Equation \eqref{eq:ui1}) become arbitrarily high, arbitrarily narrow functions with overlapping support. We will denote the limit of this solution by $ \mathcal{S}_A^{\infty} $, in which the subscript $ A $ stands for aggregation and the $\infty$ superscript denotes that the solution becomes unbounded in the local limit.  Thus $E[u_1 ,u_2 ]  $ is minimized by $ \mathcal{S}_A^{\infty} $ whenever $\gamma_{12}$ is negative, regardless of its magnitude.

One can also show, using a very similar argument to Section \ref{sec:analytic1} (details omitted), that the homogeneous steady state, $\mathcal{S}_H$, is the only other possible local minimiser of the energy that satisfies Equation (\ref{ssssic}), and this is only a local minimum when $\gamma_{12}>- L(p_1 D_1 + p_2 D_2)/(p_1 p_2)$. However, linear stability analysis (Equation \eqref{eq:eigenvalues}) suggests that, in the limit as $\alpha$ tends to zero, the homogeneous steady state is linearly stable only if 
$\gamma_{12}  > -L \sqrt{D_1 D_2/(p_1p_2)}$. Since Young's inequality for products implies that $L \sqrt{D_1 D_2/(p_1p_2)}< L(p_1 D_1 + p_2 D_2)/(p_1 p_2)$, any time  $\mathcal{S}_H$ is linearly stable it is also a local energy minimiser within the set of functions given by Equation (\ref{eq:ui1}). 
The red and black lines in Figure \ref{fig:BifurcationDiagramHysteresis} are the conclusion from combining all the results from Section \ref{sec:analytic2}, both energy functional and linear stability analysis, in the case where $p_1=p_2=D_1=D_2=L=1$.
	

	\subsubsection{Numerical verification}
	\label{sec:numerics2}
	
	The analysis of Section \ref{sec:analytic2} suggests that when $\gamma_{12}  > -L \sqrt{D_1 D_2/(p_1p_2)}$, $\gamma_{12}<0$, and $ \alpha $ is arbitrarily small, Equation \eqref{eq:system} should display bistability between the homogeneous solution and an inhomogeneous solution, whose structure tends towards $ \mathcal{S}_A^{\infty} $ as $ \alpha \rightarrow 0 $.  Here we verify this conjecture numerically, with results shown in Figures \ref{fig:BifurcationDiagramHysteresis} and \ref{fig:BlowUp}.
	
	To construct these figures, we perform a similar analysis to Section \ref{sec:numerics1}.  We simulate Equation \eqref{eq:system} with $ K=K_{\alpha}$ (Equation  \eqref{eq:top-hat}) for small values of $\alpha$. We use the parameter values $p_1=p_2=D_1=D_2=L=1$, as in Section \ref{sec:analytic1}.  For these values, the constant steady-state is stable to perturbations at all wavenumbers for $-1<\gamma_{12}<0$.  Therefore, we begin our analysis by setting $\gamma_{12}=-1.2$, reducing the magnitude of $\gamma_{12}$ by a small amount ($\Delta \gamma_{12}=0.05$) at each iteration of the analysis, as in Section \ref{sec:numerics1}.  
	
	Our results show that patterns persist beyond $\gamma_{12}=-1$, and the extent of this persistence depends on $\alpha$ (Figure \ref{fig:BifurcationDiagramHysteresis}). As $ \alpha $ is decreased, the numerical stationary states become higher, narrower functions with qualitatively similar shapes, as predicted by the previous analysis (Figure \ref{fig:BlowUp}).
	\begin{figure}[H]
		\centering
		\subfloat[]
		{\includegraphics[width=.35\textwidth]{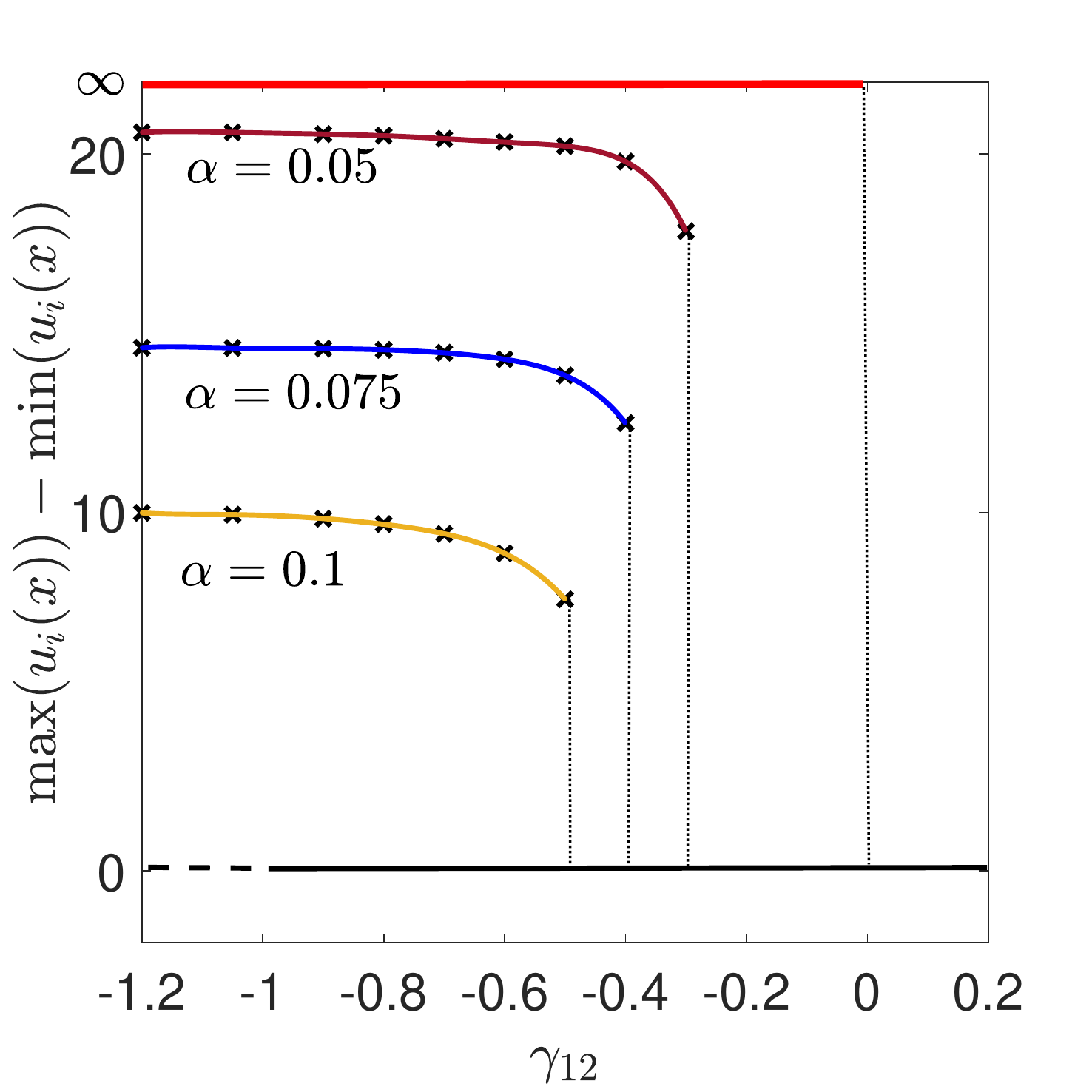}\label{fig:BifurcationDiagramHysteresis}}
		\hspace{0.0cm}
		\subfloat[]
		{\includegraphics[width=.425\textwidth]{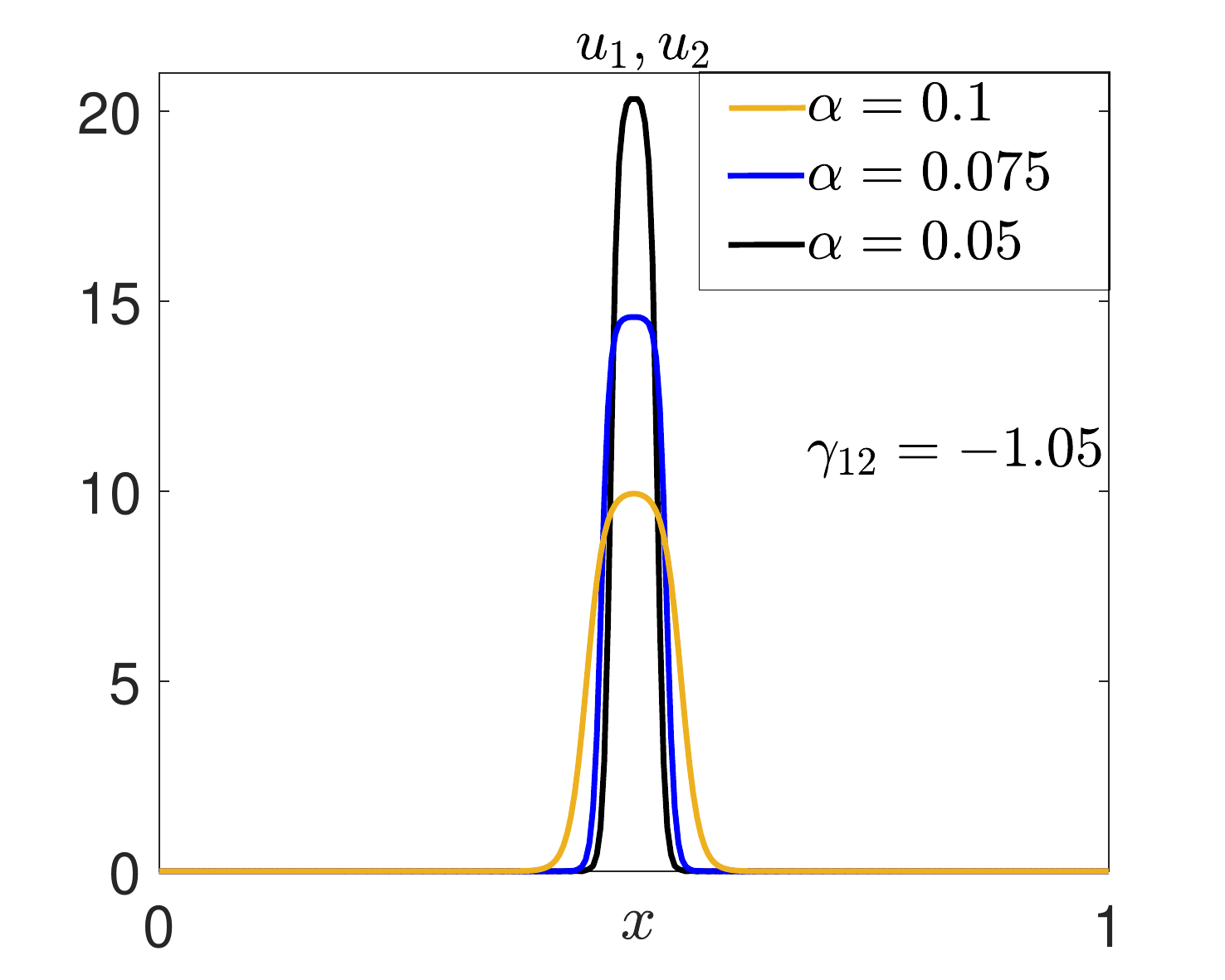}\label{fig:BlowUp}}
		\caption{Numerical investigation of Equation \eqref{eq:system}, with $ K=K_{\alpha} $ (Equation \eqref{eq:top-hat}) for $ \gamma_{12}<0 $. The parameter values are $p_1=p_2=D_1=D_2=1$, $ L=1 $. Panel (a) gives a numerical bifurcation diagram showing the bistability between the homogeneous steady state (in black) and the inhomogeneous steady states $\mathcal{S}_A^{\infty}  $, for different values of $ \alpha $. Panel (b) shows the corresponding numerical stationary solutions whan $ \gamma_{1 2}=-1.05 $, for different values of $ \alpha $.  As $ \alpha $ decreases, the solutions appear to tend towards the $\mathcal{S}_A^{\infty}  $ solution}
		\label{fig:MutualAttraction} 
	\end{figure}
	%

\subsection{The case $\mathbf{\gamma_{1 1}, \gamma_{22} \neq 0 }$}
\label{sec:gammaneq0}

The case $ \gamma_{1 1}, \gamma_{22} \neq 0  $ uses similar arguments to those in Section \ref{sub:1}.  We therefore just summarise the results here, leaving details of the calculations for Appendix \ref{appendix:A}.  

In our computations, we consider the case $\gamma_{22}=\gamma_{11}$ and fix the other parameter values as $p_1=p_2=D_1=D_2=L=1$. The analysis of this case reveals five distinct classes of qualitatively-different stable solutions (Figure  \ref{fig:Msolutions}), each of which we have verified through numerical analysis (where throughout this section we use `stable' to mean `Lyapunov stable').  These are (i) territory-like segregation patterns, $ \mathcal{S}_S^{2,2} $, the height of which remains finite as $K$ becomes arbitrarily narrow, (ii) segregation patterns where the height of both species becomes arbitrarily high as $K$ becomes arbitrarily narrow, denoted by $ \mathcal{S}^{\infty,\infty}_S $, (iii) segregation patterns where the height of just one species becomes arbitrarily high as $K$ becomes arbitrarily narrow but the other remains at finite height, denoted by $ \mathcal{S}_S^{1,\infty} $,  (iv) aggregation patterns, $ \mathcal{S}_A^{\infty} $, where the height of both species becomes arbitrarily high as $K$ becomes arbitrarily narrow,  and  (v) the spatially homogeneous solution $ \mathcal{S}_H $.

Figure \ref{fig:InstabilityRegions} shows the parameter regions in which the analysis from Appendix \ref{appendix:A} predicts we should see these various solutions. Notice that there are regions in which we have two-, three-, and even four-fold stability.  These calculations are verified numerically in Figures \ref{fig:BD} and \ref{fig:BD_mutual_avoidance_self_attraction}.  In particular, Figures \ref{fig:BD} and \ref{fig:BD_mutual_avoidance_self_attraction}  show that, as $\alpha$ becomes smaller, so the numerical results become closer to our analytic predictions. 


As shown in Figure \ref{fig:InstabilityRegions}, when species exhibit mutual attraction ($\gamma_{12}<0$), our analysis predicts two stationary states: the homogeneous distribution $S_H$ and the aggregation pattern $S_A^{\infty}$. In particular, if $\gamma_{12}<0$ and species show mutual avoidance, i.e. $\gamma_{11}>0$, there always exists a region in the parameter space in which both stationary states, $S_H$ and $S_A^{\infty}$, are stable. However, if the magnitude of self-avoidance $\gamma_{11}$ is relatively weaker than the rate of mutual-attraction $\gamma_{12}$, aggregation is more favored than the homogeneous distribution. In this case, $S_A^{\infty}$ is the only stable steady state, while the $S_H$ solution is unstable. 

On the other hand, in the mutual- and self-attraction case ($\gamma_{12}<0$, $\gamma_{11}<0$), bistability between the homogeneous distribution $S_H$ and the aggregation pattern $S_A^{\infty}$ is observed as long as the magnitudes of $\gamma_{12}$ and $\gamma_{11}$ are sufficiently small. However, if the rates of mutual and self-attraction become stronger, aggregation is favoured over the homogeneous distribution. Consequently, as the magnitudes of $\gamma_{11}$ and $\gamma_{12}$ increase, the homogeneous solution $S_H$ loses stability.

The scenario becomes even richer when $\gamma_{12}>0$. In particular, if the species exhibit mutual avoidance ($\gamma_{12}>0$) and self-avoidance ($\gamma_{11}>0$), the stable steady states predicted by our analysis are the homogeneous solution $S_H$ and segregation pattern $S_S^{2,2}$. When the strength of self-repulsion ($\gamma_{11}$) is relatively stronger than the tendency to avoid individuals from the other species ($\gamma_{12}$), the homogeneous distribution is favoured over aggregation with conspecifics, so that $S_H$ is the only stable steady state. However, if the rate of mutual avoidance $\gamma_{12}$ increases, the tendency to avoid individuals from the foreign species promotes the formation of spatial distributions in which the two species are segregated into distinct sub-regions of space. Indeed, Figure \ref{fig:InstabilityRegions} shows that as $\gamma_{12}$ increases, the segregation pattern $S_S^{2,2}$ acquires stability. However, as long as the magnitude of self-avoidance is sufficiently strong, the homogeneous distribution remains stable. Indeed, we observe that there is a parameter region in which the system shows bistability between $S_H$ and $S_S^{2,2}$. Finally, if the strength of mutual avoidance $\gamma_{12}$ becomes sufficiently stronger than the propensity to avoid conspecifics, segregation becomes more favored over the homogeneous distribution. Indeed, as $\gamma_{12}$ increases, $S_H$ loses its stability.

In the mutual avoidance ($\gamma_{12}>0$) and self-attraction ($\gamma_{11}<0$) scenario, the stable states predicted by our analysis include $S_H$ (homogeneous) and $S_S^{2,2}$ (territory-like segregation) as before, but also $ \mathcal{S}_S^{\infty, \infty}$ (self-aggregated species that are segregated from one another) and $ \mathcal{S}_S^{1,\infty}$ (segregated species where only one population is self-aggregated). If the magnitudes of self-attraction $\gamma_{11}$ and mutual avoidance $\gamma_{12}$ are sufficiently small, the homogeneous distribution, $S_H$, is also stable. However, for small values of $\gamma_{11}$, as the rate of mutual avoidance $\gamma_{12}$ increases, we observe the same scenario discussed above: $S_S^{2,2}$ gains stability and there exists a region in the parameter space in which both $S_S^{2,2}$ and $S_H$ are stable. Finally $S_H$ loses stability as $\gamma_{12}$ increases further. We also observe that high rates of self-attraction $\gamma_{11}$ favour the formation of sub-regions with high densities of individuals. Therefore, when the magnitude of $\gamma_{11}$ is strong, $ \mathcal{S}^{\infty}_S $ and $ \mathcal{S}^{\infty}_H $ solutions are favored over the homogeneous distribution $S_H$ and the inhomogeneous distribution $S_S^{2,2}$, which become unstable. 

Finally, we verify this multi-stability numerically for small $\alpha$, with results shown in Figures \ref{fig:BD} and \ref{fig:BD_mutual_avoidance_self_attraction}.  As in the $\gamma_{11}=\gamma_{22}=0$ cases, the numerics follow our analytic predictions well, giving better approximations for smaller $\alpha$.
 

\begin{figure}[H]
	\centering
	\subfloat[]
	{\includegraphics[width=1\textwidth]{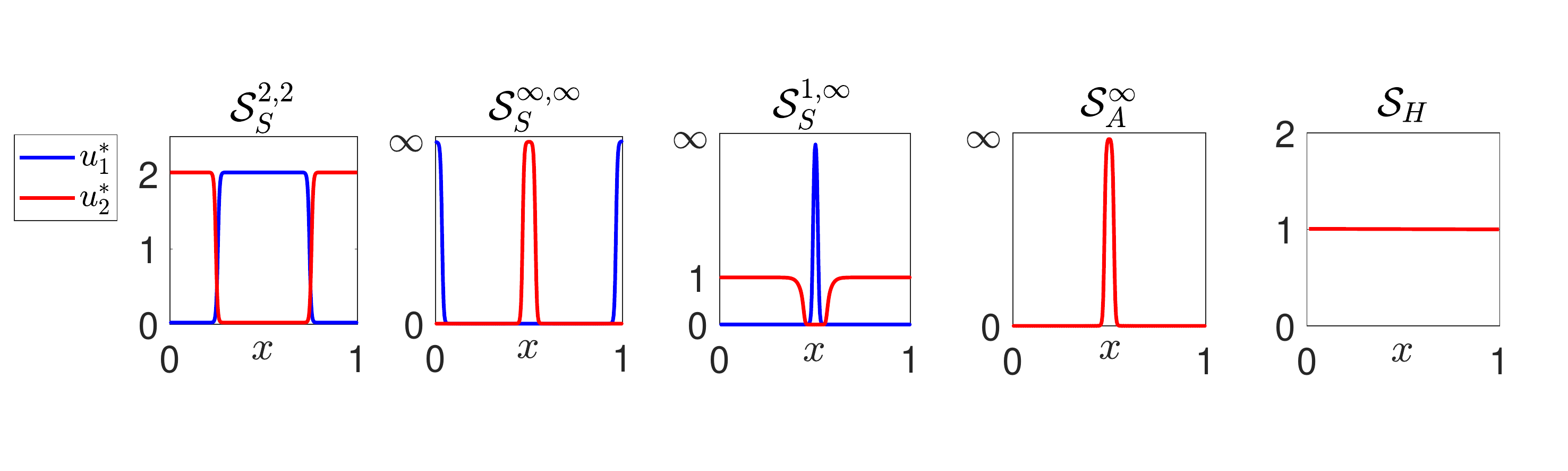}\label{fig:Msolutions}}
	\vspace{0.0cm}
	\subfloat[]
	{\includegraphics[width=0.8\textwidth]{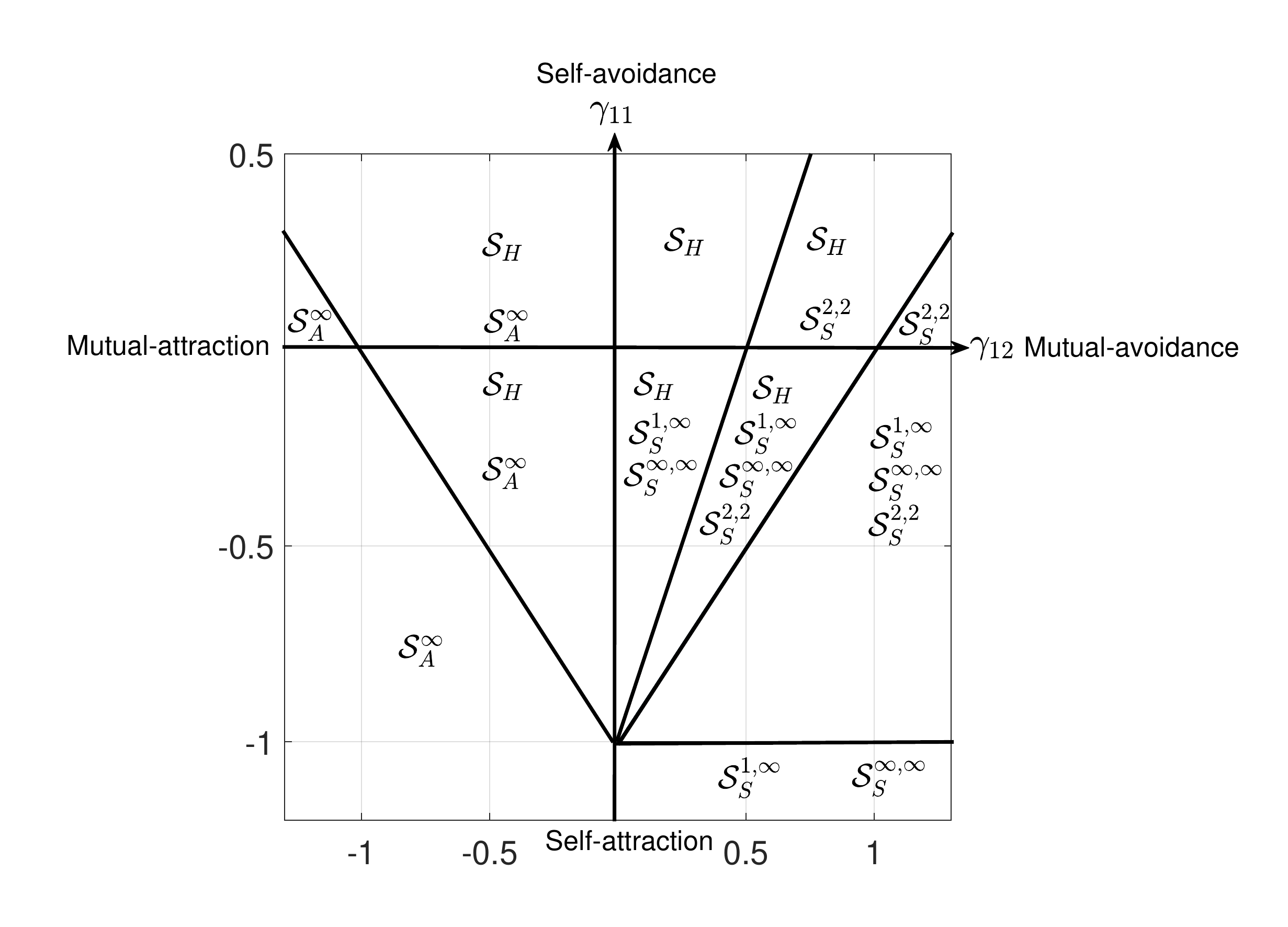}\label{fig:InstabilityRegions}}
	\caption{Panel (a) shows the five qualitatively-different local minimum energy states revealed by the analysis in Appendix \ref{appendix:A}.  Note that the $ \mathcal{S}_{S}^{1,\infty} $ solution also allows for $u_1 $ and $u_2 $ to be swapped.  These plots were produced by setting $ K=K_{\alpha} $, $\alpha=0.025$ and by fixing the following parameter values: $p_1=p_2=D_1=D_2=L=1$. For each graph of Panel (a), we fixed different values of the parameters $\gamma_{11}$ and $\gamma_{12}$, in particular we used: $\gamma_{11}=0.2$ and $\gamma_{12}=1.05$, for $S_S^{2,2}$; $\gamma_{11}=-0.15$ and $\gamma_{12}=0.4$, for $ \mathcal{S}_S^{\infty, \infty} $ and $ \mathcal{S}_S^{1,\infty}$; $\gamma_{11}=0.2$ and $\gamma_{12}=-1.05$, for $ \mathcal{S}_A^{\infty}$; $\gamma_{11}=0.2$ and $\gamma_{12}=0.2$, for $ \mathcal{S}_H$. Panel (b) shows the minimum energy solutions to Equation \eqref{eq:system} in different subregions of the plane $(\gamma_{12}, \gamma_{11}) $, for $ N=2 $, $ \gamma_{2 2}=\gamma_{1 1} $ and $\gamma_{12} = \gamma_{21}$, predicted by the analysis in Appendix \ref{appendix:A}. This graph is obtained by fixing the following parameter values: $p_1=p_2=D_1=D_2=L=1$ 
	}
	\label{fig:InstabilityRegionsComplete}
\end{figure}

In the following Lemma, we summarize the results shown in Figure \ref{fig:InstabilityRegionsComplete}, which are derived in Appendix \ref{appendix:A}.
\begin{lemma}\label{l:2}
Let $ \gamma_{2 2}=\gamma_{1 1} $, $\gamma_{21}=\gamma_{12}$ and $p_1=p_2=D_1=D_2=L=1$, and use `minimum energy' to mean `local minimum energy'.\\
\underline{Case A}: Self avoidance (${\gamma_{11}>0}$) and mutual avoidance (${\gamma_{12}>0})$.
\begin{enumerate}
    \item If  $ \gamma_{11} > 2\gamma_{12}-1 $ then the minimum energy state is $\mathcal{S}_H$.  
    \item If $ 0<\gamma_{11} < 2\gamma_{12}-1 $ then $\mathcal{S}_H$ and $\mathcal{S}_S^{2,2}$ are both minimum energy states.
\end{enumerate}
\underline{Case B}: Mutual attraction (${\gamma_{12}<0}$).
\begin{enumerate}
    \item If $\gamma_{11}> -\gamma_{12}-1$ then $\mathcal{S}_H$ and $\mathcal{S}_A^{\infty}$ are minimum energy states.
    \item If $\gamma_{11}< -\gamma_{12}-1$ then the minimum energy state is $\mathcal{S}_A^{\infty}$.
\end{enumerate}
\underline{Case C}: Self attraction (${\gamma_{11}<0}$) and mutual avoidance (${\gamma_{12}>0})$.
\begin{enumerate}
    \item If $\gamma_{11}> 2\gamma_{12}-1$ then $\mathcal{S}_H$, $\mathcal{S}_S^{\infty,\infty}$ and $\mathcal{S}_S^{1,\infty}$ are minimum energy states.
    \item If $\gamma_{12}-1<\gamma_{11}<2 \gamma_{12}-1$ then $\mathcal{S}_H$, $\mathcal{S}_S^{\infty,\infty}$, $\mathcal{S}_S^{1,\infty}$, and $\mathcal{S}_S^{2,2}$ are minimum energy states.
    \item If $-1<\gamma_{11}< \gamma_{12}-1$ then $\mathcal{S}_S^{\infty,\infty}$, $\mathcal{S}_S^{1,\infty}$, and $\mathcal{S}_S^{2,2}$ are minimum energy states.
    \item If $\gamma_{11}< -1$ then  $\mathcal{S}_S^{\infty,\infty}$ and $\mathcal{S}_S^{1,\infty}$ are minimum energy states.
\end{enumerate}
\end{lemma}

\begin{figure}[H]
	\centering
	\subfloat[$ \mathcal{S}_S^{2,2}$ solution with $ \gamma_{11}=0.2 $]
	{\includegraphics[width=.32\textwidth]{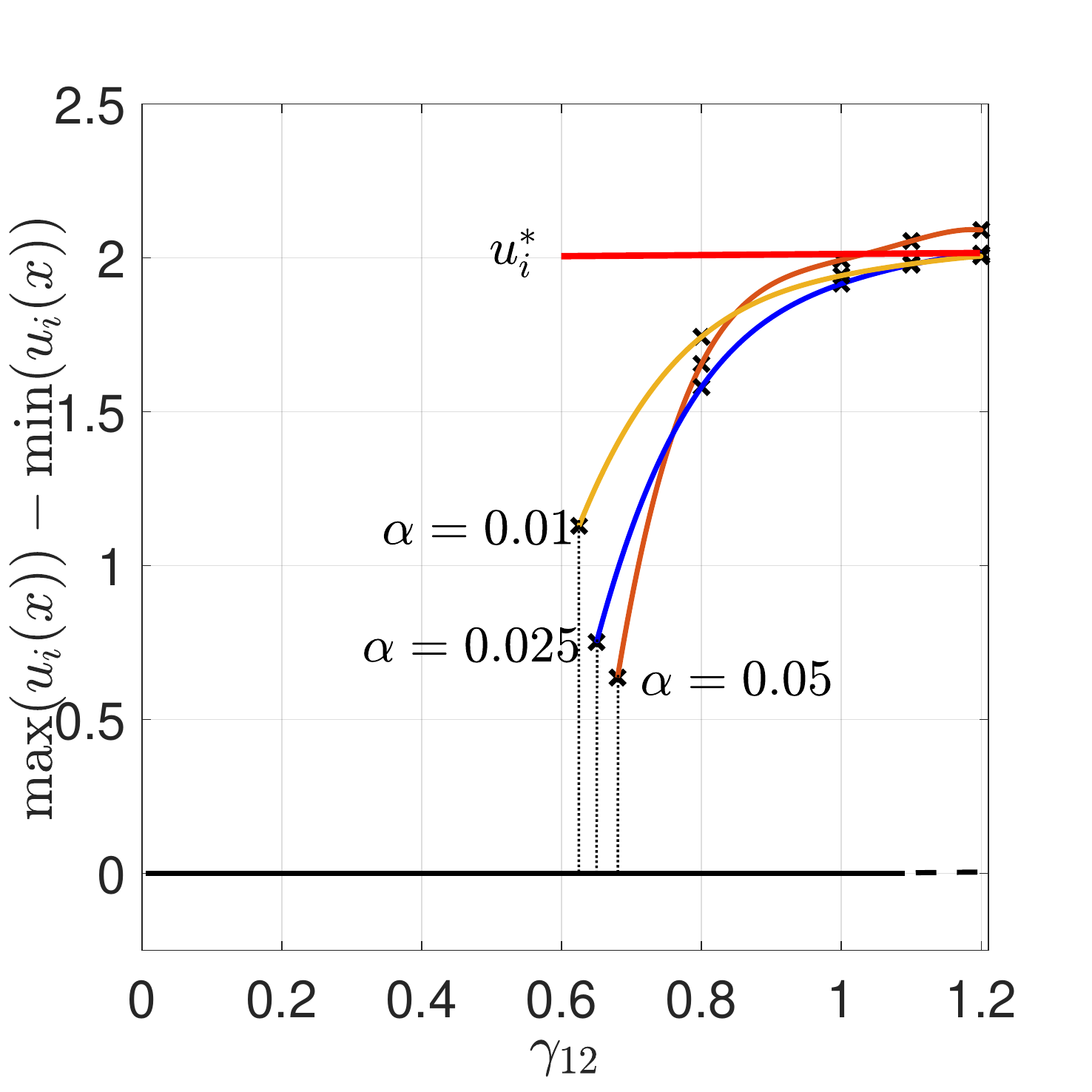}\label{fig:BD_mutualself_avoidance}}
	\hspace{0.0cm}
	\subfloat[$ \mathcal{S}_A^{\infty}$ solution with $\gamma_{11}=0.2 $]
	{\includegraphics[width=.32\textwidth]{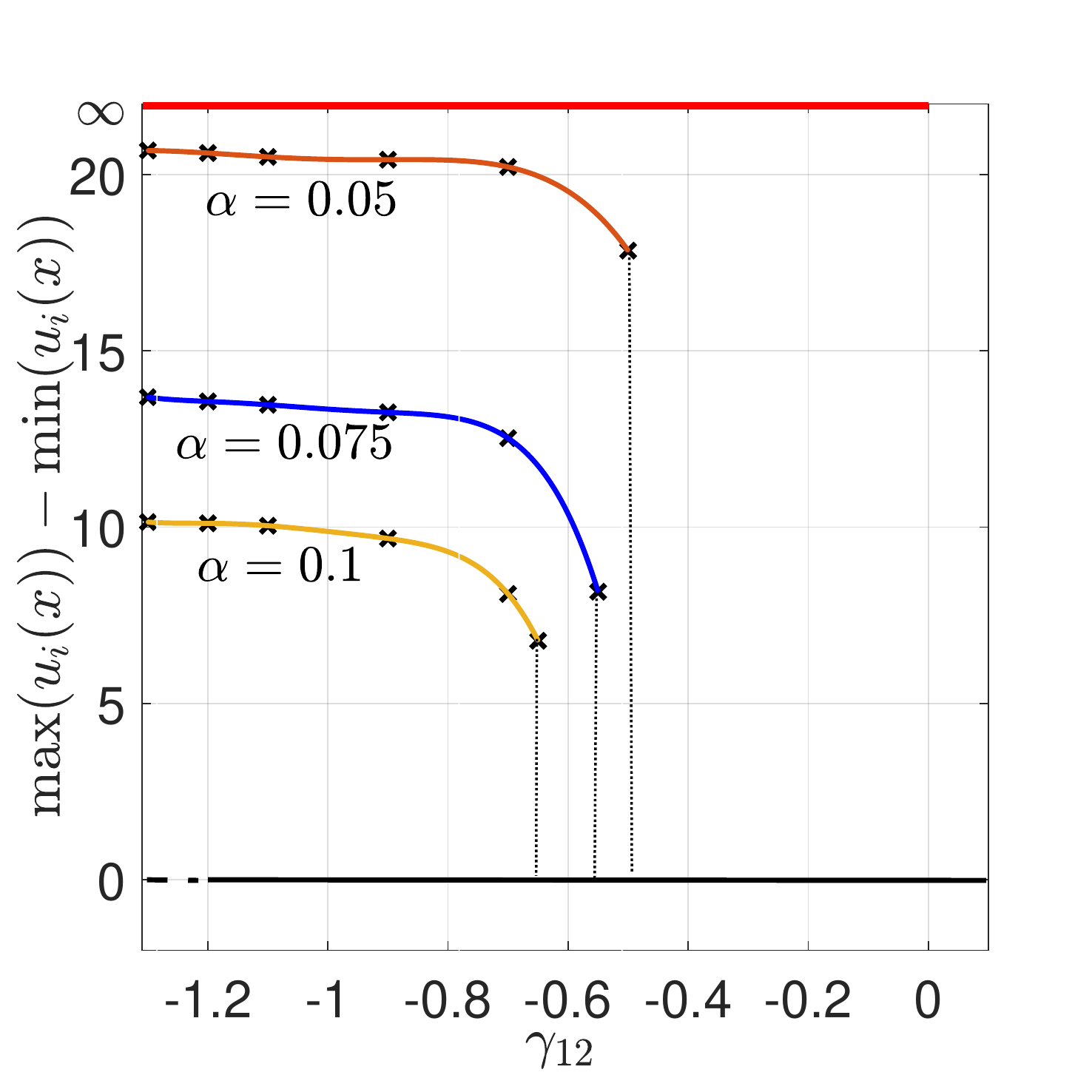}\label{fig:BD_mutual_attraction_self_avoidance}}
	\hspace{0.0cm}
	\subfloat[$\mathcal{S}_A^{\infty}$ solution with $\gamma_{11}=-0.2 $]
	{\includegraphics[width=.32\textwidth]{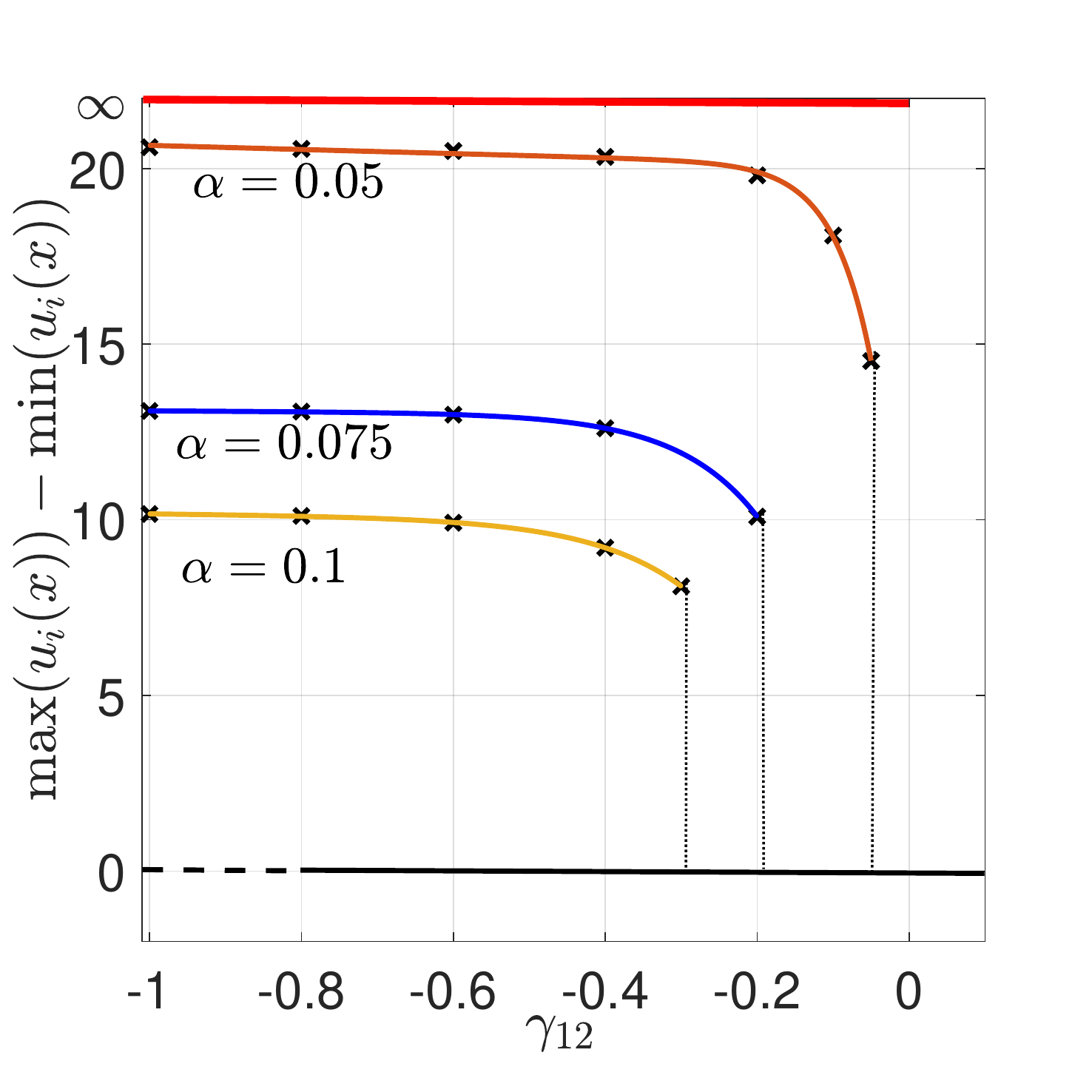}\label{fig:BD_mutualself_attraction}}
	\caption{Bifurcation diagrams of Equation \eqref{eq:system}, with $ K=K_{\alpha} $ (Equation \eqref{eq:top-hat}), for different values of $ \gamma_{11} $, as $ \alpha $ is decreased. The other parameter values are $p_1=p_2=D_1=D_2=1$, $ L=1 $. Panel (a) shows hysteresis between the homogeneous steady state $ \mathcal{S}_H $ (in black) and the stationary solution $ \mathcal{S}_S^{2,2} $ for different values of $ \alpha $. As $ \alpha $ decreases, the numerical branches tend towards the analytically-predicted branch (in red). Panels (b) and (c) show hysteresis between the homogeneous steady state $ \mathcal{S}_H $ (in black) and the stationary solution $ \mathcal{S}^{\infty}_A $ for different values of $ \alpha $. As $ \alpha $ decreases, the height of the numerical branches tends towards $\infty$}
	\label{fig:BD} 
\end{figure}

\begin{figure}[H]
	\centering
	{\includegraphics[width=1\textwidth]{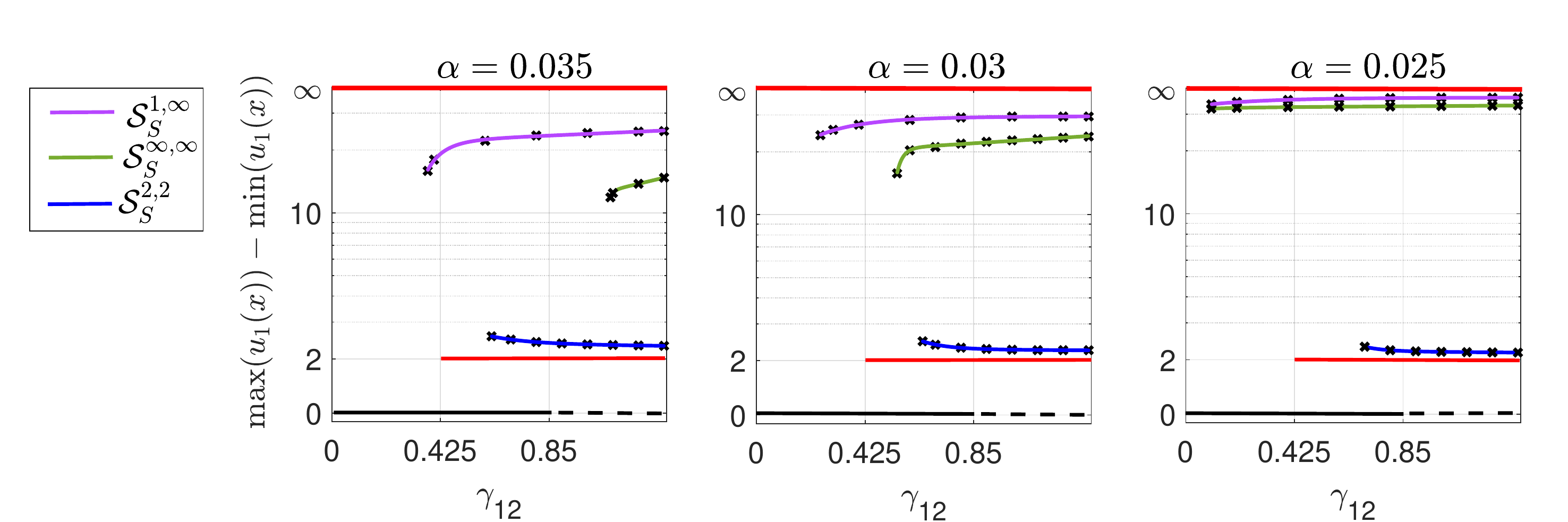}}
	\caption{Bifurcation diagrams of Equation \eqref{eq:system}, with $ K=K_{\alpha} $ (Equation \eqref{eq:top-hat}), for $ \gamma_{11}=-0.15 $ and different values of $ \alpha $. The other parameter values are $p_1=p_2=D_1=D_2=1$, $ L=1 $. The graphs show the coexistence between the homogeneous steady state $ \mathcal{S}_H $ (in black), computed analytically, and the stationary solutions $ \mathcal{S}_S^{2,2} $ (in blue), $ \mathcal{S}^{\infty}_S $ (in green) and  $ \mathcal{S}_{H}^{\infty} $ (in violet), computed numerically. As $\alpha$ decreases, the numerical branches tend towards the analytical branches (in red).}
	\label{fig:BD_mutual_avoidance_self_attraction} 
\end{figure}

\section{The steady states in the local limit}
\label{sec:grobner}
In the previous section, we found piecewise constant energy minimisers of the local limit of Equation (\ref{eq:system}).  These can attain only a discrete set of values. Here, we confirm this observation by showing that, on each subinterval where the solution is differentiable, it must be constant. 



For $N=2$ we prove that the image of any minimum energy solution must lie in a finite set. This proof works for any parameter values $D_i$ and $\gamma_{ij}$.  We were not, however, able to prove this result in full generality for arbitrary $N$.  Nonetheless, we do provide a method for constructing a proof for any particular set of parameter values, and put these ideas into practice in some example cases where $N=3$.  

\subsection{The general setup}

Let $K(x)=\delta(x)$, the Dirac delta function with mass concentrated at $x=0$.  Then in one spatial dimension Equation (\ref{eq:system}) becomes
\begin{linenomath*}\begin{align}
		\label{1Dmodel_loc}
		\frac{\partial u_i}{\partial t} &= D_i\frac{\partial^2 u_i}{\partial x^2} + \frac{\partial}{\partial x}\left(u_i \sum_{j=1}^N\gamma_{ij}\frac{\partial u_j}{\partial x}\right), \, i=1,\dots, N.
\end{align}\end{linenomath*}
Any local minimum energy solution to Equation (\ref{1Dmodel_loc}) is given by a set of functions $u_1 (x),\dots,u_N (x)$ that solve Equation \eqref{eq:steady_state} for each $i \in \{1,\dots,N\}$ with $K(x)=\delta(x)$. We therefore require that, on any subinterval where $u_i(x)\neq 0$,
\begin{linenomath*}\begin{align}
		\label{eq:steady_state_loc}
		0 = \frac{{\rm d} }{{\rm d} x}\left(D_i \text{ln}(u_i ) + \sum_{j=1}^N\gamma_{ij} u_j \right)=\frac{D_i}{u_i }\frac{{\rm d} u_i }{{\rm d} x} + \sum_{j=1}^N\gamma_{ij} \frac{{\rm d} u_j }{{\rm d} x},
\end{align}\end{linenomath*}
which implies that
\begin{linenomath*}\begin{align}
		\label{1Dmodel_loc_ss0}
		0 &= D_i\frac{{\rm d} u_i }{{\rm d} x} + u_i  \sum_{j=1}^N\gamma_{ij}\frac{{\rm d} u_j }{{\rm d} x}.
\end{align}\end{linenomath*}
Equation (\ref{1Dmodel_loc_ss0}) can be written in matrix form as
\begin{linenomath*}\begin{align}\label{1Dmodel_loc_ss0_mat}
		0 &= A_1\frac{{\rm d} {\bf u} }{{\rm d} x},  \\\nonumber			
		&\text{where } A_1:=\left(\begin{array}{cccc}D_1+\gamma_{11}u_1  & \gamma_{12}u_1  & \dots & \gamma_{1N}u_1  \\
			\gamma_{21}u_2  & D_2+\gamma_{22}u_2  & \dots & \gamma_{2N}u_2   \\
			\vdots & \vdots & \ddots & \vdots \\
			\gamma_{N1}u_N  &  \gamma_{N2}u_N  & \dots & D_N+\gamma_{NN}u_N   
		\end{array}\right),
\end{align}\end{linenomath*}
and ${\bf u} =(u_1 ,...,u_N )^T$. Equation (\ref{1Dmodel_loc_ss0_mat}) holds on each subinterval where $u_i(x)\neq 0$.  We wish to show that  differentiable solutions are necessarily constant.  Equation (\ref{1Dmodel_loc_ss0_mat}) only has a nontrivial solution if either $\det (A_1)=0$ or $\frac{\partial {\bf u} }{\partial x}=0$.  The latter means that ${\bf u} $ is constant, so we need to investigate the condition $\det (A_1)=0$.  


\subsection{The case $N=2$}
\label{sec:N2}

To make things simple, we begin by focusing on the case $N=2$.  We use the notation $A_1^{(2)}$ to mean the matrix $A_1$ (Equation  \eqref{1Dmodel_loc_ss0_mat}) for $N=2$, so that
\begin{linenomath*}\begin{align}
		\label{A1N2}
		A_1=A_1^{(2)}:=\left(\begin{array}{cc}D_1+\gamma_{11}u_1  & \gamma_{12}u_1 \\
			\gamma_{21}u_2  & D_2+\gamma_{22}u_2 
		\end{array}\right).
\end{align}\end{linenomath*}
The condition $\det (A_1^{(2)})=0$ then implies 
\begin{linenomath*}\begin{align}
		\label{detA1N2}
		(D_1+\gamma_{11}u_1 )(D_2+\gamma_{22}u_2 )-\gamma_{12}\gamma_{21}u_1 u_2 =0.
\end{align}\end{linenomath*}
If ${\bf u} $ is differentiable then we can differentiate Equation (\ref{detA1N2}) with respect to $x$, leading to the following
\begin{linenomath*}\begin{align}
		\label{detA1N2diff}
		[\gamma_{11}(D_2+\gamma_{22}u_2 )-\gamma_{12}\gamma_{21}u_2 ]\frac{{\rm d}u_1 }{{\rm d}x}+[\gamma_{22}(D_1+\gamma_{11}u_1 )-\gamma_{12}\gamma_{21}u_1 ]\frac{{\rm d}u_2 }{{\rm d}x}=0.
\end{align}\end{linenomath*}
Combining Equation (\ref{detA1N2diff}) with the first row of the vector equation $A_1^{(2)} \frac{{\rm d}{\bf u} }{{\rm d}x}=0$ gives
\begin{linenomath*}\begin{align}
		\label{A2N2}
		0 &= A_2^{(2)}\frac{{\rm d} {\bf u} }{{\rm d} x},\quad\mbox{where} \nonumber \\
		A_2^{(2)}&:=\left(\begin{array}{cc}\gamma_{11}(D_2+\gamma_{22}u_2 )-\gamma_{12}\gamma_{21}u_2  & \; \gamma_{22}(D_1+\gamma_{11}u_1 )-\gamma_{12}\gamma_{21}u_1 \\D_1+\gamma_{11}u_1  & \gamma_{12}u_1 
		\end{array}\right).
\end{align}\end{linenomath*}
Then $\left\{\det (A_1^{(2)})=0, \det (A_2^{(2)})=0\right\}$ is a system of two simultaneous equations in two unknowns.  These have at most three solutions, as we show in Appendix \ref{appendix:C}. 

The exact form of these solutions is rather cumbersome, so we omit writing them down explicitly.  However, it is instructive to give a simple example, which we do in the case $\gamma_{11}=\gamma_{22}=0$.  Here, there is a single solution to $\left\{\det (A_1^{(2)})=0, \det (A_2^{(2)})=0\right\}$ of the following form
\begin{linenomath*}\begin{align}
		\label{N2u1u2}
		u_1 =\frac{D_2}{\gamma_{21}},\quad 
		u_2 =\frac{D_1}{\gamma_{12}}.
\end{align}\end{linenomath*}
Regardless of whether or not we impose the condition $\gamma_{11}=\gamma_{22}=0$, the solution set of $(u_1 ,u_2 )$ is a finite set.  Therefore each differentiable part of a solution of Equation (\ref{1Dmodel_loc_ss0}) is constant.

\subsection{The case $N=3$}

{\noindent}We now show how to extend the arguments of Section \ref{sec:N2} to the $N=3$ case.  The expressions become too complicated in $N=3$ to give a complete analysis, so we instead give some examples to demonstrate how one can ascertain whether the or not image of ${\bf u} (x)$ is contained in a finite set.    Similar to the strategy for $N=2$, the aim is to construct a system of equations that constrain the possible solutions for ${\bf u} (x)$.  For $N=3$, this involves constructing three equations, which each take the form $\det (A_i^{(3)})=0$ for some matrix $A_i^{(3)}$ ($i \in \{1,2,3\}$), and showing that this set of simultaneous equations has a finite number of solutions. Whist for $N=2$, we were able to calculate the number of solutions exactly by solving polynomial equations, this is not possible for $N=3$ as the polynomials are usually of order 5 or more \citep{stewart2015galois}.  Instead, we use the theory of Gr\"obner bases to prove the solution set is finite.

\subsubsection{Example 1}
\label{sec:ex1}

{\noindent}For this example, we let $D_i=1$, $\gamma_{ii}=0$,  $\gamma_{12}=\gamma_{21}=\gamma_{23}=\gamma_{32}=2$, and $\gamma_{13}=\gamma_{31}=1$. Then 
\begin{linenomath*}\begin{align}
		\label{A13}
		A_1=A_1^{(3)}:=\left(\begin{array}{ccc}1 & 2u_1  &  u_1 \\
			2u_2  & 1 & 2u_2  \\
			u_3  & 2u_3  & 1 \\
		\end{array}\right).
\end{align}\end{linenomath*}
Since $\det (A_1^{(3)})=0$, we have
\begin{linenomath*}\begin{align}
		\label{detA13}
		0=1+8 u_1 u_2 u_3 -4u_1 u_2 -4u_2 u_3 -u_1 u_3 .
\end{align}\end{linenomath*}
Again, assuming ${\bf u} $ is differentiable, we can differentiate Equation (\ref{detA13}) with respect to $x$ leads to the following 
\begin{linenomath*}\begin{align}
		\label{detA13diff}
		0=\frac{{\rm d}u_1 }{{\rm d}x}(8 u_2 u_3 -4u_2 -u_3 )+\frac{{\rm d}u_2 }{{\rm d}x}(8 u_1 u_3 -4u_1 -4u_3 )+\frac{{\rm d}u_3 }{{\rm d}x}(8 u_1 u_2 -u_1 -4u_2 ).
\end{align}\end{linenomath*}
Combining Equation (\ref{detA13diff}) with the first two rows of $A_1^{(3)} \frac{{\rm d}{\bf u} }{{\rm d}x}=0$ gives
\begin{linenomath*}\begin{align}
		\label{A23}
		0 &= A_2^{(3)}\frac{{\rm d} {\bf u} }{{\rm d} x},\\\nonumber 
		& \text{where } A_2^{(3)}:=\left(\begin{array}{ccc}
			8 u_2 u_3 -4u_2 -u_3 & 8 u_1 u_3 -4u_1 -4u_3 &8 u_1 u_2 -u_1 -4u_2  \\
			1 & 2u_1  &  u_1 \\
			2u_2  & 1 & 2u_2  
		\end{array}\right).
\end{align}\end{linenomath*}
Once again, we have that $\det (A_2^{(3)})=0$, leading to the following polynomial equation
\begin{linenomath*}\begin{align}
		\label{detA23}
		0=&-u_1  -4 u_2  +20 u_1  u_2  -4 (u_1 )^2 u_2  - 32 (u_1 )^2 (u_2 )^2 + u_1  u_3  + 8 u_2  u_3  \nonumber \\
		& -36 u_1  u_2  u_3  + 16 (u_1 )^2 u_2  u_3  +32 u_1  (u_2 )^2 u_3.
\end{align}\end{linenomath*}
Differentiating Equation (\ref{detA23}) with respect to $x$ gives
\begin{linenomath*}\begin{align}
		\label{detA23diff1}
		0&=\frac{{\rm d}u_1 }{{\rm d}x}B_1(u_1 ,u_2 ,u_3 )+\frac{{\rm d}u_2 }{{\rm d}x}B_2(u_1 ,u_2 ,u_3 )+\frac{{\rm d}u_3 }{{\rm d}x}B_3(u_1 ,u_2 ,u_3 ),
\end{align}\end{linenomath*}
where
\begin{linenomath*}\begin{align}
		\label{detA23diff2}
		B_1(u_1 ,u_2 ,u_3 )=&-1 +20 u_2  -8 u_1  u_2  -64 u_1  (u_2 )^2 + u_3  -36 u_2  u_3  \nonumber
		\\&+32 u_1  u_2  u_3  +32 (u_2 )^2 u_3  \\
		B_2(u_1 ,u_2 ,u_3 )=&-4 + 20 u_1  -4 (u_1 )^2 - 64 (u_1 )^2 u_2  + 8 u_3  - 36 u_1  u_3  \nonumber
		\\& + 16 (u_1 )^2 u_3  + 64 u_1  u_2  u_3 \\
		B_3(u_1 ,u_2 ,u_3 )=&u_1  -36 u_1  u_2  + 16 (u_1 )^2 +32 u_1  (u_2 )^2.
\end{align}\end{linenomath*}
Combining Equation (\ref{detA23diff1}) with the first two rows of $A_2^{(3)} 
\frac{{\rm d}{\bf u} }{{\rm d}x}=0$ gives
\begin{linenomath*}\begin{align}
		\label{A33}
		0 &= A_3^{(3)}\frac{{\rm d} {\bf u} }{{\rm d} x},\quad\mbox{where} \nonumber \\
		A_3^{(3)}&:=\left(\begin{array}{ccc}
			B_1(u_1 ,u_2 ,u_3 ) & B_2(u_1 ,u_2 ,u_3 ) & B_3(u_1 ,u_2 ,u_3 ) \\
			8 u_2 u_3 -4u_2 -u_3 &8 u_1 u_3 -4u_1 -4u_3 &8 u_1 u_2 -u_1 -4u_2  \\
			1 & 2u_1  &  u_1 
		\end{array}\right).
\end{align}\end{linenomath*}
We now have a set of three polynomials 
\begin{linenomath*}
\begin{equation}S=\left\{\det (A_1^{(3)}),\det (A_2^{(3)}),\det (A_3^{(3)})\right\} 
\end{equation}
\end{linenomath*}
such that the image of ${\bf u} (x)$ must lie on the common zeros of this set.  In the $N=2$ case (Section \ref{sec:N2}), we had just two polynomials, both of which were cubics, thus it is possible to find formulae for the common zeros.  Here, however, we have a polynomial of degree six ($\det (A_3^{(3)})$). Since there is no general solution to a sixth degree polynomial \citep{stewart2015galois}, we cannot solve the system $\det (A_1^{(3)})=0,\det (A_2^{(3)})=0,\det (A_3^{(3)})=0$ directly.  

Instead, we use a classical result from algebraic geometry, which says that the number of common zeros of $S$ is finite iff for each $i \in \{1,2,3\}$, the Gr\"obner basis of the ideal $I(S)$ generated by $S$ contains a polynomial whose leading monomial is a power of $u_i$ \citep{adams1994introduction}.  Computation of the Gr\"obner basis of an ideal generated by a set of polynomials is an algorithmic procedure that is encoded into various mathematical packages, such as Mathematica \citep{wolfram1999mathematica} or Macauley2 \citep{eisenbud2013computations}.  

We use Mathematica to calculate the Gr\"obner basis of $I(S)$. The result is a set of five polynomials whose leading monomials are $ \beta_1 u_3^{19}$, $ \beta_2 u_2 u_3^2$, $ \beta_3 u_2^2 u_3$, $\beta_4 u_2^4$  and $\beta_5 u_1$, where $\beta_1,\dots,\beta_5$ are constants (some of which are of the order $10^{26}$ so we refrain from writing down their exact numerical values).  For each $i$, there is a polynomial in the Gr\"obner basis whose leading monomial is a power of $u_i$.  Therefore, the common zeros of $S$ are finite and the image of ${\bf u} (x)$ is contained in a finite set.  Since we have assumed ${\bf u} (x)$ is differentiable, it must also be constant.  

\subsubsection{Example 2}
\label{sec:ex2}

{\noindent}In the previous example, we were able to show that the image of ${\bf u} (x)$ is contained in a finite set by showing it lies on the intersection of three polynomials, which is the minimum number of polynomials required in the case $N=3$.  However, sometimes three polynomials is not enough.  Here, we detail an example which requires the construction of five polynomials to ensure the intersection of their zeros is a finite set.

Suppose $D_i=1$, $\gamma_{ii}=0$, and $\gamma_{ij}=2$ for all $i,j \in \{1,2,3\}$ where $i \neq j$.  Then 
\begin{linenomath*}\begin{align}
		\label{A1N3b}
		A_1=A_1^{(3)}:=\left(\begin{array}{ccc}1 & 2u_1  &  2u_1 \\
			2u_2  & 1 & 2u_2  \\
			2u_3  & 2u_3  & 1 \\
		\end{array}\right).
\end{align}\end{linenomath*}
Since $\det (A_1^{(3)})=0$, we have
\begin{linenomath*}\begin{align}
		\label{detA1N3}
		0=1+16 u_1 u_2 u_3 -4u_1 u_2 -4u_1 u_3 -4u_2 u_3 .
\end{align}\end{linenomath*}
Differentiating Equation (\ref{detA1N3}) with respect to $x$ leads to the following 
\begin{linenomath*}\begin{align}
		\label{detA1N3diff}
		0=\frac{{\rm d}u_1 }{{\rm d}x}(4 u_2 u_3 -u_2 -u_3 )+\frac{{\rm d}u_2 }{{\rm d}x}(4 u_1 u_3 -u_1 -u_3 )+\frac{{\rm d}u_3 }{{\rm d}x}(4 u_1 u_2 -u_1 -u_2 )
\end{align}\end{linenomath*}
Combining Equation (\ref{detA1N3diff}) with the first two rows of $A_1^{(3)} \frac{{\rm d}{\bf u} }{{\rm d}x}=0$ gives
\begin{linenomath*}\begin{align}
		\label{A2N3}
		0 &= A_2^{(3)}\frac{{\rm d} {\bf u} }{{\rm d} x},\quad\mbox{where} \nonumber \\
		A_2^{(3)}&:=\left(\begin{array}{ccc}
			4 u_2 u_3 -u_2 -u_3 &4 u_1 u_3 -u_1 -u_3 &4 u_1 u_2 -u_1 -u_2  \\
			1 & 2u_1  &  2u_1 \\
			2u_2  & 1 & 2u_2  
		\end{array}\right).
\end{align}\end{linenomath*}
Once again, we have that $\det (A_2^{(3)})=0$, leading to the following polynomial equation
\begin{linenomath*}\begin{align}
		\label{detA2N3}
		0=(4 u_2 u_3 -u_2 -u_3 )(4u_1 u_2 -2u_1 )&+(4u_1 u_3 -u_1 -u_3 )(4u_1 u_2 -2u_2 )\nonumber \\
		&+(4 u_1 u_2 -u_1 -u_2 )(1-4u_1 u_2 ).
\end{align}\end{linenomath*}
Differentiating Equation (\ref{detA2N3}) with respect to $x$ gives
\begin{linenomath*}\begin{align}
		\label{detA2N3diff1}
		0&=\frac{{\rm d}u_1 }{{\rm d}x}B_1(u_1 ,u_2 ,u_3 )+\frac{{\rm d}u_2 }{{\rm d}x}B_2(u_1 ,u_2 ,u_3 )+\frac{{\rm d}u_3 }{{\rm d}x}B_3(u_1 ,u_2 ,u_3 ),
\end{align}\end{linenomath*}
where
\begin{linenomath*}\begin{align}
		\label{detA2N3diff2}
		B_1(u_1 ,u_2 ,u_3 )&=(4 u_2  u_3 -u_2  -u_3 )(4 u_2 -2)+(4 u_3 -1)(4u_1 u_2 -2u_2 )\nonumber \\
		& \qquad+(4 u_1 u_3 -u_1 -u_3 )4u_2 +(4u_2 -1)(1-4u_1 u_2 )\nonumber  \\
		& \qquad-(4u_1 u_2 -u_1 -u_2 )4u_2 \\
		B_2(u_1 ,u_2 ,u_3 )&=(4u_3 -1)(4u_1 u_2 -2u_1 )+(4u_2 u_3 -u_2 -u_3 )4u_1 \nonumber \\
		& \qquad + (4 u_1  u_3 -u_1  -u_3 )(4 u_1 -2)+(4u_1 -1)(1-4u_1 u_2 )\nonumber\\
		& \qquad-(4u_1 u_2 -u_1 -u_2 )4u_1   \\
		B_3(u_1 ,u_2 ,u_3 )&=(4u_2 -1)(4u_1 u_2 -2u_1 )+(4u_1 -1)(4u_1 u_2 -2u_2 )
\end{align}\end{linenomath*}
Combining Equation (\ref{detA2N3diff1}) with the first two rows of $A_2^{(3)} 
\frac{{\rm d}{\bf u} }{{\rm d}x}=0$ gives
\begin{linenomath*}\begin{align}
		\label{A3N3}
		0 &= A_3^{(3)}\frac{{\rm d} {\bf u} }{{\rm d} x},\quad\mbox{where} \nonumber \\
		A_3^{(3)}&:=\left(\begin{array}{ccc}
			B_1(u_1 ,u_2 ,u_3 ) & B_2(u_1 ,u_2 ,u_3 ) & B_3(u_1 ,u_2 ,u_3 ) \\
			4 u_2 u_3 -u_2 -u_3 &4 u_1 u_3 -u_1 -u_3 &4 u_1 u_2 -u_1 -u_2  \\
			1 & 2u_1  &  2u_1 
		\end{array}\right).
\end{align}\end{linenomath*}
We now have a set of three polynomials $S=\left\{\det (A_1^{(3)}),\det (A_2^{(3)}),\det (A_3^{(3)})\right\}$, such that the image of ${\bf u} (x)$ must lie on the common zeros of this set. 
The Gr\"obner basis of $I(S)$ contains eight polynomials whose leading terms are $\beta_1 u_2 u_3^9$, $\beta_2 u_2 u_3^8$, $\beta_3 u_2 u_3^8$, $\beta_4 u_2 u_3^8$, $\beta_5 u_2 u_3^8$, $\beta_6 u_2 u_3^8$, $\beta_7 u_2 u_3^8$, $\beta_8 u_2 u_3^8$ for constants $\beta_1,\dots,\beta_8$. 
Here, the Gr\"obner basis of $I(S)$ does not contain a polynomial a with leading monomial that is a power of $u_i $ for any $i=1,2,3$, so the common zeros of $S$ do not form a finite set. Therefore we need to search for further polynomials on which the solution lies, to see if we can constrain the solutions into a finite set.	

To this end, we combine Equation (\ref{detA1N3diff}) with the first and the third row of $A_1^{(3)} \frac{{\rm d}{\bf u} }{{\rm d}x}=0$ to give
\begin{linenomath*}\begin{align}
		\label{A22N3}
		0 &= A_{22}^{(3)}\frac{{\rm d} {\bf u} }{{\rm d} x},\quad\mbox{where} \nonumber \\
		A_{22}^{(3)}&:=\left(\begin{array}{ccc}
			4 u_2 u_3 -u_2 -u_3 &4 u_1 u_3 -u_1 -u_3 &4 u_1 u_2 -u_1 -u_2  \\
			1 & 2u_1  &  2u_1 \\
			2u_3  & 2u_3   & 1
		\end{array}\right).
\end{align}\end{linenomath*}
Since $\det (A_{22}^{(3)})=0$, we have
\begin{linenomath*}\begin{align}
		\label{detA22N3}
		0=u_1 -2 u_1 u_2 + u_3 - 8 u_1 u_3 - 2 u_2 u_3 + 24 u_1 u_2 u_3 -16 u_1^2 u_2 u_3-16 u_1^2 u_3^2-16 u_1 u_2 u_3^2.
\end{align}\end{linenomath*}
Differentiating Equation (\ref{detA22N3}) with respect to $x$ gives
\begin{linenomath*}\begin{align}
		\label{detA22N3diff1}
		0&=\frac{{\rm d}u_1 }{{\rm d}x}B_{12}(u_1 ,u_2 ,u_3 )+\frac{{\rm d}u_2 }{{\rm d}x}B_{22}(u_1 ,u_2 ,u_3 )+\frac{{\rm d}u_3 }{{\rm d}x}B_{32}(u_1 ,u_2 ,u_3 ),
\end{align}\end{linenomath*}
where
\begin{linenomath*}\begin{align}
		\label{detA22N3diff2}
		B_{12}(u_1 ,u_2 ,u_3 )&=1 - 2 u_2 - 8 u_3 + 24 u_2 u_3 - 32 u_1 u_2 u_3 +32 u_1 u_3^2 -16 u_2 u_3^2,\\
		B_{22}(u_1 ,u_2 ,u_3 )&=-2 u_2 -2 u_3 +24 u_1 u_3 -16u_1^2 u_3 -16 u_1 u_3^2, \\
		B_{32}(u_1 ,u_2 ,u_3 )&=1 - 2 u_2 - 8 u_1 + 24 u_1 u_2 - 32 u_1 u_2 u_3 + 32 u_1^2 u_3 -16 u_1^2 u_2.
\end{align}\end{linenomath*}
Combining Equation (\ref{detA22N3diff1}) with the second and third row of $A_{22}^{(3)} 
\frac{{\rm d}{\bf u} }{{\rm d}x}=0$ gives
\begin{linenomath*}\begin{align}
		\label{A32N3}
		0 &= A_{32}^{(3)}\frac{{\rm d} {\bf u} }{{\rm d} x},\quad\mbox{where} \nonumber \\
		A_{32}^{(3)}&:=\left(\begin{array}{ccc}
			B_{12}(u_1 ,u_2 ,u_3 ) & B_{22}(u_1 ,u_2 ,u_3 ) & B_{32}(u_1 ,u_2 ,u_3 ) \\
			1 & 2u_1  &  2u_1  \\
			2u_3  &  2u_3  & 1
		\end{array}\right).
\end{align}\end{linenomath*}
We now have a set of five polynomials $S=\left\{\det (A_1^{(3)}),\det (A_2^{(3)}),\det (A_3^{(3)}),\det (A_{22}^{(3)}),\det (A_{32}^{(3)})\right\}$, such that the image of ${\bf u} (x)$ must lie on the common zeros of this set.  The Gr\"obner basis of $I(S)$ consists of seven polynomials whose leading monomials are $32768 u_3^9$, $12 u_2 u_3^2$, $ 6 u_2^2 u_3$, $ 96 u_2^3$, $ -18 u_1$, $18 u_1 u_2$, $12 u_1^2$. Since, for each $i \in \{1,2,3\}$, this set contains a power of $u_i$, the common zeros of $S$ are finite, and therefore the image of ${\bf u} (x)$ is contained in a finite set.  Hence if ${\bf u} (x)$ is differentiable, it must be constant.


\section{Discussion}\label{sec13}



A central aim of mathematical biology is to predict emergent features of biological systems, using dynamical systems models.  Stable steady states provide an important class of emergent features, so identification of these is a key task of mathematical biology.  However, for nonlinear PDEs, this is not usually an easy task \citep{Robinson2003}.  Indeed, often this is replaced by the more tractable task of examining a system's behaviour close to the constant steady state, which enables linear or weakly nonlinear approximations.  But it is the behaviour far away from the constant solution that is interesting biologically, as that is where the patterns exist that we perceive in biological systems.

Here, we have detailed a novel method to help find  local minimum energy states, which are Lyapunov stable, in a system of nonlocal advection-diffusion equations for modelling $N$ species (or groups) of mobile organisms, each of which move in response to the presence of others.  Our study system is closely related to (and often directly generalises) a wide variety of previous models, including those for cell aggregation \citep{carrillo2018aggregation} and sorting \citep{burger2018sorting}, animal territoriality \citep{potts2016memory} and home ranges \citep{briscoe2002home}, the co-movements of predators and prey \citep{di2016nonlocal}, and the spatial arrangement of human criminal gangs \citep{alsenafi2018convection}.  Therefore our results have wide applicability across various areas of the biological sciences.

Whilst analytic determination of stable steady states in PDEs remains a difficult task in general, numerical analysis always leaves the question open of whether one has found all possible steady states or whether there are more that the researcher has simply not stumbled upon.  To help guide numerical investigations, we have constructed a method, combining heuristic and analytic features, that gives clues as to where stable steady states might be found in multi-species nonlocal advection-diffusion systems.  We have demonstrated in a few examples that numerical investigations agree with the predictions of our method.  Whilst our method does not give an analytic solution, it should be a valuable tool for finding stable steady states in biological models that can be modelled by nonlocal advection-diffusion systems.

Our method relies on constructing an energy functional for the PDE system.  We were only able to do this in the case $\gamma_{ij}=\gamma_{ji}$ for all $i,j \in \{1,\dots,N\}$ and assuming that the kernel $K$ is identical for all species.  These constraints mean that each pair of species (or populations or groups) respond to one another in a symmetric fashion, either mutually avoiding or mutually attracting with identical strengths of avoidance or attraction, respectively.  This generalises a recent result of \citet{ellefsen2021equilibrium} who construct an energy functional for the case where $\gamma_{ij}=1$ for all $i,j \in \{1,\dots,N\}$.  We conjecture that this energy functional could be used to prove that the attractor of our study system is an unstable manifold of fixed points.  However, we were unable to prove this here, so encourage readers to take on this challenge.

Whilst it may be possible to construct energy functionals in some example situations where $\gamma_{ij}\neq\gamma_{ji}$ for some $i,j$, or where  the kernel is not identical for all species (we leave this as an open question), we expect that it is not possible in general,
since there are situations where the numerical analysis suggests the attractors do not consist of stable steady states, but patterns that fluctuate in perpetuity \citep{PL19}.  Perhaps the simplest situation where this has been observed is for $N=2$, $\gamma_{11},\gamma_{22}<0$, and $\gamma_{12}<0<\gamma_{21}$ \citep{HPLG21}, whereby both populations aggregate and one `chases' the other across the terrain without either ever settling to a fixed location. 
Furthermore, to keep our analysis as simple as possible, we only applied the techniques of Section \ref{sec:struct} to some concrete examples in $n=1$ spatial dimension. Nonetheless, there is no {\it a priori} reason why the techniques in Section \ref{sec:struct} could not be extended to higher dimensions in future.

Whilst our method is designed for application to models of {\it nonlocal} advection, for which there are existence and regularity results \citep{HPLG21}, it works by examining the {\it local} limit of stable solutions.  The reason for this is that these solutions are piecewise constant, so we can constrain our search for the minimum energy, enabling minimisers to be found analytically.  The disadvantage is that the local limit of stable solutions is not itself the steady state solution of a well-posed system of PDEs: in the local limit, Equation (\ref{eq:system}) becomes ill-posed.  More precisely, it is unstable to arbitrarily high wavenumbers whenever the pattern formation matrix has eigenvalues with positive real part.  Nonetheless, we have shown that the local limit of minimum energy solutions to the nonlocal problem is a useful object to study, even if it may not itself be the steady state solution of a system of PDEs.  

It would be cleaner, however, if we were able to develop theory that did not require taking this local limit.  For $N=1$, \cite{potts2021stable} developed techniques that are analogous to the ones proposed here but in discrete space.  In this case, the actual stable steady states of the discrete space system become amenable to analysis via an energy functional approach similar to the one proposed here.  However, generalisations of this technique to $N>1$ do not appear to be trivial from our initial investigations.  

Another possible way forward is to use perturbation analysis, starting with the minimum energy solutions from the local limit, studied here, and perturbing them to give solutions to the full nonlocal system.  One could then minimise the energy across this class of perturbed solutions (which would no longer be piecewise constant) to find stable steady states of the nonlocal system in Equation (\ref{eq:system}).  This is quite a nontrivial extension of the present methods, which we hope to pursue in future work.  One possible avenue might be to use a kernel that allows the non-local model to be transformed into a higher-order local model \citep{bennett2019long, ellefsen2021equilibrium}.

Figures \ref{fig:BifurcationDiagramHysteresis2}, \ref{fig:MutualAttraction}, \ref{fig:BD}, and  \ref{fig:BD_mutual_avoidance_self_attraction} show numerical bifurcation analysis of our system in certain examples.  This naturally leads to questions about the nature of these bifurcations.  In particular, the discontinuity in amplitude that occurs as the constant steady state loses stability is something that is also seen with subcritical pitchfork bifurcations.  In this case, the stable branches may be joined to one another by an unstable branch, or some more complicated structure. It would be valuable to investigate analytically whether this is the case. Standard tools include weakly non-linear analysis and Crandall-Rabinowitz bifurcation theory, both of which have been used successfully for nonlocal advection-diffusion equations \citep{buttenschon2021non, eftimie2009weakly}.


The system we study assumes that species advect in response to the population density of other species.  However, it is agnostic as to the precise mechanisms underlying this advection.  Previous studies show that Equation (\ref{eq:system}) can be framed as a quasi-equilibrium limit of various biologically-relevant processes, such as scent marking or memory \citep{potts2016memory, potts2016territorial, PL19}.  This quasi-equilibrium assumption says, in effect,  that the scent marks or memory map stabilise quickly compared to the probability density of animal locations. However, it would be valuable to examine the extent to which these processes might affect the emergent patterns away from this quasi-equilibrium limit.  Along similar lines, it would also be valuable to examine the extent to which our results translate to the situation where we model each individual as a separate entity, as in an individual based model (IBM), rather than using a population density function, which is a continuum approximation of an IBM.  We have recently begun developing tools for translating PDE analysis to the situation of individual based models, which could be useful for such analysis \citep{potts2022beyond}.

In summary, we have developed novel methods for finding nontrivial steady states in a class of nonlinear, nonlocal PDEs with a range of biological applications.  As well as revealing complex multi-stable structures in examples of these systems, our study opens the door to various questions regarding the bifurcation structure, the effect of nonlocality, and the structure of the attractor.  We believe these will lead to yet more significant, but highly fruitful, future work.

\section*{Declarations}
\bmhead{Competing interests} 
T Hillen and MA Lewis are Editors-in-Chief of the {\em Journal of Mathematical Biology}.  Other than this, the authors have no competing interests to declare that are relevant to the content of this article. 

\bmhead{Author contributions}
JR Potts led the conception and design of the study, with input from MA Lewis and T Hillen.  V Giunta led the mathematical and numerical analysis, with input from all authors.  The first draft of the manuscript was written by V Giunta and all authors commented on previous versions of the manuscript. All authors read and approved the final manuscript.

\backmatter

\bmhead{Acknowledgments}

JRP and VG  acknowledge  support of Engineering  and  Physical  Sciences  Research  Council (EPSRC) grant EP/V002988/1 awarded to JRP. VG is also grateful for support from the National Group of Mathematical Physics (GNFM-INdAM). TH is grateful for support from the Natural Science and Engineering Council of Canada (NSERC) Discovery Grant RGPIN-2017-04158. MAL  gratefully acknowledges support from NSERC Discovery Grant RGPIN-2018-05210 and the Canada Research Chair program.

%
%
%
%

\begin{appendices}

		\section{Calculations for Figure \ref{fig:InstabilityRegionsComplete}}
		\label{appendix:A}
		
		{\noindent}Here, we give details of the calculations performed to produce the plots in Figure \ref{fig:InstabilityRegionsComplete} from Section \ref{sec:gammaneq0}.  The analysis is similar to that in Sections \ref{sub:1} and \ref{sub:2}, but unlike Sections \ref{sub:1} and \ref{sub:2} we drop the assumption that $\gamma_{11}=\gamma_{22}=0$ and we keep the assumption $\gamma_{12}=\gamma_{21}$. 
				
		We will look for the local minimizers of the following energy functional, where $K=\delta$,
		\begin{linenomath*}\begin{equation}\label{eq:energyA}
				E[u_1, u_2]=\int_{\mathbb{T}} \sum_{i=1}^{2} u_i \left(D_i \text{ln}(u_i)+\frac{1}{2}\sum_{j=1}^{2}\gamma_{i j} u_j\right) dx
		\end{equation}\end{linenomath*}		
		in the class of piece-wise constant functions defined as 
			\begin{linenomath*}\begin{align}
				\label{eq:uiA}
				u_i (x)=\begin{cases}
					u_{i}^c, &\mbox{for $x \in S_{i}$}, \\
					0, & \mbox{for $x \in [0,L]$\textbackslash$S_{i}$},
				\end{cases}
		\end{align}\end{linenomath*}
		where $ u_{i}^c \in \mathbb{R}^{+} $ and $S_{i}$ are subsets of $[0,L]$, for $i\in\{1,2\}$. 
			
Recall that, by Equation (\ref{eq:int_cond}), in Equation \eqref{eq:uiA} we require the following constraint 
		\begin{linenomath*}\begin{equation}\label{uicA}
				u_{i}^c\lvert S_{i}\lvert  =p_i, \text{ for }i=1,2.
		\end{equation}\end{linenomath*}
		Placing Equation \eqref{eq:uiA} into Equation \eqref{eq:energyA} gives 
		\begin{linenomath*}\begin{align}\nonumber
				E[u_1 ,u_2 ]=&\int_0^L  \left[\sum_{i=1}^{2}\left( D_i u_{i}  \text{ln}(u_{i} )+\frac{1}{2}\gamma_{ii}u_{i} ^2\right)+\gamma_{12} u_{1}  u_{2} \right]  dx \nonumber \\\nonumber
				=&\sum_{i=1}^{2} \lvert S_{i}\lvert \left(   D_i u_{i}^c \text{ln}(u_{i}^c)  +\frac{1}{2}\gamma_{ii}  (u_{i}^c)^2\right)+ \gamma_{12} u_{1}^c u_{2}^c \lvert S_{1} \cap S_{2}\lvert \\ \label{eq:energylocA}
				=&	\sum_{i=1}^{2}  p_{i}\left(D_i \text{ln}(u_{i}^c)  +\frac{1}{2}\gamma_{ii} u_{i}^c\right) +  \gamma_{12} u_{1}^c u_{2}^c \lvert S_{1} \cap S_{2}\lvert ,
		\end{align}\end{linenomath*}	
		where the first equality uses $\gamma_{12}=\gamma_{21}$ and the third equality uses Equation \eqref{uicA}.	
		
		Since the general analysis of this case is not straightforward, we instead set $ \gamma_{1 1}=\gamma_{2 2} $ and fix the other parameter values as $ p_1=p_2=D_1=D_2=L=1 $. Therefore Equation \eqref{eq:energylocA} becomes
		\begin{linenomath*}\begin{align}\label{eq:energylocA1}
				E[u_1 ,u_2 ]=& \text{ln}(u_{1}^c)+ \text{ln}(u_{2}^c)  +\frac{1}{2}\gamma_{11} (u_{1}^c+u_{2}^c) +  \gamma_{12} u_{1}^c u_{2}^c \lvert S_{1} \cap S_{2}\lvert.
		\end{align}\end{linenomath*}
		In the following, we will look for the minimizers of Equation \eqref{eq:energylocA1} and examine different cases demarcated by the signs of $\gamma_{11}$ and $\gamma_{12}$.
		
		\subsection{Self avoidance ($\mathbf{\gamma_{11}>0}$) and mutual avoidance ($\mathbf{\gamma_{12}>0}$) }\label{appendix:A1}
		
{\noindent}Since $ \gamma_{12}>0 $, in Equation \eqref{eq:energylocA1} if we keep $ \lvert S_{1}\lvert  $ and $ \lvert S_{2}\lvert $ fixed whilst lowering $ \lvert S_{1} \cap S_{2}\lvert  $ then the energy decreases. Thus, whenever $ \lvert S_{1}\rvert+\lvert S_{2}\rvert  \leq L =1 $ we can choose disjoint sets $S_{1}$ and $S_{2}$ that will correspond to lower energy solutions than any pair of non-disjoint sets of equal measure.  Furthermore, if $ \lvert S_{1}\rvert+\lvert S_{2}\rvert  > 1 $, we can construct sets $S_{1}$ and $S_{2}$, such that $\lvert S_1\cap S_2\rvert=\lvert S_1\rvert+\lvert S_2\rvert-1$ and these will correspond to lower energy solutions than any other pair of sets of equal measure. Therefore, when $ \lvert S_{1}\rvert+\lvert S_{2}\rvert  \leq 1 $, we will assume that $S_{1} \cap S_{2}=\emptyset$, and when $ \lvert S_{1}\rvert+\lvert S_{2}\rvert  > 1 $, we will assume that $\lvert S_1\cap S_2\rvert=\lvert S_1\rvert+\lvert S_2\rvert-1$ (as in Section \ref{sec:analytic1}). 

To search for the local minimizers of the energy in Equation \eqref{eq:energylocA1}, we then define
\begin{linenomath*}\begin{align} 
		\mathcal{E}(u_1^c,u_2^c)=
		\begin{cases}
			\text{ln}(u_{1}^c)+ \text{ln}(u_{2}^c)  +\frac{1}{2}\gamma_{11} (u_{1}^c+u_{2}^c), & \text{ if }  \lvert S_1\rvert  + \lvert S_2\rvert  \leq 1, \\ \\
		\text{ln}(u_{1}^c)+ \text{ln}(u_{2}^c)  +\frac{1}{2}\gamma_{11} (u_{1}^c+u_{2}^c) \\\quad + \gamma_{12} u_{1}^c u_{2}^c \lvert S_{1} \cap S_{2}\rvert, & \text{ if } \lvert S_1\rvert  + \lvert S_2\rvert  > 1.
		\end{cases}
		\label{eq:mathcalEA}
\end{align}\end{linenomath*}	
To constrain our search, notice that Equation \eqref{uicA}, $ p_i=1 $ and $ \lvert S_i\rvert  \leq L =1 $ imply that
\begin{linenomath*}\begin{align}
		\label{ssssicA}
		u_i^c\geq 1, \text{ for } i=1,2.
\end{align}\end{linenomath*}
We analyse $\mathcal{E}(u_1^c,u_2^c)$ (Equation \eqref{eq:mathcalEA}) under the constraint in Equation \eqref{ssssicA}, first in the region where $ \lvert S_1\rvert  + \lvert S_2\rvert  \leq 1 $ and then in the region where $ \lvert S_1\rvert  + \lvert S_2\rvert  > 1 $. By combining these results we will have a complete picture of the local minima of $\mathcal{E}(u_1^c,u_2^c)$. 
		
Note that by Equation \eqref{uicA}, the case $ \lvert S_1\rvert  + \lvert S_2\rvert  \leq 1$ is equivalent to
\begin{linenomath*}\begin{equation}\label{eq:constrA}
		\frac{1}{u_1^c}+\frac{1}{u_2^c}\leq 1.
\end{equation}\end{linenomath*}
By analysing the partial derivatives of $\mathcal{E}(u_1^c,u_2^c)$ in the region of the $(u_1^c,u_2^c)$-plane defined by Equation \eqref{eq:constrA}, one can check that there are no local minima in this region.  Furthermore, $\mathcal{E}(u_1^c,u_2^c)\rightarrow \infty$ as either $u_1^c \rightarrow \infty$ or $u_2^c\rightarrow \infty$.  Therefore any minima in this region must lie on the boundary, $ {1}/{u_1^c}+{1}/{u_2^c}= 1 $ (solid black line in Figure \ref{fig:Diagram}). Analysis of the partial derivative of $\mathcal{E}(u_1^c,u_2^c)$ on this boundary shows that $ \mathcal{E}(u_1^c,u_2^c)$ has a unique minimum point, given by
\begin{linenomath*}\begin{equation}\label{eq:minimumA}
		\mathcal{M}_S=(u_{1S}^c, u_{2S}^c):=\left(2,2\right).
\end{equation}\end{linenomath*}
This is also a local minimum of the region defined by Equation \eqref{eq:constrA}. This can be shown by performing a Taylor expansion of $ \mathcal{E}(u_1^c,u_2^c) $ about the point $\mathcal{M}_S$. Since the slope of the line tangent to the curve $ {1}/{u_1^c}+{1}/{u_2^c}= 1 $ in $\mathcal{M}_S$ is $-1$, we choose two constant, $\epsilon$ and $\delta$, such that $\epsilon+\delta\geq0$ and the Taylor expansion gives
\begin{linenomath*}\begin{align}\nonumber
		\mathcal{E}(2+\epsilon,2+\delta)& \approx    \mathcal{E}(2,2)+ \partial_{u_1^c} \mathcal{E}(2,2) \epsilon + \partial_{u_2^c} \mathcal{E}(2,2)\delta\\\nonumber
		&= \mathcal{E}(2,2)+ \frac{1}{2}(1+\gamma_{1 1})(\epsilon + \delta) \\\nonumber
		& \geq \mathcal{E}(2,2),
\end{align}\end{linenomath*}
where the inequality uses $ \gamma_{1 1}>0 $, $\epsilon+\delta\geq0$. 

However, since the point $ \mathcal{M}_S $ lies on the boundary curve $ \lvert S_1\rvert  + \lvert S_2\rvert = 1 $, we do not yet know whether it is a minimum for the whole admissible region defined by Equation \eqref{ssssicA} (white region in Figure \ref{fig:Diagram}).  To this end, we examine whether $ \mathcal{M}_S $ is a minimum of $ \mathcal{E}(u_1^c,u_2^c)$ (Equation \eqref{eq:mathcalEA}) in the region where $ \lvert S_1\lvert  + \lvert S_2\rvert  > 1 $. By Equation \eqref{ssssicA}, the condition $ \lvert S_1\lvert  + \lvert S_2\rvert  > 1 $ is equivalent to $ {1}/{u_1^c}+{1}/{u_2^c}> 1 $.  Therefore we have the following constraints
\begin{linenomath*}\begin{align}
		\frac{1}{u_1^c}+\frac{1}{u_2^c}&> 1, \nonumber \\
		u_i^c&\geq 1, \text{ for } i=1,2. 
		\label{eq:constr2A}
\end{align}\end{linenomath*}
Since $\lvert S_1 \cap S_2\rvert  = \lvert S_1\rvert  + \lvert S_2\rvert  - 1 $, when $\lvert S_1\rvert  + \lvert S_2\rvert  > 1$ the function $\mathcal{E}({u}_{1}^c,{u}_{2}^c) $ (Equation \eqref{eq:mathcalEA}) can be rewritten as
\begin{linenomath*}\begin{align} \nonumber
		\mathcal{E}(u_1^c,u_2^c) & = \text{ln}(u_{1}^c)+ \text{ln}(u_{2}^c)  +\frac{1}{2}\gamma_{11} (u_{1}^c+u_{2}^c)+  \gamma_{12}  u_1^c u_2^c \lvert S_1 \cap S_2\lvert \\ \nonumber
		&=	\text{ln}(u_{1}^c)+ \text{ln}(u_{2}^c)  +\frac{1}{2}\gamma_{11} (u_{1}^c+u_{2}^c) +  \gamma_{12}  u_1^c u_2^c (\lvert S_1\lvert  + \lvert S_2\lvert  -1),
		\\ 
		&=	\text{ln}(u_{1}^c)+ \text{ln}(u_{2}^c)  +\frac{1}{2}\gamma_{11} (u_{1}^c+u_{2}^c)+ \gamma_{12}  u_1^c u_2^c \left(\frac{1}{u_1} + \frac{1}{u_2} -1\right) ,
		\label{eq:mathcalEsA}
\end{align}\end{linenomath*}
where the third equality uses $\lvert S_i\lvert =\frac{1}{u_i^c}$.

To verify whether $ \mathcal{M}_S $ is also a minimum on the part of the domain given by Equation \eqref{eq:constr2A}, we perform a Taylor expansion of $\mathcal{E}(u_1^c,u_2^c)$ in a neighbourhood of $ \mathcal{M}_S $ within the region $1/u_1^c+1/u_2^c<1$. Since the slope of  the tangent line to the curve $1/u_1^c+1/u_2^c  = 1$ at the point  $ \mathcal{M}_S $  is $ -1 $, we choose two arbitrary constants, $\epsilon$ and $\delta$, such that $ \epsilon+ \delta\leq 0$.  Then Taylor expansion of $\mathcal{E}(u_1^c,u_2^c)$ is
\begin{linenomath*}\begin{align}\nonumber
		\mathcal{E}(2+\epsilon,2+\delta)& \approx    \mathcal{E}(2,2)+ \partial_{u_1^c} \mathcal{E}(2,2) \epsilon + \partial_{u_2^c} \mathcal{E}(2,2)\delta \\\nonumber
		& = \mathcal{E}(2,2)+\frac{1}{2}(1+ \gamma_{11}-2\gamma_{12})(\epsilon+\delta)\\ 
		& \geq \mathcal{E}(2,2),
\end{align}\end{linenomath*}
if $ \gamma_{11} < 2\gamma_{12}-1 $, where the inequality uses $ \epsilon+ \delta\leq 0$. 		

Next, we look for any other minima in the region defined by Equation \eqref{eq:constr2A}.  By analysing first partial derivatives, one can show that there are no local minima of $ \mathcal{E}(u_1^c,u_2^c) $ (Equation \eqref{eq:mathcalEsA}) in the interior of this region.  Therefore any local minima must occur on the boundaries.  On the part of the boundary given by $ u_i^c =1 $, for $ i=1,2 $, there is a unique minimum at 
\begin{linenomath*}\begin{equation}\label{eq:minimum2A}
		\mathcal{M}_H=(u_{1H}^c, u_{2H}^c):=\left(1,1\right).
\end{equation}\end{linenomath*}
This is also a local minimum of the region defined by Equation \eqref{eq:constr2A}.  This can be shown by performing a Taylor expansion of $ \mathcal{E}(u_1^c,u_2^c) $ about the point $\mathcal{M}_H$, to give
\begin{linenomath*}\begin{align}\nonumber
		\mathcal{E}(1+\epsilon,1+\delta)& \approx    \mathcal{E}(1,1)+ \partial_{u_1^c} \mathcal{E}(1,1) \epsilon + \partial_{u_2^c} \mathcal{E}(1,1)\delta\\\nonumber
				&= \mathcal{E}(1,1)+ \left(1+\frac{1}{2}\gamma_{11}\right) (\epsilon +\delta) \\ \nonumber
				& \geq \mathcal{E}(1,1),
\end{align}\end{linenomath*}
where the inequality uses $ \gamma_{1 1}>0 $, $\epsilon\geq0$ and $\delta\geq0$. Here, $\epsilon$ and $\delta$ are chosen to be non-negative so that we remain in the $ u_i^c \geq 1 $ region (Figure \ref{fig:Diagram}).
Therefore, if  $ \gamma_{11} > 2\gamma_{12}-1 $, $ \mathcal{E}(u_1^c,u_2^c) $ (Equation \eqref{eq:energylocA1}) has a unique minimum, given by $\mathcal{M}_H$.  Whilst if $ 0<\gamma_{11} < 2\gamma_{12}-1 $, then $ \mathcal{E}(u_1^c,u_2^c) $ has two local minima, given by $\mathcal{M}_H$ and $\mathcal{M}_S$.
		
Finally, we write down the functions $u_i(x)$ (Equation \eqref{eq:uiA}) which locally minimize the energy $E[u_1 ,u_2 ]$ (Equation \eqref{eq:energylocA}). If $(u_1^c,u_2^c)=\mathcal{M}_H$ then $u_1 (x)=u_2 (x)=1$, the homogeneous steady state, which we denote by $\mathcal{S}_H$.  If $(u_1^c,u_2^c)=\mathcal{M}_S$ then 
\begin{linenomath*}\begin{align}\label{eq:St2}
		u_i (x)=\begin{cases}
			2, &\mbox{for $x \in S_i$} \\
			0, & \mbox{for $x \in [0,1]$\textbackslash$S_i$},
		\end{cases}
\end{align}\end{linenomath*}
with $ \lvert S_i\rvert  =1/2 $, for $ i=1,2 $, and $ \lvert S_1  \cap S_2\rvert =0 $, denoted by $\mathcal{S}_S^{2,2}$.
		
In conclusion, if  $ \gamma_{11} > 2\gamma_{12}-1 $, the energy $E(u_1^c,u_2^c)$ (Equation \eqref{eq:energylocA}) has a unique minimum, given by $\mathcal{S}_H$.  However, if $ 0<\gamma_{11} < 2\gamma_{12}-1 $ the energy has two local minima, given by $\mathcal{S}_H$ and $\mathcal{S}_S^{2,2}$. Furthermore, linear stability analysis (Equation \eqref{eq:turingcondition}) suggests that when $\alpha$ tends to zero, the homogeneous steady state is stable if $\gamma_{11} > \gamma_{12}-1  $.  This gives rise to the diagram of analytically-predicted steady states given by the red and black lines in Figure \ref{fig:BD_mutualself_avoidance}.
		
		\subsection{Mutual attraction ($\mathbf{\gamma_{12}<0}$) }
		
		{\noindent}In this section, we analyze the local minimizers of the energy (Equation \eqref{eq:energylocA}) for $\gamma_{12}<0$, $\gamma_{11} \in \mathbb{R}$ and $\gamma_{12}=\gamma_{21}$.  We observe that the energy in Equation \eqref{eq:energylocA} decreases as $ \lvert S_{1} \cap S_{2} \lvert  $ increases, whilst keeping everything else constant. Therefore if we keep $\lvert S_1\lvert$ and $\lvert S_2 \lvert $ unchanged, then $\lvert S_1\cap S_2\lvert $ is maximised when either $S_1 \subseteq S_2$ or $S_2 \subseteq S_1$, so that $\lvert S_1\cap S_2 \lvert =\min_i\lvert S_i \lvert$. Thus by repeating the same argument presented in Section \ref{sec:analytic2} for $ \gamma_{12}<0 $ and $\gamma_{11}=0$, we see that $E[u_1 ,u_2 ] \rightarrow -\infty$ as $\min\{u_1^c,u_2^c\} \rightarrow \infty$. As we approach this limit, $u_1^c, u_2^c$ become arbitrarily large, so $u_1 $ and $ u_2 $ (Equation \eqref{eq:uiA}) become arbitrarily high, arbitrarily narrow functions with overlapping support. We will denote the limit of this solution by $ \mathcal{S}_A^{\infty} $.
		
		One can also show, using a very similar argument to Appendix \ref{appendix:A1} (details omitted), that the homogeneous steady state, $\mathcal{S}_H$, is the only other possible local minimiser of the energy that satisfies $u_i^c \geq 1$, for $i=1,2$, and this is only a local minimum when $\gamma_{12}> -\gamma_{11}-2$. However, linear stability analysis (Equation \eqref{eq:eigenvalues}) suggests that, in the limit as $\alpha$ tends to zero, the homogeneous steady state is linearly stable only if $\gamma_{12}  > -\gamma_{11}-1$. Therefore, any time  $\mathcal{S}_H$ is linearly stable, it is also a local energy minimiser within the set of functions given by Equation (\ref{eq:uiA}).
		These results give rise to the diagram of analytically-predicted steady states given by the red and black lines in Figures \ref{fig:BD_mutual_attraction_self_avoidance}-\ref{fig:BD_mutualself_attraction}.
		
		\subsection{Self attraction ($\mathbf{\gamma_{11}<0}$) and mutual avoidance ($\mathbf{\gamma_{12}>0}$) }
		
{\noindent}By following the same argument of Appendix \ref{appendix:A1}, to search for the local minimizers of the energy in Equation \eqref{eq:energylocA1}, we define
\begin{linenomath*}\begin{equation} 
	\mathcal{E}(u_1^c,u_2^c)=
	\begin{cases}
		\text{ln}(u_{1}^c)+ \text{ln}(u_{2}^c)  +\frac{1}{2}\gamma_{11} (u_{1}^c+u_{2}^c), & \text{ if }  \lvert S_1\rvert  + \lvert S_2\rvert  \leq 1, \\ \\
		\text{ln}(u_{1}^c)+ \text{ln}(u_{2}^c)  +\frac{1}{2}\gamma_{11} (u_{1}^c+u_{2}^c) \\\quad + \gamma_{12} u_{1}^c u_{2}^c \lvert S_{1} \cap S_{2}\rvert, & \text{ if } \lvert S_1\rvert  + \lvert S_2\rvert  > 1.
	\end{cases}
	\label{eq:mathcalEA1}
\end{equation}\end{linenomath*}	
We analyse $\mathcal{E}(u_1^c,u_2^c)$ (Equation \eqref{eq:mathcalEA1}) under the constraint
\begin{linenomath*}\begin{align}
		\label{ssssicA1}
		u_i^c\geq 1, \text{ for } i=1,2,
\end{align}\end{linenomath*}
first when $ \lvert S_1\rvert  + \lvert S_2\rvert  \leq 1 $ and then when $ \lvert S_1\rvert  + \lvert S_2\rvert  > 1 $. Recall that the condition in Equation \eqref{ssssicA1} is obtained by Equation \eqref{uicA}, using $ p_i=1 $ and $ \lvert S_i\rvert  \leq L =1 $.
		
When $ \lvert S_1\rvert  + \lvert S_2\rvert  \leq 1 $, $\mathcal{E}(u_1^c,u_2^c)\rightarrow -\infty$ as either $u_1^c \rightarrow \infty$ or $u_2^c\rightarrow \infty$. As we approach this limit, $u_1^c, u_2^c$ become arbitrarily large, so the functions $u_1(x) $ and $ u_2(x) $ (Equation \eqref{eq:uiA}) become arbitrarily high, arbitrarily narrow functions with $\lvert S_1 \cap S_2 \rvert= \emptyset$. We denote the limit of this solution by $ \mathcal{S}_S^{\infty,\infty} $, in which the subscript $ S $ stands for aggregation and the $\infty,\infty$ superscript denotes that both $u_1(x) $ and $ u_2(x) $ become unbounded and separated as $u_1^c,u_1^c \rightarrow \infty$.
		
As discussed in Appendix \ref{appendix:A1}, $ \lvert S_1\rvert  + \lvert S_2\rvert  \leq 1 $ is equivalent to the following condition
\begin{linenomath*}\begin{equation}\label{eq:constrA1}
	\frac{1}{u_1^c}+\frac{1}{u_2^c}\leq 1.
\end{equation}\end{linenomath*}
Thus, by analysing the partial derivatives of $\mathcal{E}(u_1^c,u_2^c)$ in the region of the $(u_1^c,u_2^c)$-plane defined by Equation \eqref{eq:constrA1}, one can check that there are no local minima in the interior of this region. Analysis of the partial derivative of $\mathcal{E}(u_1^c,u_2^c)$ on the boundary $1/{u_1^c}+1/{u_2^c}=1$ shows that $ \mathcal{E}(u_1^c,u_2^c)$ has a unique minimum point, given by
\begin{linenomath*}\begin{equation}\label{eq:minimumA1}
	\mathcal{M}_S=(u_{1S}^c, u_{2S}^c):=\left(2,2\right).
\end{equation}\end{linenomath*}
This is also a local minimum of the region defined by Equation \eqref{eq:constrA1} when $ \gamma_{11}>-1 $. This can be shown by performing a Taylor expansion of $ \mathcal{E}(u_1^c,u_2^c) $ about the point $\mathcal{M}_S$, to give
\begin{linenomath*}\begin{align}\nonumber
		\mathcal{E}(2+\epsilon,2+\delta)& \approx    \mathcal{E}(2,2)+ \partial_{u_1^c} \mathcal{E}(2,2) \epsilon + \partial_{u_2^c} \mathcal{E}(2,2)\delta\\\nonumber
		&= \mathcal{E}(2,2)+ \frac{1}{2}(1+\gamma_{1 1})(\epsilon + \delta) \\\nonumber
		& \geq \mathcal{E}(2,2),
\end{align}\end{linenomath*}
where the inequality uses $ \gamma_{1 1}>-1 $, $\epsilon+\delta\geq0$. We recall that $\epsilon+\delta \geq 0$ ensures that we remain in the $\lvert S_1\lvert +\lvert S_2\lvert \leq 1$ region (Figure \ref{fig:Diagram}).
		
Since the point $ \mathcal{M}_S $ lies on the boundary curve $ \lvert S_1\lvert  + \lvert S_2\rvert = 1 $, we have so far only established that when $ \gamma_{1 1}>-1 $, $\mathcal{M}_S$ is a minimum of $ \mathcal{E}(u_1^c,u_2^c)$ (Equation \eqref{eq:mathcalEA1}) in the region where $ \lvert S_1\rvert  + \lvert S_2\rvert \leq 1 $. We also need to show $\mathcal{M}_S$ is a minimum in the region where $ \lvert S_1\rvert  + \lvert S_2\rvert  > 1 $. By Equation \eqref{ssssicA}, the condition $ \lvert S_1\rvert  + \lvert S_2\rvert  > 1 $ is equivalent to $ {1}/{u_1^c}+{1}/{u_2^c}> L $.  Therefore we have the following constraints
\begin{linenomath*}\begin{align}
		\frac{1}{u_1^c}+\frac{1}{u_2^c}&> 1, \nonumber \\
				u_i^c&\geq 1, \text{ for } i=1,2. 
		\label{eq:constr2A1}
\end{align}\end{linenomath*}
As already shown in Appendix \ref{appendix:A1} (see Equation \eqref{eq:mathcalEsA}), when $\lvert S_1\rvert  + \lvert S_2\rvert  > 1$ the function $\mathcal{E}({u}_{1}^c,{u}_{2}^c) $ (Equation \eqref{eq:mathcalEA}) can be rewritten as
\begin{linenomath*}\begin{align} 
		\mathcal{E}(u_1^c,u_2^c) & =\text{ln}(u_1^c)+\text{ln}(u_2^c) + \frac{1}{2}\gamma_{11}  (u_1^c+u_1^c)+ \gamma_{12}  u_1^c u_2^c \left(\frac{1}{u_1} + \frac{1}{u_2} -1\right).
		\label{eq:mathcalEsA1}
\end{align}\end{linenomath*}
To show that $\mathcal{M}_S$ (Equation \eqref{eq:minimumA1}) is a minimum on the region of the domain given by Equation \eqref{eq:constr2A1}, we perform a Taylor expansion of $\mathcal{E}(u_1^c,u_2^c)$ (Equation \eqref{eq:mathcalEsA1}) around $\mathcal{M}_S$ within this region. Since the slope of the tangent line to the curve $1/u_1^c+1/u_2^c  = 1$ at the point  $ \mathcal{M}_S $  is $ -1 $, we choose two arbitrary constants, $\epsilon$ and $\delta$, such that $ \epsilon+ \delta\leq 0$.  The Taylor expansion is then
\begin{linenomath*}\begin{align}\nonumber
		\mathcal{E}(2+\epsilon,2+\delta) & \approx    \mathcal{E}(2,2)+ \partial_{u_1^c} \mathcal{E}(2,2) \epsilon + \partial_{u_2^c}\mathcal{E}(2,2)\delta\\\nonumber
		&=\mathcal{E}(2,2)+\frac{1}{2}(1+ \gamma_{11}-2\gamma_{12})(\epsilon +\delta)\\ 
		& \geq \mathcal{E}(2,2),
\end{align}\end{linenomath*}
if $ \gamma_{11} < 2\gamma_{12}-1 $. Therefore, $ \mathcal{M}_S $ (Equation \eqref{eq:minimumA1}) is a local minimum of $\mathcal{E}(u_1^c,u_2^c)$ (Equation \eqref{eq:mathcalEA1}) when $ -1<\gamma_{11} < 2\gamma_{12}-1 $. We recall that if $(u_1^c,u_2^c)=\mathcal{M}_S$ then the functions $u_i(x)$ (Equation \eqref{eq:uiA}) that locally minimize the energy $E[u_1 ,u_2 ]$ (Equation \eqref{eq:energylocA1}) correspond to the class of functions $ \mathcal{S}_{S}^{2,2} $ defined in Equation \eqref{eq:St2}.

Next we look for other local minima within the region of the domain given by Equation \eqref{eq:constr2A1}.  A direct calculation using partial derivatives shows that there are no local minima of $ \mathcal{E}(u_1^c,u_2^c)$ in the interior of this region.  We now verify whether local minima occur on the boundaries. On the part of the boundary given by $ u_i^c =1 $, for $ i=1,2 $, there is a local minimum at 
\begin{linenomath*}\begin{equation}\label{eq:minimum2A1}
		\mathcal{M}_H=(u_{1H}^c, u_{2H}^c):=\left(1,1\right).
\end{equation}\end{linenomath*}
This is also a local minimum of the region defined by Equation \eqref{eq:constr2A1} when $ \gamma_{11}>-2 $.  This can be shown by performing a Taylor expansion of $ \mathcal{E}(u_1^c,u_2^c) $ (Equation \eqref{eq:mathcalEsA1}) about the point $\mathcal{M}_H$, to give
\begin{linenomath*}\begin{align}\nonumber
		\mathcal{E}(1+\epsilon,1+\delta)& \approx    \mathcal{E}(1,1)+ \partial_{u_1^c} \mathcal{E}(1,1) \epsilon + \partial_{u_2^c} \mathcal{E}(1,1)\delta\\\nonumber
		&= \mathcal{E}(1,1)+ \left(1+\frac{1}{2}\gamma_{11}\right) (\epsilon +\delta) \\ \nonumber
		& \geq \mathcal{E}(1,1),
\end{align}\end{linenomath*}
where the inequality uses $ \gamma_{1 1}>-2 $, $\epsilon\geq0$ and $\delta\geq0$. Note that $\epsilon$ and $\delta$ are chosen to be non-negative so that we remain in the $ u_i^c \geq 1 $ region (Figure \ref{fig:Diagram}). We recall that if $(u_1^c,u_2^c)=\mathcal{M}_H$ then the functions $u_i(x)$ (Equation \eqref{eq:uiA}) that locally minimize the energy $E[u_1 ,u_2 ]$ (Equation \eqref{eq:energylocA1}) correspond to $\mathcal{S}_H$, the homogeneous steady state.
			
Notice also that on the boundary $u_{1}^c=1$, $\mathcal{E}(u_1^c,u_2^c)$  (Equation \eqref{eq:mathcalEsA1}) decreases as $u_2^c\rightarrow \infty$ and, analogously, on the boundary $u_{2}^c=1$,  $\mathcal{E}(u_1^c,u_2^c)$  (Equation \eqref{eq:mathcalEsA1}) decreases as $u_1^c\rightarrow \infty$. Therefore, by keeping $u_{i}^c=1$ fixed, for $i=1,2$, $\mathcal{E}(u_1^c,u_2^c)\rightarrow -\infty$  as $u_{j}^c \rightarrow \infty$, for $ j\neq i $. As we approach this limit, the function $u_j(x) $ (Equation \eqref{eq:uiA}) becomes an arbitrarily high function with an arbitrarily narrow support, while $u_i(x) $ (Equation \eqref{eq:uiA}), for $i\neq j$, remains at finite height. We denote the limit of these solutions by $\mathcal{S}_S^{1,\infty}$.

In conclusion: 
\begin{itemize}
	\item If $\gamma_{11}> 2\gamma_{12}-1$, then the energy $E(u_1^c,u_2^c)$ (Equation \eqref{eq:uiA}) has the following local minima: $\mathcal{S}_H$, $\mathcal{S}_S^{\infty,\infty}$ and $\mathcal{S}_S^{1,\infty}$.
	\item If $-1<\gamma_{11}<2\gamma_{12}-1$, the energy $E(u_1^c,u_2^c)$ (Equation \eqref{eq:uiA}) has the following local minima: $\mathcal{S}_H$, $\mathcal{S}_S^{\infty,\infty}$, $\mathcal{S}_S^{1,\infty}$ and $\mathcal{S}_S^{2,2}$.
	\item If $-2<\gamma_{11}<-1$, the energy $E(u_1^c,u_2^c)$ (Equation \eqref{eq:uiA}) has the following local minima: $\mathcal{S}_H$, $\mathcal{S}_S^{\infty,\infty}$ and $\mathcal{S}_S^{1,\infty}$.
	\item If $\gamma_{11}<-2$, the energy $E(u_1^c,u_2^c)$ (Equation \eqref{eq:uiA}) has the following local minima: $\mathcal{S}_S^{\infty,\infty}$ and $\mathcal{S}_S^{1,\infty}$.
\end{itemize}
Furthermore, linear stability analysis (Equation \eqref{eq:turingcondition}) suggests that when $\alpha$ tends to zero, the homogeneous steady state is stable if $\gamma_{11}> \gamma_{12}-1$.  This gives rise to the diagram of analytically-predicted steady states given by the red and black lines in Figure \ref{fig:BD_mutual_avoidance_self_attraction}.

\section{Details of calculations from Section \ref{sec:N2}}
\label{appendix:C}

{\noindent}Here, we analyze the solutions to the system $\left\{\det (A_1^{(2)})=0, \det (A_2^{(2)})=0\right\}$, where $A_1^{(2)}$ and $A_2^{(2)}$ are given in Equation \eqref{A1N2} and Equation \eqref{A2N2}, respectively. We write the system $\left\{\det (A_1^{(2)})=0, \det (A_2^{(2)})=0\right\}$ in full as
\begin{linenomath*}\begin{align}
		\label{eq:Asystem1}
			0&=(D_1+\gamma_{11}u_1 )(D_2+\gamma_{22}u_2 )-\gamma_{12}\gamma_{21}u_1 u_2 , \\
		\label{eq:Asystem2} \nonumber
			0&=\gamma_{12} u_1 \left(\gamma _{11} \left(D_2+\gamma _{22} u_2\right)-\gamma_{12}\gamma_{21} u_2\right)\\
			& \quad-\left(\gamma _{22} \left(D_1+\gamma _{11} u_1\right)-\gamma_{12}\gamma_{21} u_1\right) \left(D_1+\gamma _{11} u_1\right),
\end{align}\end{linenomath*}
By subtracting Equation \eqref{eq:Asystem2} from Equation \eqref{eq:Asystem1}, we obtain the following linear equation in $u_2$
\begin{linenomath*}\begin{align}\nonumber
		\gamma _{11} D_2 u_1-\gamma _{11} \gamma _{12} D_2 u_1-\gamma _{12} \gamma _{21} D_1 u_1+\gamma _{22} \left(D_1+\gamma _{11} u_1\right){}^2\\\label{eq:u2}+D_1 D_2-\gamma _{11} \gamma _{12} \gamma _{21} u_1^2 +
		u_2 (\gamma _{22} D_1+\left(\gamma _{12}-1\right) \left(\gamma _{12} \gamma _{21}-\gamma _{11} \gamma _{22}\right) u_1)=0
\end{align}\end{linenomath*}
By using Equation \eqref{eq:u2} to find $u_2$ in terms of $u_1$ and then substituting this into Equation \eqref{eq:Asystem1}, we obtain the following cubic equation in $u_1$
\begin{linenomath*}\begin{align}\nonumber
		\gamma _{22}^2 D_1^3 - D_1 u_1 \left(\gamma _{21} \gamma _{12}^2 D_2+\gamma _{22} \left(2 \gamma _{12} \gamma _{21}-3 \gamma _{11} \gamma _{22}\right) D_1\right) &\\\nonumber+ u_2^2 D_1 \left(\gamma _{12} \gamma _{21}-3 \gamma _{11} \gamma _{22}\right) \left(\gamma _{12} \gamma _{21}-\gamma _{11} \gamma _{22}\right) & \\\label{eq:u1} +u_1^3\gamma _{11} \left(\gamma _{12} \gamma _{21}-\gamma _{11} \gamma _{22}\right){}^2 &=0.
\end{align}\end{linenomath*}
Since Equation \eqref{eq:u1} has at most three roots,  System \eqref{eq:Asystem1}-\eqref{eq:Asystem2} has at most three solutions.

\end{appendices}


\bibliography{bib_O3}


\end{document}